\documentclass{article}
\usepackage{graphicx} % Required for inserting images
\usepackage{amsmath, amsfonts, amsthm, amssymb,mathrsfs,mathtools}
\usepackage{cases, color, bm, bbold}
\usepackage{xcolor, cite}

\usepackage{enumerate}
\usepackage[text={6.2in, 7.9in},centering]{geometry}
\definecolor{ForestGreen}{RGB}{20,140,80}
\RequirePackage[colorlinks,citecolor=ForestGreen,urlcolor=blue,linkcolor=blue]{hyperref}

\usepackage[english]{babel}

\newcommand{\R}{\mathbb{R}}

\newcommand\bL{\mathbb{L}}
\newcommand\bR{\mathbb{R}}
\newcommand\bH{\mathbb{H}}

\newcommand\bZ{\mathbb{Z}}

\newcommand\bD{\mathbb{D}}
\newcommand\bS{\mathbb{S}}

\newcommand\bE{\mathbb{E}}

\newcommand\bP{\mathbb{P}}

\newcommand\ba{\mathbb{a}}
\newcommand\bb{\mathbb{b}}
\newcommand\bc{\mathbb{c}}

\newcommand\cC{\mathcal{C}}

\newcommand\cF{\mathcal{F}}
\newcommand\cG{\mathcal{G}}
\newcommand\cH{\mathcal{H}}

\newcommand\cL{\mathcal{L}}

\let \R \bR

\newcommand\fA{\mathfrak{A}}
\newcommand\fB{\mathfrak{B}}
\newcommand\fC{\mathfrak{C}}
\newcommand\fD{\mathfrak{D}}

\renewcommand{\P}{\mathbb{P}}
\newcommand{\E}{\mathbb{E}}

\newcommand{\Ctem}{C_{\mathrm{tem}}}
\newcommand{\1}{\textbf{1}}
\newcommand{\ud}{\mathrm{d}}

\newcommand\cbrk{\text{$]$\kern-.15em$]$}}
\newcommand\opar{\text{\,\raise.2ex\hbox{${\scriptstyle
|}$}\kern-.34em$($}}
\newcommand\cpar{\text{$)$\kern-.34em\raise.2ex\hbox{${\scriptstyle |}$}}\,}
\newcommand\supp{\mathrm{supp}}

\newcommand\ep{\varepsilon}

\renewcommand{\epsilon}{\ep}

\newtheorem{stat}{Statement}[section]
\newtheorem{proposition}[stat]{Proposition}

\newtheorem{corollary}[stat]{Corollary}
\newtheorem{theorem}[stat]{Theorem}
\newtheorem{lemma}[stat]{Lemma}

\newtheorem{innercustomthm}{Theorem}
\newenvironment{customthm}[1]
  {\renewcommand\theinnercustomthm{#1}\innercustomthm}
  {\endinnercustomthm}

\theoremstyle{definition} 
\newtheorem{definition}[stat]{Definition}
\newtheorem{assumption}[stat]{Assumption}

\newtheorem{remark}[stat]{Remark}

\newtheorem{example}[stat]{Example}

\numberwithin{equation}{section}

\title{Instantaneous shrinking of supports for stochastic PDEs\footnote{JY was supported by a KIAS Individual Grant (HP090401) and by the National Science Foundation under Grant No. DMS-2424139 while in residence at the Simons Laufer Mathematical Sciences Institute in Berkeley, California, during the Fall 2025 semester. KK was supported in part by the National Research Foundation of Korea (RS-2026-25479024).}}

\author{Beom-Seok Han \footnote{Sungshin Women's University, Seoul, South Korea E-mail: b\_han@sungshin.ac.kr}
\and Kunwoo Kim  \footnote{Pohang University of Science and Technology (POSTECH), Pohang, South Korea E-mail: kunwoo@postech.ac.kr}
\and Jaeyun Yi  \footnote{Korea Institute for Advanced Study (KIAS), Seoul, South Korea E-mail: jaeyun@kias.re.kr}
}

\date{\today}

\begin{document}

\maketitle
\begin{abstract}
We study instantaneous shrinking of supports for nonnegative solutions of the stochastic partial differential equation
\[
\partial_t u=a(t,x)\,\partial_x^2 u + b(t,x)\,\partial_x u + c(t,x)\,u +\sigma(u)\,\xi(t,x), \qquad (t,x)\in(0,\infty)\times\mathbb R,
\]
where $\xi$ is space-time white noise, the coefficients $a$, $b$, $c$ may be random, and the noise coefficient $\sigma$ vanishes at the origin and is sublinear there. The model case is $\sigma(u)=u^\gamma$ with $\gamma\in(0,1)$. We show that, under a uniqueness-in-law assumption, if the initial datum has a sufficiently light spatial tail, then every nonnegative solution has compact support at every positive time, even though the initial support is not compact. The initial datum may also be a measure, such as
a Dirac mass. When $\gamma\in(0,1/2]$, finite initial mass suffices; this covers the super-Brownian case $\gamma=1/2$. When $\gamma\in(1/2,1)$, we identify a polynomial moment condition on the initial state whose order diverges as $\gamma\uparrow1$, quantifying the trade-off between the strength of the noise near zero and the decay of the initial data required for instantaneous shrinking. As a step of independent interest, we establish weak existence of solutions started from measure-valued initial data for non-Lipschitz $\sigma$ and random operators. Our results provide a stochastic counterpart of the instantaneous shrinking phenomenon of Evans
and Knerr for deterministic parabolic equations with strong absorption, in which the role of the absorption term is played entirely by the noise.
\vspace{1cm}

    
\end{abstract}

{\hypersetup{linkcolor=black}
\tableofcontents}

%%%%%%%%%%%%%%%%%%%%%%%%%%%%%%%%%%%%%%%%%%%%%%%%%%%%%%%%%%%%%%%%%%%%%%%%%%%%%%%%%%%%%%%%%%%%%%%%

\section{Introduction}
 We study whether compact support can be created instantaneously by sublinear multiplicative noise. More precisely, we consider nonnegative solutions to the one-dimensional stochastic partial differential equation
\begin{equation}\label{eq:SPDE}
\begin{aligned}
\partial_t u(t, x)  = L u(t, x) + \sigma(u(t,x)) \xi(t, x), \quad (t,x)\in \R_+\times\R;\quad 
u(0, \cdot) = u_0,
\end{aligned}
\end{equation}
where
\begin{equation}
\label{second order operator}
L u(t,x) = a(t,x)\partial_{x}^{2}u(t,x) + b(t,x) \partial_{x}u(t,x) + c(t,x)u(t,x)
\end{equation} 
is a possibly random parabolic second-order differential operator, $\xi$ is space-time white noise on $(0,\infty)\times\bR$, and the noise coefficient $\sigma:\bR\to[0,\infty)$ vanishes on $(-\infty,0]$ and is allowed to be non-Lipschitz (sublinear) at the origin. The model example to keep in mind is
\[
    \sigma(u)=u^\gamma{\bf 1}_{\{u\ge0\}},
    \qquad
    \gamma\in(0,1).
\]
Our main question is the following:

\begin{quote}
\emph{if the initial state is strictly positive everywhere on $\bR$,
can the solution nevertheless have compact support at every positive
time?}
\end{quote}   

This question should be contrasted with the classical \emph{compact support property}. For stochastic heat equations, the compact support property (CSP) is usually formulated as a \emph{preservation} result: if the initial datum has compact support, then the solution keeps compact support for all times. Such results have a long history. For 
\begin{equation}\label{eq:model_SPDE}
\partial_t u(t,x)=\partial_x^2 u(t,x) + u(t,x)^\gamma \xi(t,x), \qquad  (t,x)\in(0,\infty)\times\bR,
\end{equation}
the case $\gamma=1/2$ is connected with the density of one-dimensional super-Brownian motion, and CSP follows from the duality argument (see Iscoe \cite{iscoe1988supports} and also Perkins \cite{Per02}). Shiga \cite{shiga1994two} treated the case $\gamma \in (0, 1/2)$ by following Iscoe's idea. The range $\gamma\in(1/2,1)$ is more delicate, since the coefficient $u^\gamma$ is smaller than $u^{1/2}$ near zero. Mueller and Perkins \cite{mueller1992compact} proved CSP in this regime by using a historical process representation. On the other hand, Mueller \cite{mueller1991support} showed that CSP may fail when the noise is too weak near zero, i.e., $\gamma\ge1$.

Krylov \cite{krylov1997result} later gave an analytic proof of CSP for one-dimensional SPDEs with space-time white noise and a general random parabolic second-order differential operator. His method is based on weak solutions and $L_p$-theory for SPDEs. More recently, Han--Kim--Yi \cite{han2023compact} proved CSP for SPDEs with spatially colored noise, extending the one-dimensional white-noise theory to a broader class of noises, while Hughes \cite{Hug25} established CSP for the stochastic heat equation driven by white stable noise. Related phase-transition phenomena depending
on the strength of the noise near zero have also been studied in \cite{han2024support}.

All the results mentioned above are support-preservation results. This paper establishes a different phenomenon, in which compact support is \emph{created instantaneously from noncompact initial data}. This is far from obvious. The noise in \eqref{eq:SPDE} enters as a martingale term and, up to lower-order terms, the total mass of a nonnegative solution is a supermartingale. In other words, there is no absorption in the equation. Whatever kills the tail must therefore be the fluctuations themselves. In regions where the solution is small, the sublinear coefficient makes the noise dominant, since $\sigma(u)\gg u$ as $u\downarrow0$. The fluctuations then drive the local mass to zero, where it is trapped because $\sigma(0)=0$. The question is whether this mechanism can eliminate an unbounded initial tail before any prescribed positive time.

For deterministic parabolic equations, the corresponding phenomenon is produced by an explicit absorption term. That is, for
$\partial_t u=\Delta u-u^\gamma$ with $\gamma\in(0,1)$, it is known as \emph{instantaneous shrinking of supports} and has been studied extensively (see Evans and Knerr
\cite{evans1979instantaneous} and also \cite{GK90,galaktionov1994extinction,ILS17}). For SPDEs, where the role of the absorption is played entirely by a mean-zero noise, we are aware of only one case in which instantaneous shrinking has
been established: the super-Brownian one ($\gamma=1/2$ in
\eqref{eq:model_SPDE}). There, finite initial mass---for instance,
$u_0\in L_1(\bR)$, or even a finite initial measure---forces compact
support at all positive times (see \cite[Corollary III.1.4]{Per02}). The proof, however, relies on the duality and branching
structure special to $\gamma=1/2$, and gives no information about other coefficients $\sigma$, let alone random operators $L$. We establish instantaneous shrinking for a broad class of SPDEs of the form \eqref{eq:SPDE}. To illustrate our results in the simplest unconditional form, we
state the following special case (see
Corollary \ref{cor:model_case} for the precise statement).

\begin{customthm}{A}[Informal]\label{thm:informal}
Let $u$ be a nonnegative solution of the model equation
\eqref{eq:model_SPDE} with $\gamma\in[1/2,1)$. If the nonnegative initial function (or measure) $u_0$ satisfies
\[
    \int_{\bR}(1+x^2)^{q/2}\,u_0(x)\,\ud x<\infty
    \quad\text{for some }q> \frac{2\gamma-1}{2-2\gamma}, \text{ or for }q=0\text{ when }\gamma=1/2,
\]
then, almost surely, $u(t)$ has compact support at every $t>0$.
\end{customthm}
Theorem~\ref{thm:informal} is a special case of our two main results, which treat
general random operators $L$, a general class of sublinear noise
coefficients $\sigma$, and both function-valued and measure-valued
initial data. Let us now describe them in turn.

Our first main theorem (Theorem \ref{thm:CSP_regular_initial}) treats regular initial functions. Roughly speaking, if $u_0$ is nonnegative, has enough Sobolev regularity, and its spatial tail is sufficiently light, then \emph{every} nonnegative solution of \eqref{eq:SPDE} has compact support at every positive time, with probability one. The tail condition is related to the strength of the noise near zero, which we
quantify by the exponent $\lambda\in[1,2)$ in Assumption
\ref{asp:sigma}. For the model coefficient $\sigma(u)=u^\gamma$ (see \eqref{eq:model_SPDE}), we can take $\lambda=1$ when $\gamma\in(0,1/2]$ and $\lambda=2\gamma$ when $\gamma\in(1/2,1)$. In the strong-noise regime $\gamma \in (0, 1/2]$ (i.e. $\lambda=1$), which includes the super-Brownian case, finite mass 
\[
\int_{\bR} u_0(x)\, \ud x <\infty
\]
essentially suffices. In the weaker-noise regime $\gamma \in (1/2, 1)$ (i.e. $\lambda\in(1,2)$), we require a polynomial moment
\begin{equation}
\label{eq:intro_moment}
\int_{\bR}(1+x^2)^{q/2}\,u_0(x)\,\ud x<\infty
\qquad\text{for some}\qquad q>\frac{\lambda-1}{2-\lambda}.
\end{equation}
This quantifies the trade-off between the strength of the noise and the decay of the initial data required for instantaneous shrinking: the weaker the noise near 0, the lighter the initial tail must be. The threshold degenerates as $\gamma\uparrow1$, which is consistent with the failure of the compact support property at $\gamma=1$, even for compactly supported initial data.

Our second main result (Theorem \ref{thm:CSP_measure_initial}) treats measure-valued initial data. We allow the initial data to be nonnegative random Radon measures $\mu_0$ with 
\[ 
\mu_0 \in L_2\bigl(\Omega;H_2^{-1/2-\kappa}(\bR)\bigr) \qquad \text{for every }\kappa\in(0,1/2).
\]
This condition on $\mu_0$ is natural since all finite Radon measures on $\bR$ belong to $H_2^{-1/2-\kappa}(\bR)$ for every
$\kappa>0$. In particular, the initial data may be a Dirac mass, or an infinite superposition of atoms satisfying suitable moment conditions.  

In this generality, existence of a (weak) solution from measure-valued initial data is not part of the standard theory, because $\sigma$ is non-Lipschitz and the operator $L$ may have random coefficients. To the best of our knowledge, the resulting existence result (Theorem \ref{thm:CSP_measure_initial}(a)) is new even in its own right, and may be of independent interest. Then, under the moment conditions corresponding to \eqref{eq:intro_moment} (finite mass when $\lambda=1$), we prove that every nonnegative solution starting from $\mu_0$ has compact support at every positive time.

Both theorems are proved under a uniqueness-in-law assumption for
\eqref{eq:SPDE}, which is a well-known open problem for
non-Lipschitz coefficients in general. We emphasize, however, that
this assumption is known to hold in the model case
\eqref{eq:model_SPDE} when $\gamma\in[1/2,1)$ (see e.g. \cite{Myt98, Per02}).  Consequently, our results are unconditional in these cases, which yields Theorem \ref{thm:informal} above (see Remark \ref{rmk:uniqueness_discussion} and Corollary \ref{cor:model_case}).

Let us describe the main ideas of the proofs. Consider first the
setting of Theorem \ref{thm:CSP_regular_initial}, i.e., a regular
initial function $u_0$.  The strategy is to split $u_0$ into a compactly supported part and a tail part, i.e., 
\[
u_0=u_0\,\rho(\cdot/n)+u_0\,(1-\rho(\cdot/n)),
\]
where $\rho$ is a smooth cutoff. We then construct two auxiliary nonnegative processes $v_n$ and $w_n$ driven by two independent white noises (see Section \ref{subsec:splitting_construction}). The process $v_n$ starts from the compactly supported part and solves \eqref{eq:SPDE}, whereas $w_n$ starts from the tail and solves an equation with noise coefficient $\bigl(\sigma^2(v_n+w_n)-\sigma^2(v_n)\bigr)^{1/2}$. This coupling is designed precisely so that $y_n:=v_n+w_n$ is again a nonnegative solution of \eqref{eq:SPDE}, driven by a suitable space-time white noise. This recombination step is what will allow us to pass from the split system to an arbitrary solution, since uniqueness in law gives $\cL(u)=\cL(y_n)$.

The compactly supported part $v_n$ has compact support at all times by the compact-initial-data results of \cite{krylov1997result,han2023compact}. The heart of the proof is therefore to show that the tail part $w_n$ dies out, i.e.,  for every fixed $t_*>0$,
\[
\lim_{n\to\infty} \bP\Bigl(\inf_{0\le t\le t_*} M_n(t)>0\Bigr)=0, \quad \text{where} \quad M_n(t):=\int_{\bR}w_n(t,x)\,\ud x
\]
(see Lemma \ref{lem:tail_extinction}). Up to an exponential discounting, $M_n$ is a nonnegative supermartingale satisfying $\ud M_n\le K M_n\,\ud t+\ud N_n$ for a continuous local martingale $N_n$. By a one-dimensional extinction criterion for such semimartingales (Lemma \ref{lem:extinction}), $M_n$ then hits zero before time $t_*$ with high probability, provided $\ud\langle N_n\rangle_t\ge M_n^\alpha\,\ud t$ for some $\alpha<2$.

Producing this lower bound is the main technical step. Assumption
\ref{asp:sigma} gives $\sigma^2(v_n+w_n)-\sigma^2(v_n)\ge c\,w_n^\lambda$ on events where $v_n$ and $w_n$ are bounded, and we control the corresponding events
uniformly in $n$ by supremum estimates based on Krylov's $L_p$-theory (Section
\ref{sec:splitting_estimates}). It remains to compare $\int w_n^\lambda\,\ud x$
with $M_n$. For $\lambda=1$ the two agree, and finite mass suffices. For
$\lambda\in(1,2)$ an interpolation against the weighted norm $W_q(w_n)$
(Corollary \ref{cor:L_epsilon_bound_w_n}) is needed, and this is where the
moment threshold $q>(\lambda-1)/(2-\lambda)$ enters. Since zero is absorbing
for $M_n$, the tail stays extinct once it dies, and $y_n=v_n$ afterwards.

For measure-valued initial data, a solution must first be constructed. 
We mollify $\sigma$ and $\mu_0$ and pass to the limit in the resulting approximations. Since no estimate on them is uniform in $n$ up to $t=0$, we establish tightness only in $C((0,T];C_{\rm tem})$, and the limit is identified through the martingale problem. 
This last step requires care, since the coefficient field must be carried through the Skorokhod representation, so that the approximations and the limit are adapted to different filtrations, and the limiting bracket is identified by a localization with stopping times (Section \ref{sec:construction-measure}).

Once a solution is available, an additional regularization step is needed. After constructing a solution from $\mu_0$, we prove that for every $t_0>0$,  $u(t_0,\cdot)$ belongs almost surely to a Bessel potential space $H_p^\beta(\bR)$ with $\beta>1/p$. This is obtained by applying Krylov's $L_p$-estimate to $\chi u$, where $\chi$ is a smooth time cutoff vanishing near the origin. The cutoff removes the singular initial state, at the cost of a drift term that is controlled by the weighted mass of the solution (Section \ref{sec:proof_thm_measure}). Once this positive-time regularity is available, we restart the equation at time $t_0$ and apply Theorem \ref{thm:CSP_regular_initial} to the shifted equation. Here, the assumed uniqueness in law for the shifted equation then transfers the compact support property to the actual shifted process.

The paper is organized as follows. Section \ref{sec:main_results} states the assumptions and the main results, and Section \ref{sec:prelim} collects the analytic preliminaries. Sections
\ref{sec:splitting_estimates} and \ref{sec:tail_extinction} carry out the splitting construction and the tail extinction argument, proving Theorem \ref{thm:CSP_regular_initial}. Sections \ref{sec:construction-measure} and \ref{sec:proof_thm_measure} construct solutions from measure-valued initial data and prove Theorem \ref{thm:CSP_measure_initial}.

\paragraph{Notation.}
Throughout the paper, $N=N(a_1,\ldots,a_k)$ denotes a finite positive
constant depending only on the parameters $a_1,\ldots,a_k$; its value may
change from line to line. We write $N_{a_1,\ldots,a_k}$ to shorten the notation when convenient. For $a,b\in\bR$, we write
\[
    a\wedge b:=\min\{a,b\},
    \qquad
    a\vee b:=\max\{a,b\}.
\]
For $q\ge0$ and a nonnegative function or measure $f$ on $\bR$, set
\begin{equation}\label{eq:weight_ft}
    \Phi_q(x):=(1+x^2)^{q/2},
    \qquad
    W_q(f):=\int_{\bR}\Phi_q(x)\,f(\ud x),
\end{equation}
with the convention that $f(\ud x)=f(x)\,\ud x$ when $f$ is a function. In
particular, $W_0(f)$ is the total mass of $f$. We also fix the exponential
weights
\begin{equation}\label{eq:zeta}
    \zeta(x):=\frac1{\cosh x},
    \qquad
    \zeta_m(x):=\zeta(x/m),
    \quad m\ge1,
\end{equation}
which satisfy $0<\zeta_m\le1$, $\zeta_m\uparrow1$ as $m\to\infty$,
$\zeta_m\in L_p(\bR)$ for every $p\in[1,\infty)$, and
\begin{equation}\label{eq:zeta_bounds}
    |\zeta_m'(x)|\le N\zeta_m(x),
    \qquad
    |\zeta_m''(x)|\le N\zeta_m(x),
\end{equation}
with $N$ independent of $m$.

For a random variable $X:\Omega\to\bR$ and $1<p<\infty$, define
\[
    \|X\|_{L_p(\Omega)}
    :=
    \bigl(\bE[|X|^p]\bigr)^{1/p},
\]
and, for a Banach space $E$, we denote by $L_p(\Omega;E)$ the space of
$E$-valued random variables with $\bE[\|X\|_E^p]<\infty$. 

We denote the support of a function $f=f(t,x)$ at time $t>0$ by $\supp(f(t))$.

%%%%%%%%%%%%%%%%%%%%%%%%%%%%%%%%%%%%%%%%%%%%%%%%%%%%%%%%%%%%%%%%%%%%%%%%%%%%%%%%%%%%%%%%%%%%%%%%

\section{Main results}
\label{sec:main_results}
Throughout the paper, we fix an arbitrary nonrandom constant $T>0$.
All filtered probability spaces
$(\Omega,\mathcal F,\{\mathcal F_t\}_{t\ge0},\bP)$ are assumed to
satisfy the usual conditions, and $\mathcal P$ denotes the predictable
$\sigma$-field associated with $\{\mathcal F_t\}$.

A \emph{space-time white noise} relative to $\{\mathcal F_t\}$ is a
mean-zero Gaussian family
$\{\xi(\varphi):\varphi\in C_c^\infty(\bR_+\times\bR)\}$
with covariance
\[
    \E[\xi(\varphi_1)\xi(\varphi_2)]
    =\int_0^\infty\!\!\int_{\bR}
    \varphi_1(t,x)\varphi_2(t,x)\,\ud t\,\ud x,
\]
such that $\xi(\varphi)$ is $\mathcal F_t$-measurable whenever
$\varphi$ vanishes on $(t,\infty)\times\bR$, and the noise on
$(s,\infty)\times\bR$ is independent of $\mathcal F_s$ for every
$s\ge0$. The covariance identity extends $\xi$ uniquely to
$L_2(\bR_+\times\bR)$, and we write
\[
    \xi(\varphi)
    :=\int_0^\infty\!\!\int_{\bR}\varphi(t,x)\,\xi(\ud t,\ud x),
\]
with stochastic integrals against predictable integrands understood
in the sense of \cite{walsh1986introduction,dalang1999extending}.

Equation \eqref{eq:SPDE} is considered on a filtered probability
space carrying a space-time white noise $\xi$ relative to
$\{\mathcal F_t\}$. The auxiliary constructions in
Sections \ref{sec:splitting_estimates} and \ref{sec:construction-measure} may take place on
different filtered probability spaces. We first provide assumptions on the noise coefficient $\sigma$.

\begin{assumption}
\label{asp:sigma}
The function $\sigma:\bR\to[0,\infty)$ is continuous and satisfies the following
conditions:
\begin{enumerate}[(i)]
\item $\sigma(u)=0$  for all $u\le0$.

\item
There exist constants $\bc_1>0$ and $\gamma\in(0,1)$ such that
\[
\sigma(u)\le \bc_1(u^\gamma+u),
\qquad u\ge0.
\]

\item
There exists $\lambda\in[1,2)$ such that, for every $K>0$, there exists
$\bc_2=\bc_2(K)>0$ satisfying
\begin{equation}
\label{eq:lambda_condition}
    \bc_2|u-v|^\lambda
    \le
    \sigma(u)^2-\sigma(v)^2
    \qquad
    \text{for all }0\le v\le u\le K.
\end{equation}
\end{enumerate}
\end{assumption}

\begin{remark}
Assumption \ref{asp:sigma}(iii) implies that
$u\mapsto\sigma(u)$ is nondecreasing on $[0,\infty)$. In addition, there is no continuous function $\sigma$ satisfying
\eqref{eq:lambda_condition} with $\lambda\in(0,1)$. Indeed, if such a function
existed, then for the uniform partition $x_i=iK/N$ where $i=0, \dots, N$, we would have
\[
    \sigma^2(K)-\sigma^2(0)
    =
    \sum_{i=1}^N\left(\sigma^2(x_i)-\sigma^2(x_{i-1})\right)
    \ge \bc_2 K^\lambda N^{1-\lambda}, 
\]
which diverges as $N\to\infty$ if $\lambda<1$.
\end{remark}

\begin{example}
A typical example of $\sigma$ is
\[
    \sigma(u)=u^\gamma\mathbf 1_{\{u\ge0\}},
    \qquad \gamma\in(0,1).
\]
If \(\gamma\in(1/2,1)\), then \eqref{eq:lambda_condition} holds with
\(\lambda=2\gamma\) since  $(u-v)^{2\gamma} \le u^{2\gamma}-v^{2\gamma}$ for  $0\le v\le u.$

\noindent If \(\gamma\in(0,1/2]\), then \eqref{eq:lambda_condition} holds with
\(\lambda=1\). Indeed, for \(\gamma=1/2\), this is immediate. For
\(\gamma\in(0,1/2)\), if \(0\le v\le u\le K\), then
\[
    u^{2\gamma}-v^{2\gamma}
    =
    \int_v^u 2\gamma z^{2\gamma-1}\,\ud z
    \ge
    2\gamma K^{2\gamma-1}(u-v).
\]
%The case \(v=0\) is also immediate, since for \(0\le u\le K\), $u^{2\gamma}\ge K^{2\gamma-1}u.$
The class of admissible coefficients is not limited to pure powers. For example,
\[
    \sigma(u)
    =
    u^\theta\bigl(\log(1+u)+1\bigr)\mathbf 1_{\{u\ge0\}},
    \qquad \theta\in(0,1),
\]
also satisfies Assumption \ref{asp:sigma}, with the same value of \(\lambda\)
as in the power case. Moreover,
\[
    \sigma(u)
    =
    \sqrt{u+u^{2\theta}}\,\mathbf 1_{\{u\ge0\}},
    \qquad \theta\in(0,1),
\]
satisfies \eqref{eq:lambda_condition} with \(\lambda=1\).
\end{example}

Below are  assumptions on the operator coefficients, i.e., $a$, $b$, and $c$ in \eqref{second order operator}. 
\begin{assumption}
\label{asp:coefficients}
The coefficient field $(a,b,c)$ is a random element of 
\begin{equation}\label{eq:def_coefficient_space}
    \cC : =C( [0,T]; C^2_{\mathrm{loc}}(\bR)) \times C( [0,T]; C^1_{\mathrm{loc}}(\bR)) \times C([0,T]; C_{\mathrm{loc}}(\bR)) ,
\end{equation} endowed with its natural Fr\'echet topology. Moreover, $(a,b,c)$ satisfies the following:

\begin{enumerate}[(i)]
\item
The coefficients $a,b,c$ are $\mathcal P\times\mathcal B(\bR)$-measurable.

\item
The functions $\partial_x^2a, \partial_xa, a, \partial_xb, b, c$ are continuous in $x$.

\item
There exists a finite nonrandom constant $\nu\ge1$ such that for every $(\omega,t,x)\in\Omega\times[0,T]\times\bR$,
% \begin{equation}
% \label{ellipticity}
%     a(t,x)\ge \nu^{-1}
%     \qquad
%     \text{for all }(\omega,t,x)\in\Omega\times[0,T]\times\bR,
% \end{equation}
\begin{equation}
\label{ellipticity}
    a(t,x)\ge \nu^{-1},
\end{equation}
and
% \[
%     \|a(t,\cdot)\|_{C^2(\bR)}
%     +
%     \|b(t,\cdot)\|_{C^1(\bR)}
%     +
%     \|c(t,\cdot)\|_{C(\bR)}
%     \le \nu \quad\text{for all $(\omega,t)\in\Omega\times[0,T]$},
% \]
\[
    \|a(t,\cdot)\|_{C^2(\bR)}
    +
    \|b(t,\cdot)\|_{C^1(\bR)}
    +
    \|c(t,\cdot)\|_{C(\bR)}
    \le \nu,
\]where $C^{k}(\bR)$ is the set of all $k$-times continuously differentiable functions on $\bR$ with the norm
\[
    \|f\|_{C^k(\bR)}
    :=
    \sum_{j=0}^k\sup_{x\in\bR}|D^jf(x)|.
\]
\end{enumerate}
\end{assumption}

\begin{remark}
The requirement in Assumption \ref{asp:coefficients} that the coefficient field $(a,b,c)$ be a $\cC$-valued random variable is used only in the proof of Theorem~\ref{thm:CSP_measure_initial}, specifically in the Skorokhod representation argument in Proposition~\ref{prop:existence_measure_initial}.
\end{remark}

We now state the compact support property (CSP) for regular initial functions. Below solutions are understood in the sense of Definition \ref{def:weak_soln} (see also Section \ref{sec:prelim} for the definitions of $C_{\rm tem}$ and $H_p^\alpha(\R)$). 

\begin{theorem}[CSP for regular initial functions]
\label{thm:CSP_regular_initial}
Suppose Assumptions \ref{asp:sigma} and \ref{asp:coefficients} hold,
and let $\gamma\in(0,1)$ and $\lambda\in[1,2)$ be the constants in
Assumption \ref{asp:sigma}.
Let $u_0\ge0$ be an $\mathcal F_0$-measurable random initial function
satisfying
\begin{equation}\label{eq:asp_initial_regularity}
     u_0\in L_p(\Omega;H_p^{\eta-2/p}(\bR))
\end{equation}
for some $\eta>0$ and some $p>6$ such that $p>\frac3\eta$.
Assume in addition one of the following conditions.
\begin{enumerate}
\item[\rm (i)]
If $\lambda=1$, assume either
\[
    \E\int_{\bR}u_0(x)\,\ud x<\infty
    \qquad\text{and}\qquad
    p\gamma\ge1,
\]
or that there exists $q>0$ such that
\[
    \E\int_{\bR}(1+x^2)^{q/2}u_0(x)\,\ud x<\infty
    \qquad\text{and}\qquad
    p\gamma >\frac1{q+1}.
\]
\item[\rm (ii)]
If $\lambda\in(1,2)$, assume that there exists
\[
    q>\frac{\lambda-1}{2-\lambda}
\]
such that
\[
    \E\int_{\bR}(1+x^2)^{q/2}u_0(x)\,\ud x<\infty
    \qquad\text{and}\qquad
    p\gamma>\frac1{(q+1)}.
\]
\end{enumerate}
Assume finally that uniqueness in law holds for nonnegative
$C([0,T];C_{\rm tem})$-valued weak solutions of \eqref{eq:SPDE},
with the joint law $\mathcal L(u_0,a,b,c)$ of the initial function
and the coefficient field fixed.
Then every nonnegative solution $u\in C([0,T];C_{\rm tem})$ of
\eqref{eq:SPDE} with initial function $u_0$ has compact support at
every $t\in(0,T]$, almost surely.
\end{theorem}

\begin{remark}
   The proof of Theorem~\ref{thm:CSP_regular_initial} yields the following slightly stronger conclusion: almost surely, for every
$m\ge1$ there exists $R=R(m,\omega)>0$ such that
\[
    \supp\big(u(t)\big)\subseteq[-R,R]
    \qquad\text{for all }t\in\Big[\tfrac{T}{2m},\,T\Big].
\]
In other words, almost surely, the supports of $u(t)$ are uniformly bounded on every compact subinterval of $(0,T]$. See Section \ref{subsec:CSP_proof_regular}.
\end{remark}

\begin{remark}
Let us briefly explain the roles of the parameter assumptions in Theorem \ref{thm:CSP_regular_initial}.

First, the condition $p>\frac3\eta$ ensures positive H\"older
regularity of the initial function via the Sobolev embedding
(Lemma~\ref{lem:prop_of_bessel_space}\eqref{sobolev-embedding}),
which is needed to apply the compact-support result to $v_n$, whose
initial datum is $u_0\rho(\cdot/n)$; see \eqref{eq:spde_v_n}. The condition $p>6$ together with $p>\frac3\eta$ allows $\eta$ to be lowered below $1/2$ while preserving $\eta p>3$. The upper bound $\eta<1/2$ is needed to apply Lemma \ref{lem:white_noise_estimate} to the noise term; see the beginning of Section \ref{sec:splitting_estimates}.

Second, the conditions involving $p\gamma$ control the noise term in the uniform $L^\infty$-estimates for $v_n$ and $w_n$; see
Lemma \ref{lem:uniform_bound} and Corollary \ref{cor:L_epsilon_bound_w_n}.

Finally, the assumptions on the spatial tail of $u_0$ force the tail
part to become extinct before any fixed positive time; see Lemma \ref{lem:tail_extinction}. Note that when $\lambda=1$ and $p\gamma\ge1$,
this mechanism requires only finite mass. In particular, for the super-Brownian coefficient $\sigma(u)=\sqrt u$ (so that $\gamma=1/2$ and $\lambda=1$), the condition $p\gamma\ge1$ holds
automatically since $p>6$, and finite mass alone suffices. When $\lambda\in(1,2)$, the extinction argument requires quantitative decay of the initial tail, leading to the threshold $q>\frac{\lambda-1}{2-\lambda}$. Thus, the weaker the noise near
zero (i.e., the larger $\lambda$), the stronger the required spatial decay.
\end{remark}

We now turn to measure-valued initial data. In this case the nonnegative solution is regular only at positive times. We therefore restart the equation at a fixed time $t_0>0$ and apply Theorem \ref{thm:CSP_regular_initial} to the restarted process. Accordingly, the uniqueness-in-law hypothesis is required not for \eqref{eq:SPDE} itself but for its time-shifted versions, which we now formulate.

\begin{assumption}
\label{ass:shifted_uniqueness}
For every $s\in[0,T)$, uniqueness in law holds for nonnegative weak
solutions on $[0,T-s]$ of the shifted equation
\[
    \partial_t z(t,x)
    = L^{(s)}z(t,x) + \sigma(z(t,x))\,\xi^{(s)}(t,x),
    \qquad z(0,\cdot)=Z_0,
\]
where $L^{(s)} := a(s+t,x)\partial_x^2 + b(s+t,x)\partial_x + c(s+t,x)$, $Z_0$ is a nonnegative $\mathcal F_s$-measurable random initial
condition,  and $\xi^{(s)}$ is the shifted white noise, defined by
$\xi^{(s)}(\varphi):=\xi(\varphi(\cdot-s,\cdot))$ for
$\varphi\in C_c^\infty(\bR_+\times\bR)$, which is a space-time white
noise relative to the filtration $\{\mathcal F_{s+t}\}_{t\ge0}$.
Uniqueness in law is understood with the joint law of $Z_0$ and the
shifted coefficient field
$\bigl(a(s+\cdot,\cdot),\,b(s+\cdot,\cdot),\,c(s+\cdot,\cdot)\bigr)$
fixed.
\end{assumption}

\begin{theorem}[Existence and CSP for measure-valued initial data]
\label{thm:CSP_measure_initial}
Suppose Assumptions \ref{asp:sigma} and \ref{asp:coefficients} hold,
and let $\gamma\in(0,1)$ and $\lambda\in[1,2)$ be the constants in
Assumption \ref{asp:sigma}.
Let $\mu_0$ be a nonnegative $\mathcal F_0$-measurable random Radon
measure on $\bR$ such that
\[
    \mu_0\in L_2\bigl(\Omega;H_2^{-1/2-\kappa}(\bR)\bigr)
    \qquad\text{for every }\kappa\in(0,1/2),
\]
and assume one of the following conditions:
\begin{enumerate}
\item[\rm (i)] if $\lambda=1$, then $\E[\mu_0(\bR)]<\infty$;
\item[\rm (ii)] if $\lambda\in(1,2)$, then there exists
$q>\frac{\lambda-1}{2-\lambda}$ such that
$\E\int_{\bR}(1+x^2)^{q/2}\,\mu_0(\ud x)<\infty$.
\end{enumerate}
Then the following hold.
\begin{enumerate}
\item[\rm (a)] {\rm(Existence)} There exists a nonnegative solution $u\in C((0,T];C_{\rm tem})$ of \eqref{eq:SPDE}.
\item[\rm (b)] {\rm(Compact support)} If, in addition,
Assumption \ref{ass:shifted_uniqueness} holds, then every
nonnegative solution $u\in C((0,T];C_{\rm tem})$ of \eqref{eq:SPDE} has compact support at every $t\in(0,T]$, almost surely.
\end{enumerate}
\end{theorem}

To the best of our knowledge, part {\rm(a)} is the first weak existence result from measure-valued initial data at this level of generality, where $\sigma$ is non-Lipschitz and $L$ is a general, possibly random, second-order operator. The overall strategy---mollifying the non-Lipschitz coefficient, solving the regularized equations, and extracting a limit by tightness---goes back to \cite{mueller1992compact,shiga1994two} for continuous initial data, and is outlined for measure-valued initial data in \cite[Section 2]{ChenXia}. All of these works, however, concern the Laplacian, for which the identification of the limit is comparatively direct. For a general  random operator $L$, the coefficient field must be carried through the Skorokhod representation. Here, the approximating solutions and the limit are adapted to different filtrations, and we identify the limiting martingale problem by a localization argument that accommodates both the varying filtrations and the random coefficients (Proposition \ref{prop:existence_measure_initial}). The measure-valued initial
datum enters earlier, at the compactness step. Since no estimate uniform in $n$ is available up to $t=0$,  tightness can only be established in $C((0,T];C_{\rm tem})$ (Proposition \ref{prop:tightness}).

\begin{remark}
The reader may notice that Theorem \ref{thm:CSP_measure_initial} imposes no
condition of the form $p\gamma\ge1$, in contrast with Theorem
\ref{thm:CSP_regular_initial}. The reason is that in Theorem
\ref{thm:CSP_regular_initial} the exponent $p$ is part of the initial Sobolev
regularity assumption, so that the moment condition must distinguish between
$p\gamma\ge1$ and $p\gamma<1$. For measure-valued initial data the proof first
regularizes the solution at a positive time and is then free to choose $p$
before restarting the equation (see Proposition
\ref{prop:positive_time_Sobolev}). In particular, when $\lambda=1$ finite mass
alone suffices. That is, we choose $p$ with $p\gamma\ge1$ after the regularization and
apply the finite-mass branch of Theorem \ref{thm:CSP_regular_initial} (see
Section \ref{sec:proof_thm_measure}).
\end{remark}

\begin{remark}[On the uniqueness-in-law assumptions]
\label{rmk:uniqueness_discussion}
Theorems \ref{thm:CSP_regular_initial} and \ref{thm:CSP_measure_initial} are
conditional on uniqueness in law for nonnegative solutions of \eqref{eq:SPDE}
(and, for Theorem~\ref{thm:CSP_measure_initial}, of its time-shifted versions;
see Assumption \ref{ass:shifted_uniqueness}). For sublinear noise coefficients, well-posedness (especially uniqueness) is a challenging and open problem (see, e.g.,  \cite{Myt98, MPS06, MP11}). For the model equation
\eqref{eq:model_SPDE} with $\sigma(u)=u^\gamma\mathbf 1_{\{u\ge0\}}$, pathwise
non-uniqueness of \emph{signed} solutions holds for $\gamma<3/4$
(\cite{MMP14}), whereas uniqueness in law of \emph{nonnegative} solutions is
known for $\gamma\in[1/2,1)$: for $\gamma=1/2$ by the classical theory of
one-dimensional super-Brownian motion \cite{Per02}, and for $\gamma\in(1/2,1)$
by Mytnik's duality argument \cite{Myt98}, which extends with minor
modifications to finite measure-valued initial data.

For $\gamma\in(0,1/2)$ the question is open, but uniqueness in law of
nonnegative solutions is expected. The counterexample of \cite{BMP10} for
nonnegative solutions establishes pathwise, not weak, non-uniqueness, and
relies on an additive immigration term absent from \eqref{eq:model_SPDE}.
The corresponding one-dimensional equation $\ud X_t=X_t^\gamma\,\ud B_t$ is
instructive. Here, a nonnegative solution is a nonnegative local martingale, hence
absorbed at $0$, and uniqueness in law holds for every $\gamma>0$, whereas
without the sign constraint it already fails for $\gamma<1/2$.

Since the coefficients in \eqref{eq:model_SPDE} are time-independent, every
time-shifted equation coincides with the original one, and Assumption
\ref{ass:shifted_uniqueness} holds automatically. Our results are therefore
unconditional in these cases; see Corollary~\ref{cor:model_case}.
\end{remark}

\begin{corollary}\label{cor:model_case}
Suppose $L=\partial_x^2$  and $\sigma(u)=u^\gamma\mathbf 1_{\{u\ge0\}}$ for $\gamma\in[1/2,1)$. 
\begin{enumerate}[\rm (i)]
\item
If $u_0$ satisfies the regularity and moment assumptions of Theorem \ref{thm:CSP_regular_initial}, with $\lambda=2\gamma$, then every nonnegative solution $u\in C([0,T];C_{\rm tem})$ has compact support for every $t\in (0,T]$ almost surely.
% \[
%     \bP\left( \text{$u(t)$ has compact support for every $t\in(0,T]$}
%     \right)=1.
% \]

\item
If $\mu_0$ satisfies the regularity and moment assumptions of
Theorem \ref{thm:CSP_measure_initial}, with $\lambda=2\gamma$, then every nonnegative solution $u\in C((0,T];C_{\rm tem})$ with initial measure $\mu_0$ has compact support for every $t\in (0,T]$ almost surely.
% \[
%     \bP\left( \text{$u(t)$ has compact support for every $t\in(0,T]$}
%     \right)=1.
% \]
\end{enumerate}
\end{corollary}

%%%%%%%%%%%%%%%%%%%%%%%%%%%%%%%%%%%%%%%%%%%%%%%%

\section{Preliminaries}\label{sec:prelim} 
In this section we introduce the notion of solution, the function spaces, and some basic estimates used throughout the paper. Most of these are standard and we collect them here, with references, for completeness and to fix notation.

\begin{definition}[Continuous function spaces]
For $A\subset\bR$, let $C(A)$ denote the space of real-valued continuous
functions on $A$. We define
\[
    C_{\rm tem}(A)
    :=
    \left\{
        u\in C(A):
        \sup_{x\in A}|u(x)|e^{-a|x|}<\infty
        \text{ for every }a>0
    \right\}.
\]
For $\delta\in(0,1)$, define
\[
C^\delta(A)
:=
\left\{
u\in C(A):
\|u\|_{C^\delta(A)}
:=
\sup_{x\in A}|u(x)|
+
\sup_{\substack{x,y\in A\\x\ne y}}
\frac{|u(x)-u(y)|}{|x-y|^\delta}
<\infty
\right\}.
\] 
For any subsets $A$ and $B$ of Banach spaces, we denote by $C(A;B)$  the space of continuous functions from $A$ to $B$. We omit $A$ (or $B$) if $A$ (or $B$) is $\bR$. We define $C_{c}^{\infty}$ as the space of infinitely differentiable function with compact support.
\end{definition}

Given a filtered probability space carrying a space-time white noise $\xi$ and a coefficient field $(a,b,c)$ as in Section \ref{sec:main_results}, we can now specify what it means for a process $u$ on that space to be a solution.

\begin{definition}[Solution]\label{def:weak_soln}
For a nonnegative $\mathcal F_0$-measurable random initial function $u_0$, we say that
$u$ is a solution of \eqref{eq:SPDE} on $[0, T]$ with initial
function $u_0$ if $u$ is a $C_{\rm tem}(\bR)$-valued
process $u=u(t,\cdot)$, defined on $\Omega\times[0, T]$, such that for every $\varphi\in C_c^\infty(\bR)$:
\begin{enumerate}[(i)]
\item
The process $\int_{\bR}\varphi(x)u(t\wedge T,x)\,\ud x$ is well-defined, $\mathcal F_t$-adapted, and continuous.

\item The process  $\int_{\bR}|\sigma(u(t\wedge T,x))\varphi(x)|^2\,\ud x$ is well-defined, $\mathcal F_t$-adapted, measurable in $(\omega,t)$, and
\[
\int_0^T\int_{\bR}
|\sigma(u(t,x))\varphi(x)|^2\,\ud x\,\ud t
<\infty \quad \text{a.s.}
\]
\item
For all $0<t\le T$, 
\[
(u(t,\cdot),\varphi) = (u_0,\varphi)
+ \int_0^t (u(s,\cdot),( L^*\varphi)(s,\cdot))\,\ud s  
+\int_0^t\int_{\bR}
\sigma(u(s,x))\varphi(x)\xi(\ud s,\ud x) \quad \text{a.s.}
\]
Here $L^*$ is the adjoint operator of $L$ introduced in \eqref{second order operator}, namely
$L^*\varphi = \partial_x^2(a\varphi) -\partial_x(b\varphi)+c\varphi$. 
\end{enumerate}
When the initial state is a nonnegative random Radon measure
$\mu_0$, we say that $u\in C((0,T];C_{\rm tem})$ is a solution of
\eqref{eq:SPDE} with initial measure $\mu_0$ if $u$ satisfies
(i)--(iii) with $(u_0,\varphi)$ replaced by
$\int_{\bR}\varphi\,\ud\mu_0$.
\end{definition}

Such solutions are weak in two respects. Analytically, the equation is required to hold only after testing against $\varphi\in C_c^\infty(\bR)$.
Probabilistically, our results do not fix the filtered probability space or the
white noise. Existence means that such a space can be constructed together with a solution on it, uniqueness in law is understood across all such spaces, and an assertion about every nonnegative solution refers to a solution on any of
them. What is fixed throughout is only the joint law of the initial state and the coefficient field.

\begin{remark}
Throughout the paper, we frequently write the stochastic integral in
Definition~\ref{def:weak_soln}(iii) as an infinite sum of It\^o
integrals:
\begin{equation*}
    \int_0^t\int_{\bR}\sigma(u(s,x))\varphi(x)\,\xi(\ud s,\ud x)
    = \sum_{k=1}^\infty\int_0^t\!\!\int_{\bR}
      \sigma(u(s,x))\varphi(x)e_k(x)\,\ud x\,\ud w_s^k,
\end{equation*}
where $(w^k)_{k\ge1}$ is a family of independent one-dimensional
Brownian motions and $(e_k)_{k\ge1}$ is an orthonormal basis of
$L_2(\bR)$; see \cite{kry1999analytic}.
\end{remark}

\begin{definition}[Bessel potential spaces]
Let $1<p<\infty$ and $\eta\in\bR$. The space $H_p^\eta=H_p^\eta(\bR)$ is the set of all tempered distributions $u$ on $\bR$ satisfying
\[
\|u\|_{H_p^\eta}
:=
\|(1-\Delta)^{\eta/2}u\|_{L_p}
=
\left\|
\mathcal F^{-1}\left[
(1+|\xi|^2)^{\eta/2}\mathcal F(u)(\xi)
\right]
\right\|_{L_p}
<\infty .
\]
Similarly, $H_p^\eta(\ell_2)=H_p^\eta(\bR;\ell_2)$ is the space of
$\ell_2$-valued functions $g=(g^1,g^2,\ldots)$ satisfying
\[
\|g\|_{H_p^\eta(\ell_2)}
:=
\left\|
\left|(1-\Delta)^{\eta/2}g\right|_{\ell_2}
\right\|_{L_p}
<\infty .
\]
For $\eta=0$, we set $H_p^0:=L_p$ and $H_p^0(\ell_2):=L_p(\ell_2)$.
\end{definition}

\begin{remark}[Bessel kernel]
\label{rem:bessel_kernel}
For $\eta>0$ and $u\in C_c^\infty$,
\[
(1-\Delta)^{-\eta/2}u(x)
=
\int_{\bR}R_\eta(x-y)u(y)\,\ud y,
\]
where the Bessel kernel $R_\eta$ satisfies
\[
|R_\eta(x)|
\le
N(\eta)\left(
e^{-|x|/2}\mathbf 1_{\{|x|\ge2\}}
+
A_\eta(x)\mathbf 1_{\{|x|<2\}}
\right),
\]
with
\[
A_\eta(x)
=
\begin{cases}
|x|^{\eta-1}+1+O(|x|^{\eta+1}), & 0<\eta<1,\\
\log(2/|x|)+1+O(|x|^2), & \eta=1,\\
1+O(|x|^{\eta-1}), & \eta>1.
\end{cases}
\]
In particular, for every $1\le r<\infty$,
\[
R_\alpha\in L_r(\bR)
\qquad\text{whenever}\qquad
\alpha>1-\frac1r .
\]
Moreover, for every finite signed Radon measure $\mu$ and every
$\alpha>1-1/p$,
\[
\|R_\alpha*\mu\|_{L_p}
\le
\|R_\alpha\|_{L_p}|\mu|(\bR).
\]
See \cite[Proposition 1.2.5]{grafakos2009modern}.
\end{remark}
The kernel bounds in Remark \ref{rem:bessel_kernel} will be used repeatedly to estimate convolutions against $R_\alpha$. 

We next introduce the spaces of pointwise multipliers on $H_p^\eta$. To estimate, for instance, products $au$ in $H_p^\eta$, we need the multiplier $a$ to carry enough smoothness; the following spaces quantify this, with the required regularity increasing with $|\eta|$. 

\begin{definition}\label{def_pointwise_multiplier}
Fix $\eta\in\bR$ and $\alpha\in[0,1)$ such that $\alpha=0$ if
$\eta\in\bZ$, and $\alpha>0$ if $|\eta|+\alpha$ is not an integer. Define
\[
B^{|\eta|+\alpha}
=
\begin{cases}
B(\bR), & \eta=0,\\
C^{|\eta|-1,1}(\bR), & \eta\text{ is a nonzero integer},\\
C^{|\eta|+\alpha}(\bR), & \text{otherwise},
\end{cases}
\]
and
\[
B^{|\eta|+\alpha}(\ell_2)
=
\begin{cases}
B(\bR;\ell_2), & \eta=0,\\
C^{|\eta|-1,1}(\bR;\ell_2), & \eta\text{ is a nonzero integer},\\
C^{|\eta|+\alpha}(\bR;\ell_2), & \text{otherwise}.
\end{cases}
\]
Here $B(\bR)$ denotes the space of bounded Borel functions,
$C^{|\eta|-1,1}$ denotes the space of $|\eta|-1$ times continuously differentiable
functions whose derivatives of order $|\eta|-1$ are Lipschitz continuous, and
$C^{|\eta|+\alpha}$ denotes the usual H\"older space.
\end{definition}

The following shows some properties of $H_p^\eta$ that will be used frequently throughout this paper.
\begin{lemma}\label{lem:prop_of_bessel_space}
Let $p>1$ and $\eta\in\bR$.
\begin{enumerate}[(i)]
\item 
\label{dense_subset_besseL_potential} 
The space $C_c^\infty$ is dense in $H_p^\eta$. 
\item 
\label{sobolev-embedding} 
Let $\eta-\frac1p=n+\nu$ for some $n=0,1,\ldots$ and $\nu\in(0,1]$. Then, for every
$i\in\{0,1,\ldots,n\}$,
\begin{equation}
    \|D^i u\|_{C(\bR)}
    +
    \|D^n u\|_{\mathcal C^\nu(\bR)}
    \le
    N\|u\|_{H_p^\eta},
\end{equation}
where $N=N(p,\eta)$, and $\mathcal C^\nu$ is the Zygmund space.
\item 
\label{bounded_operator}
For every integer $i\ge0$, the operator $D^i:H_p^\eta\to H_p^{\eta-i}$ is bounded. Moreover, for every $u\in H_p^{\eta+i}$,
\[
    \|D^i u\|_{H_p^\eta}
    \le
    N\|u\|_{H_p^{\eta+i}},
\]
where $N=N(p,\eta,i)$.

\item 
\label{norm_bounded}
Let $p\in(1,\infty)$ and $\mu\le\eta$. If $u\in H_p^{\eta}$, then $u\in H_p^{\mu}$ and
\[ \|u\|_{H_p^{\mu}}\le\|u\|_{H_p^{\eta}}. \]
Moreover, if $q\in(1,\infty)$ satisfies $\eta-\frac{1}{p}=\mu-\frac{1}{q}$, then
\[ \|u\|_{H_q^{\mu}}\le N\|u\|_{H_p^{\eta}}, \]
where $N=N(\eta,\mu,p,q)$.

\item 
\label{multi_ineq}
Let
\[
\varepsilon\in[0,1],
\qquad
p_i\in(1,\infty),
\qquad
\eta_i\in\bR,
\qquad i=0,1,
\]
and suppose
\[
    \eta=\varepsilon\eta_0+(1-\varepsilon)\eta_1,
    \qquad
    \frac1p=\frac{\varepsilon}{p_0}+\frac{1-\varepsilon}{p_1}.
\]
Then 
\[
    \|u\|_{H_p^\eta}
    \le
    \|u\|_{H_{p_0}^{\eta_0}}^\varepsilon
    \|u\|_{H_{p_1}^{\eta_1}}^{1-\varepsilon}.
\]
In particular, for every $\delta>0$,
\[
    \|u\|_{H_p^\eta}
    \le
    \delta\|u\|_{H_{p_0}^{\eta_0}}
    +
    N(\delta)\|u\|_{H_{p_1}^{\eta_1}} .
\]

\item \label{pointwise_multiplier} If $u\in H_p^\eta$, $a\in B^{|\eta|+\alpha}$ and $b\in B^{|\eta|+\alpha}(\ell_2)$, then
\[
    \|au\|_{H_p^\eta}
    \le
    N\|a\|_{B^{|\eta|+\alpha}}\|u\|_{H_p^\eta},
\]
and
\[
    \|bu\|_{H_p^\eta(\ell_2)}
    \le
    N\|b\|_{B^{|\eta|+\alpha}(\ell_2)}\|u\|_{H_p^\eta},
\] 
where $N=N(\alpha,\eta,p)$. 
\item
\label{convolution}
Let $\Theta\in C_c^\infty$ be nonnegative with
$\int_{\bR}\Theta(x)\,\ud x=1$, and set $\Theta_k(x):=k\Theta(kx)$ for
$k\ge1$. Then, for every $u\in H_p^\eta$,
\[
    \|u*\Theta_k\|_{H_p^\eta}
    \le
    \|u\|_{H_p^\eta}.
\]

\end{enumerate}
\end{lemma}

\begin{proof}
These results are standard. For \eqref{dense_subset_besseL_potential}--\eqref{multi_ineq}, 
see \cite[Theorem 13.3.7(i), Theorem 13.8.1, Theorem 13.3.10, Corollary 13.3.9, Theorem 13.8.7,
Theorem 13.3.7(ii), Exercise 13.3.20]{krylov2008lectures}, respectively.
For \eqref{pointwise_multiplier} and \eqref{convolution}, see \cite[Lemma 5.2]{kry1999analytic}.
\end{proof}

We now introduce the stochastic Banach spaces, which are  counterparts of the Bessel potential spaces, incorporating the temporal and probabilistic variables.
\begin{definition}[Stochastic Banach spaces]
For a bounded stopping time $\tau\le T$, set
\[
    \opar0,\tau\cbrk
    :=
    \{(\omega,t):0<t\le \tau(\omega)\}.
\]
\begin{enumerate}[(i)]
\item
\label{def:Hp,Up}
For $\tau\le T$, define
\begin{gather*}
\bH_{p}^{\eta}(\tau) := L_p(\opar0,\tau\cbrk, \mathcal{P}, \ud\P \times \ud t ; H_{p}^\eta),\\
\bH_{p}^{\eta}(\tau,\ell_2) := L_p(\opar0,\tau\cbrk,\mathcal{P}, \ud\P \times \ud t;H_{p}^\eta(\ell_2)),\\
U_{p}^{\eta} :=  L_p(\Omega,\cF_0, \ud\P ; H_{p}^{\eta-2/p}).
\end{gather*}  
For convenience, we write $\mathbb L_p(\tau):=\mathbb H_p^0(\tau)$ and  $\mathbb L_p(\tau,\ell_2):=\mathbb H_p^0(\tau,\ell_2).$ 

\item
The norm of each space is defined in the natural way. For example,
\begin{equation}
\label{norm}
    \|u\|_{\mathbb H_p^\eta(\tau)}^p
    :=
    \E\int_0^\tau
    \|u(t)\|_{H_p^\eta}^p\,\ud t.
\end{equation}
\end{enumerate} For another bounded stopping time $\bar{\tau}\le \tau$, we write $\bH_{p}^{\eta}( \bar{\tau},\tau)$ for the corresponding space on the shifted time interval $\opar \bar{\tau},\tau\cbrk$.
\end{definition}

The following elementary estimate controls the white-noise term of \eqref{eq:SPDE} in the $\ell_2$-valued Bessel potential spaces, allowing Krylov's $L_p$-theory to be applied to the stochastic integral. 

\begin{lemma}\label{lem:white_noise_estimate}
Let $(e_k)_{k\ge1}$ be an orthonormal basis of $L_2(\bR)$. Let $p\ge2$ and $0<\eta<1/2$. Then, for every measurable
$f:\bR\to\bR$ with $f\in L_p(\bR)$,
\[
    \|f\mathbf e\|_{H_p^{\eta-1}(\ell_2)}
    \le
    N\|f\|_{L_p},
\]
where
\[
    f\mathbf e:=(fe_1,fe_2,\ldots).
\]
Consequently, for every bounded stopping time $\tau\le T$ and every
predictable process $f\in\mathbb L_p(\tau)$,
\[
    \|f\mathbf e\|_{\mathbb H_p^{\eta-1}(\tau,\ell_2)}
    \le
    N\|f\|_{\mathbb L_p(\tau)}.
\]
\end{lemma}

\begin{proof}
By density it suffices to treat $f\in C_c^\infty(\bR)$. By the Bessel kernel
representation and Parseval's identity,
\[
    \|f\mathbf e\|_{H_p^{\eta-1}(\ell_2)}^p
    =
    \int_{\bR}
    \left(
        \int_{\bR}R_{1-\eta}(x-y)^2|f(y)|^2\,\ud y
    \right)^{p/2}\ud x
    =
    \bigl\|R_{1-\eta}^2*|f|^2\bigr\|_{L_{p/2}}^{p/2}.
\]
Since $R_{1-\eta}\in L_2(\bR)$ for $\eta<1/2$, Young's inequality yields
$\bigl\|R_{1-\eta}^2*|f|^2\bigr\|_{L_{p/2}}\le\|R_{1-\eta}\|_{L_2}^2\|f\|_{L_p}^2$,
which gives the first estimate. The second estimate follows by taking the
$L_p(\opar0,\tau\cbrk)$-norm.
\end{proof}

We introduce the solution space $\mathcal H_p^\eta(\tau)$ below.  It is the natural space in which Krylov's theory below is formulated.
\begin{definition}
\label{def of cH}
Let $\tau\le T$ be a bounded stopping time and $u\in\mathbb H_p^\eta(\tau)$.
\begin{enumerate}[(i)]
\item 
We write $u\in\cH^{\eta}_p(\tau)$ if $u_0\in U_{p}^{\eta}$ and there exists $(f,g)\in
\bH_{p}^{\eta-2}(\tau)\times\bH_{p}^{\eta-1}(\tau,\ell_2)$ such that
\begin{equation*}
\ud u = f\ud t+\sum_{k=1}^{\infty} g^k \ud w_t^k,\quad   t\in (0, \tau]\,; \quad u(0,\cdot) = u_0
\end{equation*}
in the sense of distributions, i.e., for any $\varphi\in C_c^\infty$, the equality
\begin{equation} \label{def_of_soL_2}
(u(t,\cdot),\varphi) = (u_0,\varphi) + \int_0^t(f(s,\cdot),\varphi)\ud s + \sum_{k=1}^{\infty} \int_0^t(g^k(s,\cdot),\varphi)\ud w_s^k
\end{equation}
holds for all $t\in [0,\tau]$ almost surely. Here, $(w^k)_{k\geq 1}$ is a family of independent one-dimensional Brownian motions. In this case, we write
\begin{equation*}
\bD u:= f,\quad \bS u:=g.
\end{equation*}
\item
The norm of $\cH_{p}^{\eta}(\tau)$ is defined by
\begin{equation*}
\| u \|_{\cH_{p}^{\eta}(\tau)} :=  \| u \|_{\bH_{p}^{\eta}(\tau)} + \| \bD u \|_{\bH_{p}^{\eta-2}(\tau)} + \| \bS u \|_{\bH_{p}^{\eta-1}(\tau,\ell_2)} + \| u_0 \|_{U_{p}^{\eta}}.
\end{equation*}

\end{enumerate} For another bounded stopping time $\bar{\tau}\le \tau$, we write $\cH_{p}^{\eta}( \bar{\tau},\tau)$ for the corresponding space on the shifted time interval $\opar \bar{\tau},\tau\cbrk$.
\end{definition}

Next, we introduce the H\"older embedding theorem for $\cH_p^\eta(\tau)$, which will be used repeatedly to turn $\mathcal H_p^\eta(\tau)$-bounds into pointwise supremum and H\"older estimates for the solution.

\begin{theorem}
\label{thm:embedding_theorems}
Let $\tau\le T$ be a bounded stopping time. 
\begin{enumerate}[(i)]
\item
\label{sol space is banach}
For every $p\ge2$, $\mathcal H_p^\eta(\tau)$ is a Banach space with the norm $\|\cdot\|_{\mathcal H_p^\eta(\tau)}$.
\item  \label{embedding in time}
If $\eta\in\R$, $p>2$, and $1/p<\ba<\bb<1/2$, then for every $u\in\cH_p^\eta(\tau)$, we have 
\[ u\in C^{\ba-1/p}([0,\tau];H_p^{\eta-2\bb}) \quad a.s.\]
and
\begin{equation}
\label{holder_embedding}
\E\| u \|_{C^{\ba-1/p}([0,\tau];H_p^{\eta-2\bb})}^p \leq N(\ba,\bb,p,T)\| u \|_{\cH_p^\eta(\tau)}^p.
\end{equation}
\item
\label{embedding in time 2}
If $p=2$, then \eqref{holder_embedding} holds with
$\ba=\bb=1/2$, i.e., $u\in C([0,\tau];H_2^{\eta-1})$ a.s. 
and
\begin{equation*}
% \label{holder_embedding_2}
    \E\sup_{t\le\tau}
    \|u(t)\|_{H_2^{\eta-1}}^2
    \le
    N(T)
    \|u\|_{\mathcal H_2^\eta(\tau)}^2 .
\end{equation*}

\item  \label{embedding in time and space}
If $\eta>0$, $p>2$, $\ba,\bb\in(0,\infty)$ satisfy
\begin{equation*}
\frac{1}{p} < \ba < \bb < \frac{1}{2}\left( \eta - \frac{1}{p} \right),
\end{equation*}
then for every $u\in\cH_p^\eta(\tau)$, we have $u\in C^{\ba-1/p}([0,\tau];C^{\eta-2\bb-1/p})$ a.s. and 
\begin{equation*}
\E\| u \|_{C^{\ba-1/p}([0,\tau];C^{\eta-2\bb-1/p})}^p \leq N(\ba,\bb,p,T)\| u \|_{\cH_p^\eta(\tau)}^p.
\end{equation*}

\end{enumerate}
\end{theorem}

\begin{proof}
For \eqref{sol space is banach}--\eqref{embedding in time 2}, see
\cite[Theorems 3.7, 7.1(iii) and 7.2]{kry1999analytic}. For
\eqref{embedding in time and space}, combine
\eqref{holder_embedding} with the Sobolev embedding Lemma~\ref{lem:prop_of_bessel_space}\eqref{sobolev-embedding}. Indeed,
\[{}
\begin{aligned}
\E
\|u\|_{C^{\mathfrak a-1/p}([0,\tau];C^{\eta-2\mathfrak b-1/p})}^p
&\le
N
\E
\|u\|_{C^{\mathfrak a-1/p}([0,\tau];H_p^{\eta-2\mathfrak b})}^p  \le
N\|u\|_{\mathcal H_p^\eta(\tau)}^p .
\end{aligned}
\]
\end{proof}

We conclude this section with Krylov's $L_p$-estimate \cite[Theorem~5.1]{kry1999analytic}, which provides the basic solvability and a priori bounds for linear SPDEs in the spaces $\mathcal H_p^\eta(\tau)$.

\begin{theorem}[Krylov's $L_p$-estimate, Theorem~5.1 of \cite{kry1999analytic}]
\label{thm:krylov_lp_estimate}
Let $p\ge2$, $\eta\in\bR$, and let $\tau\le T$ be a bounded stopping time.
Assume that the coefficients of $L$ satisfy
Assumption \ref{asp:coefficients}. Suppose
\[
    u_0\in U_p^\eta,
    \qquad
    f\in\mathbb H_p^{\eta-2}(\tau),
    \qquad
    g\in\mathbb H_p^{\eta-1}(\tau,\ell_2).
\]
If $u\in\mathcal H_p^\eta(\tau)$ satisfies
\begin{equation}\label{eq:SPDE1}
    \ud u=(L u+f)\,\ud t+\sum_{k=1}^{\infty}g^k\,\ud w_t^k,
    \qquad
    u(0)=u_0,
\end{equation}
then
\[
    \|u\|_{\mathcal H_p^\eta(\tau)}
    \le
    N\left(
        \|u_0\|_{U_p^\eta}
        +
        \|f\|_{\mathbb H_p^{\eta-2}(\tau)}
        +
        \|g\|_{\mathbb H_p^{\eta-1}(\tau,\ell_2)}
    \right),
\]
where $N=N(p,\eta,\nu,T)$. 
% In particular, for $p=2$, \violet{JY: Why do we write it for $p=2$?}
% \[
%     \|u\|_{\mathcal H_2^\eta(\tau)}^2
%     \le
%     N\left(
%         \|u_0\|_{U_2^\eta}^2
%         +
%         \|f\|_{\mathbb H_2^{\eta-2}(\tau)}^2
%         +
%         \|g\|_{\mathbb H_2^{\eta-1}(\tau,\ell_2)}^2
%     \right).
% \]
Moreover, under the same assumptions, the equation \eqref{eq:SPDE1} has a unique solution
$u\in\mathcal H_p^\eta(\tau)$.
\end{theorem}

%%%%%%%%%%%%%%%%%%%%%%%%%%%%%%%%%%%%%%%%%%%%%%%%%%%%%%

\section{The splitting construction and a priori estimates}
\label{sec:splitting_estimates}

In this section we provide several estimates that will be used in the proof of Theorem \ref{thm:CSP_regular_initial}. The main idea of the proof of Theorem \ref{thm:CSP_regular_initial} is to split
the initial function into a compactly supported part and a tail part. The compactly supported part inherits the compact support property from the known case of compactly supported initial data, while the tail part will be shown to become extinct before any fixed positive time with high probability. 

Throughout this section and the next, $T>0$ is fixed, Assumptions \ref{asp:sigma} and
\ref{asp:coefficients} are in force, and $u_0$ is a nonnegative
$\cF_0$-measurable random function satisfying
\eqref{eq:asp_initial_regularity}. Since $p>3/\eta$ and $p>6$, we may lower
$\eta$ and assume throughout that
\[
    0<\eta<\tfrac12\qquad \text{and} \qquad \eta p>3 ,
\]
which is legitimate because $\|u_0\|_{U_p^{\eta'}}\le\|u_0\|_{U_p^{\eta}}$ for
$\eta'\le\eta$ (Lemma \ref{lem:prop_of_bessel_space}\eqref{norm_bounded}). The first condition is needed to apply Lemma \ref{lem:white_noise_estimate} to the noise term, and the second gives, by the Sobolev embedding
(Lemma \ref{lem:prop_of_bessel_space}\eqref{sobolev-embedding}), that $u_0\in C^\delta(\bR)$ almost surely for every $0<\delta<\eta-\frac3p$.

Since the coefficient $\sigma$ is non-Lipschitz, the auxiliary
processes below are constructed as weak solutions on a filtered
probability space that may differ from the original one and that
carries a copy of $(u_0,a,b,c)$ with the same joint law (see
Proposition \ref{prop:existence_vw}). As the assumptions on $(u_0,a,b,c)$---the bounds in Assumption \ref{asp:coefficients} and the regularity and moment conditions on $u_0$---depend only on the law of $(u_0,a,b,c)$, they are inherited by any copy with the same law. We therefore do not distinguish the copy notationally in the sequel.

\subsection{The splitting construction}
\label{subsec:splitting_construction}

Let $\rho\in C_c^\infty(\bR)$ satisfy
\[
    0\le \rho\le 1,\qquad
    \rho(x)=1\ \text{for } |x|\le1,\qquad
    \rho(x)=0\ \text{for } |x|\ge2.
\]
For $n\ge1$, we decompose
\[
    u_0 = u_0\rho(\cdot/n) + u_0(1-\rho(\cdot/n)),
\]
and consider the system
\begin{equation}\label{eq:spde_v_n}
    \begin{aligned}
        \partial_t v_n(t,x)
        &= L v_n(t,x)+\sigma(v_n(t,x))\,\xi_1(t,x),\\
        v_n(0,x) &= u_0(x)\rho(x/n),
    \end{aligned}
\end{equation}
and
\begin{equation}\label{eq:spde_w_n}
    \begin{aligned}
        \partial_t w_n(t,x)
        &= L w_n(t,x)
        + \left(\sigma^2(w_n(t,x)+v_n(t,x))
                -\sigma^2(v_n(t,x))\right)^{1/2}\xi_2(t,x),\\
        w_n(0,x) &= u_0(x)(1-\rho(x/n)),
    \end{aligned}
\end{equation}
where $\xi_1$ and $\xi_2$ are independent space-time white noises. Weak existence for this system is established in the following proposition,
where $\delta\in(0,\eta-\frac3p)$ is fixed as at the beginning of this section. 

\begin{proposition}\label{prop:existence_vw}
For each $n\ge1$, there exists a filtered probability space
$(\bar\Omega,\bar{\mathcal F},\{\bar{\mathcal F}_t\},\bar\bP)$
carrying:
\begin{enumerate}
\item[\rm(i)] an $\bar{\mathcal F}_0$-measurable nonnegative random
function $\bar u_0$ and a coefficient field $(\bar a,\bar b,\bar c)$
satisfying Assumption \ref{asp:coefficients}, with
$\mathcal L(\bar u_0,\bar a,\bar b,\bar c)
 =\mathcal L(u_0,a,b,c)$;
\item[\rm(ii)] two independent space-time white noises
$\xi_1,\xi_2$ relative to $\{\bar{\mathcal F}_t\}$;
\item[\rm(iii)] nonnegative processes
$v_n,w_n\in C([0,T];C_{\rm tem}(\bR))$ solving
\eqref{eq:spde_v_n} and \eqref{eq:spde_w_n} with
$(u_0,a,b,c)$ replaced by $(\bar u_0,\bar a,\bar b,\bar c)$.
\end{enumerate}
Moreover, for every $p\ge2$, $\ep\in(0,\delta/2)$, and $a>0$, there
exists $N=N(a,\ep,p,T)>0$, independent of $n$, such that
\[
    \E\|\Psi_a v_n\|_{C^{\ep}([0,T]\times\bR)}^p
    + \E\|\Psi_a w_n\|_{C^{\ep}([0,T]\times\bR)}^p
    \le N,
    \qquad
    \Psi_a(x):=\zeta(ax)=\frac1{\cosh(ax)}.
\]
\end{proposition}

\begin{proof}
The equation \eqref{eq:spde_v_n} falls directly under \cite[Theorem 2.9]{han2023compact}, which yields a nonnegative solution $v_n\in C([0,T];C_{\rm tem}(\bR))$ together with the stated moment bound. Note that the constant $N$ does not depend on $n$ since $v_n(0)\le u_0$ uniformly in $n$.

For \eqref{eq:spde_w_n}, after $v_n$ is constructed, we regard
\[
    \widetilde\sigma_n(t,x,z)
    :=
    \left(\sigma^2(z+v_n(t,x))-\sigma^2(v_n(t,x))\right)^{1/2},
    \qquad z\ge0,
\]
as a predictable random coefficient. By Assumption \ref{asp:sigma}, $\widetilde\sigma_n(t,x,0)=0$, $\widetilde\sigma_n$ is jointly continuous in $(t,x,z)$ by the continuity of $v_n$, and the linear growth estimate follows from
\[
    \widetilde\sigma_n(t,x,z)
    \le
    \sigma(z+v_n(t,x))+\sigma(v_n(t,x))
    \le
    N\left(1+v_n(t,x)+z\right).
\]
Together with the uniform moment bound for $\Psi_a v_n$, these properties are all that the existence proof of \cite[Theorem 2.9]{han2023compact} requires of the noise coefficient, so the
same argument applies with $\sigma$ replaced by $\widetilde\sigma_n$ and yields a nonnegative solution $w_n\in C([0,T];C_{\rm tem}(\bR))$. The moment bound for $w_n$ follows in the same way, with constants uniform in $n$ because $w_n(0)\le u_0$.
\end{proof}

We now consider $y_n:=v_n+w_n$, which turns out to be a solution of \eqref{eq:SPDE} itself, as the following lemma shows.

\begin{lemma}\label{lem:summing}
Let $y_n:=v_n+w_n$. Then, for each $n\ge1$, on the filtered probability space of Proposition \ref{prop:existence_vw}, there exists a space-time white noise $\xi^{(n)}$ relative to $\{\bar{\mathcal F}_t\}$ such that  $y_n$ is a solution of
\[
    \partial_t y_n(t,x)
    =
     L y_n(t,x)+\sigma(y_n(t,x))\xi^{(n)}(t,x),
    \qquad
    y_n(0,x)=u_0(x).
\] 
\end{lemma}

\begin{proof}
Since $0\le v_n\le y_n$ and $r\mapsto\sigma^2(r)$ is nondecreasing on $[0,\infty)$ by Assumption \ref{asp:sigma}(iii),
\[
    D_n:=\sigma^2(y_n)-\sigma^2(v_n)\ge0,
\]
and moreover $\sigma(v_n)=0$ whenever $\sigma(y_n)=0$. Define the predictable coefficients
\[
    \alpha_n:=\frac{\sigma(v_n)}{\sigma(y_n)}\mathbf 1_{\{\sigma(y_n)>0\}}
    +\mathbf 1_{\{\sigma(y_n)=0\}},
    \qquad
    \beta_n:=\frac{\sqrt{D_n}}{\sigma(y_n)}\mathbf 1_{\{\sigma(y_n)>0\}},
\]
so that $\alpha_n^2+\beta_n^2=1$, $\sigma(y_n) \alpha_n=\sigma(v_n)$, and $\sigma(y_n)\beta_n=\sqrt{D_n}$. Since a linear combination $\alpha_n\xi_1+\beta_n\xi_2$ of independent space-time white noises with predictable coefficients satisfying $\alpha_n^2+\beta_n^2=1$ is again a space-time white noise \cite[Corollary A.7]{KKMS23}, $\xi^{(n)}:=\alpha_n\xi_1+\beta_n\xi_2$ is a space-time white noise and
\[
    \sigma(v_n)\xi_1+\sqrt{D_n}\,\xi_2=\sigma(y_n)\xi^{(n)}.
\]
Adding the equations for $v_n$ and $w_n$ shows that $y_n$ is a solution of the original SPDE driven by $\xi^{(n)}$, with
$y_n(0,\cdot)=u_0\rho(\cdot/n)+u_0(1-\rho(\cdot/n))=u_0$.
\end{proof}

Note that, thanks to Lemma \ref{lem:summing}, for each $n\ge1$, $y_n$ is a solution of the original SPDE with initial function $u_0$, although the driving white noise $\xi^{(n)}$ may depend on $n$. The uniqueness in law assumption in Theorem \ref{thm:CSP_regular_initial} will imply that the law of $y_n$ is independent of $n$.

\subsection{Weighted mass and $L_\epsilon$-estimates}

In this subsection we prove weighted mass estimates for the tail process $w_n$. These estimates have two roles. First, they imply that the total mass of $w_n$ is finite and has a supermartingale structure. Second, when a positive weighted moment is available, they give $L_\epsilon$-bounds for $\epsilon<1$, which will be used later to control sublinear terms of the form $v_n^{p\gamma}$ and $w_n^{p\gamma}$.

Recall the notation
\[
    \Phi_q(x):=(1+x^2)^{q/2},
    \qquad
    W_q(f):=\int_{\bR}\Phi_q(x)f(x)\,\ud x,
\]
so that $W_0(f)=\|f\|_{L_1(\bR)}$.

\begin{lemma}
\label{lem:mtg_bound_w}
Let $\alpha\in(0,1)$, $q\ge0$, and $n\in\mathbb N$, and assume
$\E[W_q(w_n(0))^\alpha]<\infty$. Then there exists a constant
$N=N(\alpha,q,\nu,T)>0$ such that
\begin{equation}\label{eq:weighted_mass_bound}
\E\left[
\sup_{t\le T}
W_q(w_n(t))^\alpha
\right]
\le
N
\E\left[
W_q(w_n(0))^\alpha
\right].
\end{equation}
In particular, $w_n(t)\in L_1(\bR)$ for all $t\in[0,T]$ almost surely.
Furthermore, there exists $K=K(\nu)>0$ such that, for all $0\le s\le t\le T$,
\begin{equation}\label{eq:mass_supermartingale_ineq}
\begin{aligned}
e^{-Kt}\int_{\bR}w_n(t,x)\,\ud x
&\le
e^{-Ks}\int_{\bR}w_n(s,x)\,\ud x  \\
&\quad+
\int_s^t e^{-Kr}\int_{\bR}
\sqrt{
\sigma^2(w_n(r,x)+v_n(r,x))-\sigma^2(v_n(r,x))
}\,
\xi_2(\ud r,\ud x).
\end{aligned}
\end{equation}
Consequently,
$\left(e^{-Kt}\int_{\bR}w_n(t,x)\,\ud x\right)_{t\in[0,T]}$
is a nonnegative continuous supermartingale.
\end{lemma}

\begin{proof}
For simplicity, set
\[
    D_n(t,x)
    :=
    \sigma^2(w_n(t,x)+v_n(t,x))-\sigma^2(v_n(t,x)).
\]
The proof has two parts. We first establish the weighted bound \eqref{eq:weighted_mass_bound} by testing the equation for $w_n$ against the weight $\Phi_q\zeta_m$ and using the maximal inequality for the resulting nonnegative supermartingale. Then, we set $q=0$ and pass to the limit $m\to\infty$ to obtain the supermartingale inequality \eqref{eq:mass_supermartingale_ineq}. Throughout, $\Phi_q$ and $\zeta_m$ are as in \eqref{eq:weight_ft} and \eqref{eq:zeta}, and we use the derivative bounds \eqref{eq:zeta_bounds} for $\zeta_m$, together with $|\Phi_q'|+|\Phi_q''|\le N_q\Phi_q$. In particular, we have 
\[
    |(\Phi_q\zeta_m)'|+|(\Phi_q\zeta_m)''|\le N_q\Phi_q\zeta_m .
\]

We begin with the weighted estimate. By Assumption \ref{asp:coefficients} and the bound above, there exists $K=K(q,\nu)>0$, independent of $m$ and $n$, such
that
\begin{equation}\label{eq:Lstar_weight_bound}
    L^*(\Phi_q\zeta_m)(t,x)-K\Phi_q(x)\zeta_m(x)
    \le0
    \qquad\text{for all }(t,x).
\end{equation}
Since $\Phi_q\zeta_m$ is not compactly supported, we first replace it by $\Phi_q\zeta_m\chi_R$ with a standard cutoff $\chi_R\in C_c^\infty(\bR)$, derive the estimates uniformly in $R$, and let $R\to\infty$. After passing to the limit $R\to\infty$, testing the weak formulation for $w_n$ against $\Phi_q\zeta_m$ and multiplying by $e^{-Kt}$ gives
\begin{equation}\label{eq:weighted_mass_decomp_m}
\begin{aligned}
e^{-Kt}\int_{\bR}w_n(t,x)\Phi_q(x)\zeta_m(x)\,\ud x
&=
\int_{\bR}w_n(0,x)\Phi_q(x)\zeta_m(x)\,\ud x \\
&\quad+
\int_0^t e^{-Ks}
\int_{\bR}w_n(s,x)
\left[L^*(\Phi_q\zeta_m)(s,x)-K\Phi_q(x)\zeta_m(x)
\right]\ud x\,\ud s \\
&\quad+
\int_0^t e^{-Ks}\int_{\bR}
\sqrt{D_n(s,x)}\,\Phi_q(x)\zeta_m(x)\,\xi_2(\ud s,\ud x).
\end{aligned}
\end{equation}
By \eqref{eq:Lstar_weight_bound}, the drift term is nonpositive, so
\[
    S_t^m
    :=
    e^{-Kt}\int_{\bR}w_n(t,x)\Phi_q(x)\zeta_m(x)\,\ud x
\]
is a nonnegative continuous local supermartingale, hence a supermartingale. The maximal inequality for nonnegative supermartingales with exponent $\alpha\in(0,1)$ then gives 
\[ 
\E\left[\sup_{t\le T}(S_t^m)^\alpha\right]\le N_\alpha\E[(S_0^m)^\alpha]. 
\]
Since $e^{-KT}\le e^{-Kt}\le1$, letting $m\to\infty$ and using
$\zeta_m\uparrow1$ with monotone convergence yields
\eqref{eq:weighted_mass_bound}. Since $\Phi_q\ge1$, this also shows $\int_{\bR}w_n(t,x)\,\ud x<\infty$ for all $t\in[0,T]$ almost surely.

It remains to establish \eqref{eq:mass_supermartingale_ineq}. Here, we take $q=0$ in $\Phi_q$, i.e., $\Phi_0\equiv1$. We also take $K=K(\nu)>0$ large enough so that $L^*\zeta_m-K\zeta_m\le0$. Set
\[
    X_t^m
    :=
    e^{-Kt}\int_{\bR}w_n(t,x)\zeta_m(x)\,\ud x
    \quad\text{and}\quad
    M_t^m
    :=
    \int_0^t e^{-Ks}\int_{\bR}
        \sqrt{D_n(s,x)}\,\zeta_m(x)\,\xi_2(\ud s,\ud x),
\]
so that $X_t^m=X_0^m+A_t^m+M_t^m$ with $A^m$ nonincreasing and $A_0^m=0$. In particular,
\begin{equation}\label{eq:two_time_ineq_m}
    X_t^m\le X_s^m+M_t^m-M_s^m,
    \qquad 0\le s\le t\le T,
\end{equation}
and $X_0^m+M_t^m=X_t^m-A_t^m\ge0$ is a nonnegative continuous local martingale, hence a nonnegative supermartingale. From $|M_t^m|\le X_0^m+(X_0^m+M_t^m)$ and the maximal inequality for supermartingales applied to
$X_0^m+M_t^m$, we obtain, for any $\alpha\in(0,1)$,
\begin{equation}\label{eq:maximal_ineq}
\E\left[
    \sup_{t\le T}|M_t^m|^\alpha
\right]
\le
N_\alpha
\E\left[
    \left(\int_{\bR}w_n(0,x)\,\ud x\right)^\alpha
\right].
\end{equation}
Since $\langle M^m\rangle_T\ge e^{-2KT}\int_0^T\int_{\bR}D_n\zeta_m^2\,\ud x\,\ud s$,
the Burkholder--Davis--Gundy inequality together with \eqref{eq:maximal_ineq}
gives
\[
\E\left[
    \left(\int_0^T\int_{\bR}
    D_n(s,x)\zeta_m^2(x)\,\ud x\,\ud s\right)^{\alpha/2}
\right]
\le
N_{\alpha,T}
\E\left[\sup_{t\le T}|M_t^m|^\alpha\right]
\le
N_{\alpha,T}
\E\left[
    \left(\int_{\bR}w_n(0,x)\,\ud x\right)^\alpha
\right],
\]
and letting $m\to\infty$, monotone convergence yields
\begin{equation}
\E\left[
    \left(\int_0^T\int_{\bR}
    D_n(s,x)\,\ud x\,\ud s\right)^{\alpha/2}
\right]
\le
N_{\alpha,T}
\E\left[
    \left(\int_{\bR}w_n(0,x)\,\ud x\right)^\alpha
\right].
\end{equation}
In particular, we have that 
\begin{equation}\label{eq:D_n}
    \int_0^T\int_{\bR}D_n(s,x)\,\ud x\,\ud s<\infty
    \qquad\text{a.s.}
\end{equation}
Thus the stochastic integral
\begin{equation}\label{eq:M_t}
    M_t
    :=
    \int_0^t e^{-Ks}\int_{\bR}
        \sqrt{D_n(s,x)}\,\xi_2(\ud s,\ud x)
\end{equation}
is well-defined and continuous. Moreover,
\[
    \langle M^m-M\rangle_T
    =
    \int_0^T e^{-2Ks}
    \int_{\bR}
        D_n(s,x)(\zeta_m(x)-1)^2\,\ud x\,\ud s
    \longrightarrow0
    \qquad\text{a.s.}
\]
Hence $M^m\to M$ uniformly on $[0,T]$ in probability. Letting $m\to\infty$ in \eqref{eq:two_time_ineq_m} along a subsequence converging almost surely, we have \eqref{eq:mass_supermartingale_ineq}.

Finally, we verify the supermartingale property. Set
\[
    X_t := e^{-Kt}\int_{\bR}w_n(t,x)\,\ud x,
    \qquad t\in[0,T],
\]
which is finite for every $t\in[0,T]$ almost surely by
\eqref{eq:weighted_mass_bound}. By
\eqref{eq:mass_supermartingale_ineq}, for all $0\le s\le t\le T$,
\begin{equation}\label{eq:X_ineq}
    X_t \le X_s + M_t - M_s ,
\end{equation}
where $M$ is the continuous local martingale constructed in \eqref{eq:M_t}. In
particular, $X$ is continuous. Let $(\tau_k)_{k\ge1}$ be a localizing
sequence for $M$, so that $M_{\cdot\wedge\tau_k}$ is a martingale for
each $k$ and $\tau_k\uparrow T$ almost surely. Fix
$0\le s\le t\le T$. Applying \eqref{eq:X_ineq} with the pair
$(s\wedge\tau_k,\,t\wedge\tau_k)$ and taking conditional
expectations, we obtain
\[
    \E\left[X_{t\wedge\tau_k}\,\big|\,\mathcal F_s\right]
    \le
    \E\left[X_{s\wedge\tau_k}\,\big|\,\mathcal F_s\right]
    +
    \E\left[M_{t\wedge\tau_k}-M_{s\wedge\tau_k}
        \,\big|\,\mathcal F_s\right]
    =
    \E\left[X_{s\wedge\tau_k}\,\big|\,\mathcal F_s\right],
\]
where the conditional expectations are well-defined in $[0,\infty]$
since $X\ge0$. We now use conditional Fatou's lemma to obtain 
\[
    \E\left[X_t\,\big|\,\mathcal F_s\right]
    \le
    \liminf_{k\to\infty}
    \E\left[X_{s\wedge\tau_k}\,\big|\,\mathcal F_s\right].
\]
Since  $\{\tau_k\ge s\}\in\mathcal F_{s\wedge\tau_k}\subseteq\mathcal F_s$, we have  $\E\left[X_{s\wedge\tau_k}\,\big|\,\mathcal F_s\right] \mathbf 1_{\{\tau_k\ge s\}} = X_s\,\mathbf 1_{\{\tau_k\ge s\}}$. 
In addition, since $\mathbf 1_{\{\tau_k\ge s\}}\uparrow1$ almost surely, we conclude
\[
    \E\left[X_t\,\big|\,\mathcal F_s\right]\le X_s
    \qquad\text{a.s.}
\]
Hence $X=\left(e^{-Kt}\int_{\bR}w_n(t,x)\,\ud x\right)_{t\in[0,T]}$
is a nonnegative continuous supermartingale.
\end{proof}

The weighted mass bound also yields the following $L_\epsilon$-estimate for $\epsilon<1$, and this will be used later to control the sublinear term $w_n^{p\gamma}$. 

\begin{corollary}\label{cor:L_epsilon_bound_w_n}
Let $q>0$ and $\epsilon\in\left(\frac1{q+1},1\right)$. Then, for every bounded stopping time $\tau\le T$, there exists $N=N(\epsilon,q,\nu,T)>0$ such that
\[
\E\left[
    \sup_{0<t\le\tau}
    \int_{\bR}w_n(t,x)^\epsilon\,\ud x
\right]
\le
N
\E\left[
    W_q(w_n(0))^\epsilon
\right].
\]
\end{corollary}

\begin{proof}
For every nonnegative measurable function $f$, H\"older's inequality with exponents $1/\epsilon$ and $1/(1-\epsilon)$ gives
\begin{equation}\label{eq:Holder_epsilon}
\int_{\bR}f(x)^\epsilon\,\ud x
=
\int_{\bR}
\left(\Phi_q(x)f(x)\right)^\epsilon
\Phi_q(x)^{-\epsilon}\,\ud x
\le
W_q(f)^\epsilon
\left(
\int_{\bR}
\Phi_q(x)^{-\epsilon/(1-\epsilon)}\,\ud x
\right)^{1-\epsilon},
\end{equation}
and the last integral is finite exactly when $\epsilon>\frac1{q+1}$. Applying this with $f=w_n(t)$, taking the supremum over $t\le\tau$, and using Lemma \ref{lem:mtg_bound_w} with $\alpha=\epsilon$, gives the result.
\end{proof}

\begin{remark}\label{rmk:L_epsilon_bound_v_n}
Lemma \ref{lem:mtg_bound_w} and Corollary \ref{cor:L_epsilon_bound_w_n} also hold with $w_n$ replaced by $v_n$ and, more generally, by any nonnegative solution whose noise coefficient vanishes at $0$ and satisfies the growth bound of Assumption \ref{asp:sigma} -- with the same proofs. This general form will be used again in Proposition \ref{prop:existence_measure_initial}.
\end{remark}

\subsection{Uniform supremum estimates for the split solutions}
\label{subsec:uniform-sup-estimates}

In this subsection we prove supremum estimates in time and space for the two split solutions $v_n$ and $w_n$, uniformly in $n$, which will be used in the tail extinction argument in Section \ref{sec:tail_extinction}. The proof is based on Krylov's $L_p$-estimate applied to the localized processes $v_n\zeta_m$ and $w_n\zeta_m$. The weighted mass estimates from the previous subsection enter only through the control of the lower-order terms $v_n^{p\gamma}$ and $w_n^{p\gamma}$, which arise from the growth condition in Assumption \ref{asp:sigma}(ii):
\begin{equation*}
    \sigma(u)^p \le N_p(u^{p\gamma} + u^p),\qquad u\ge0.
\end{equation*} When $p\gamma\ge1$, these are controlled by the mass and the $L_p$-norm appearing in Krylov's estimate; when $p\gamma<1$, we use the weighted $L_\epsilon$-estimate with $\epsilon=p\gamma$.

\begin{lemma}
\label{lem:uniform_bound}
Let $\tau\le T$ be a bounded stopping time. Recall the parameters $\gamma$ and $p$ from Assumption \ref{asp:sigma} and \eqref{eq:asp_initial_regularity}, respectively. Assume in addition that one of the following conditions holds:
\begin{enumerate}
\item[\rm (a)]
$p\gamma\ge1$ and
\[
    \bE\int_{\bR}u_0(x)\,\ud x<\infty;
\]

\item[\rm (b)]
$p\gamma<1$, and for some $q>0$,
\[
    p\gamma>\frac1{q+1},
    \qquad
    \bE\int_{\bR}(1+x^2)^{q/2}u_0(x)\,\ud x<\infty .
\]
\end{enumerate}
Then there exists a constant $N>0$, independent of $n$, such that
\[
\sup_{n\ge1}
\bE\left[
\sup_{\substack{0\le t\le\tau\\x\in\bR}}v_n(t,x)
+
\sup_{\substack{0\le t\le\tau\\x\in\bR}}w_n(t,x)
\right]
\le N .
\]
\end{lemma}

\begin{proof}
First of all, we note that since $0<p\gamma\le p$, we have, for every $y\ge0$,
\begin{equation}\label{eq:elementary_power}
    y^{p\gamma}\le 1+y^p
    \qquad\text{and}\qquad
    y^{p\gamma}\le y+y^p \quad\text{if }p\gamma\ge1.
\end{equation}
Both inequalities above will be used repeatedly below.

Let us now show the estimate for $v_n$. Fix $m\ge 1$. To estimate the $\cH^\eta_p(\tau)$--norm of $v_n\zeta_m$, consider the SPDE
\begin{equation}\label{eq:localized_v_sup}
\begin{aligned}
\ud Y_m &= \left[  L Y_m
- \left\{ 2a(v_n\zeta_m')_x - a v_n\zeta_m'' + b v_n\zeta_m' \right\} \right]\ud t
+
\sigma(v_n)\zeta_m\,\xi_1(\ud t,\ud x),
\\
Y_m(0,x)
&=
u_0(x)\rho(x/n)\zeta_m(x),
\end{aligned}
\end{equation}
regarded as a linear equation for $Y_m$ with prescribed free terms. First, since $\rho(\cdot/n)$ and $\zeta_m$ are pointwise multipliers on $H_p^{\eta-2/p}$ with bounds independent of $n$ and $m$, Lemma \ref{lem:prop_of_bessel_space}\eqref{pointwise_multiplier} shows that
\[
    \|Y_m(0)\|_{U_p^\eta}^p
    =
    \E\|u_0\rho(\cdot/n)\zeta_m\|_{H_p^{\eta-2/p}}^p
    \le
    N\|u_0\|_{U_p^\eta}^p<\infty.
\]
Next, note that $\zeta_m'/\zeta_m$ and $\zeta_m''/\zeta_m$ are bounded uniformly in $m$ and $x$ (see \eqref{eq:zeta_bounds}). Thus, by Lemma \ref{lem:prop_of_bessel_space}\eqref{pointwise_multiplier} and
\eqref{bounded_operator}, the isometry $\|g\|_{H_p^{\eta-1}}=\|R_{1-\eta}*g\|_{L_p}$,
Young's inequality, and the fact that $R_\alpha\in L_1$ for every $\alpha>0$
(Remark \ref{rem:bessel_kernel}), we have that 
\[
\| 2a(v_n\zeta_m')_x \|_{H_{p}^{\eta-2}}
\le N \| (v_n\zeta_m')_x \|_{H_{p}^{\eta-2}}
\le N \| v_n\zeta_m' \|_{H_{p}^{\eta-1}}
= N \|R_{1-\eta} * (v_n\zeta_m') \|_{L_p}
\le N \|v_n \zeta_m\|_{L_p}.
\]
Similarly, using
$R_{2-\eta}$ in place of $R_{1-\eta}$, we get that 
\[
    \| a v_n\zeta_m'' \|_{H_{p}^{\eta-2}}
    +
    \| b v_n\zeta_m' \|_{H_{p}^{\eta-2}}
    \le
    N\|v_n\zeta_m\|_{L_p}.
\]
We now integrate over $\Omega\times(0,t\wedge\tau]$ and use 
Proposition \ref{prop:existence_vw} for finiteness to see that 
\begin{equation}\label{eq:drift}
\left\|
2a(v_n\zeta_m')_x - a v_n\zeta_m'' + b v_n\zeta_m'
\right\|_{\bH_{p}^{\eta-2}(t\wedge\tau)}^{p}
\le
N\E\int_{0}^{t\wedge\tau}\|v_{n}(s)\zeta_{m}\|_{L_p}^p \,\ud s
<\infty.
\end{equation}
For the noise term, Lemma \ref{lem:white_noise_estimate} and the growth
bound in Assumption \ref{asp:sigma} imply
\begin{equation}\label{eq:diffusion}
\|\sigma(v_{n})\zeta_{m}\mathbf{e}\|_{\bH_{p}^{\eta-1}(t\wedge\tau,\ell_2)}^{p}
\le
N\E\int_{0}^{t\wedge\tau}\int_{\bR}
\left(v_n^{p\gamma}+v_n^p\right)\zeta_m^p
\,\ud x\,\ud s
<\infty,
\end{equation}
where finiteness again follows from Proposition \ref{prop:existence_vw} by \eqref{eq:elementary_power} and $\zeta_m\in L_p(\bR)$. 

By \eqref{eq:drift} and \eqref{eq:diffusion}, Krylov's $L_p$-estimate (Theorem \ref{thm:krylov_lp_estimate}) can be applied to \eqref{eq:localized_v_sup} so that 
there exists a unique solution $Y_m\in\cH_{p}^{\eta}(\tau)$, and
\begin{equation}\label{eq:Y_m}
\| Y_m \|_{\cH_{p}^{\eta}(t\wedge\tau)}^{p}
\le
N\| u_{0} \|_{U_{p}^{\eta}}^{p}
+N\left\|2a(v_n\zeta_m')_x - a v_n\zeta_m'' + b v_n\zeta_m'
\right\|_{\bH_{p}^{\eta-2}(t\wedge\tau)}^{p}
+N\|\sigma(v_{n})\zeta_{m}\mathbf{e}\|_{\bH_{p}^{\eta-1}(t\wedge\tau,\ell_2)}^{p},
\end{equation}
where $N=N(p,\eta,\nu,T)$ is independent of $m$ and $n$. On the other hand, since $v_n$ solves \eqref{eq:spde_v_n}, a direct computation shows that $v_n\zeta_m\in\cH_{p}^{0}(\tau)$ also satisfies \eqref{eq:localized_v_sup} in the sense of Definition \ref{def of cH}. Since $Y_m\in\cH_p^\eta(\tau)\subset\cH_p^0(\tau)$ and the solution to \eqref{eq:localized_v_sup}  in
$\cH_p^0(\tau)$ is unique, we conclude that 
\[
    Y_m(t,x)=v_n(t,x)\zeta_m(x)
    \qquad\text{for all }(t,x)\in[0,\tau]\times\bR
    \text{ almost surely.}
\]

We now refine the  estimate \eqref{eq:diffusion} into a form suitable
for Gronwall's inequality. The second term on the right-hand side of \eqref{eq:diffusion} equals
$N\E\int_0^{t\wedge\tau}\|Y_m(s)\|_{L_p}^p\,\ud s$. Thus,  it remains to control the
term involving $v_n^{p\gamma}$, and here the moment assumption on $u_0$ enters.

If $p\gamma<1$, then $0\le\zeta_m\le1$ and Corollary \ref{cor:L_epsilon_bound_w_n}
together with Remark \ref{rmk:L_epsilon_bound_v_n}, applied with
$\epsilon=p\gamma$ (admissible since $p\gamma>\tfrac1{q+1}$), yields
\[
\E\int_0^{t\wedge\tau}\int_{\bR}v_n^{p\gamma}\zeta_m^p\,\ud x\,\ud s
\le
T\,\E\left[\sup_{0\le r\le\tau}\int_{\bR}v_n(r,x)^{p\gamma}\,\ud x\right]
\le N .
\]
If $p\gamma\ge1$, then $0\le\zeta_m\le1$ and $p\gamma\le p$ give $v_n^{p\gamma}\zeta_m^p\le(v_n\zeta_m)^{p\gamma}$, so that by \eqref{eq:elementary_power},
\[
\E\int_0^{t\wedge\tau}\int_{\bR}v_n^{p\gamma}\zeta_m^p\,\ud x\,\ud s
\le
N\,\E\int_0^{t\wedge\tau}\|Y_m(s)\|_{L_p}^p\,\ud s
+
N\,\E\int_0^{t\wedge\tau}\int_{\bR}v_n\zeta_m\,\ud x\,\ud s .
\]
By the supermartingale property of Lemma \ref{lem:mtg_bound_w} (applied to
$v_n$ with $q=0$, via Remark~\ref{rmk:L_epsilon_bound_v_n}),
\[\E\int_{\bR}v_n(s,x)\zeta_m(x)\,\ud x\le e^{KT}\E\int_{\bR}u_0(x)\,\ud x,\] so
the last term is bounded uniformly in $n$ and $m$. In either case,
\begin{equation}\label{eq:noise_v}
\|\sigma(v_n)\zeta_m\mathbf e\|_{\bH_p^{\eta-1}(t\wedge\tau,\ell_2)}^p
\le
N+N\,\E\int_0^{t\wedge\tau}\|Y_m(s)\|_{L_p}^p\,\ud s .
\end{equation}
Plugging \eqref{eq:drift} and \eqref{eq:noise_v} into \eqref{eq:Y_m}, and using
$Y_m=v_n\zeta_m$ in \eqref{eq:drift}, we obtain
\[
\|Y_m\|_{\cH_p^\eta(t\wedge\tau)}^p
\le
N+N\,\E\int_0^{t\wedge\tau}\|Y_m(s)\|_{L_p}^p\,\ud s .
\]
By Theorem \ref{thm:embedding_theorems}\eqref{embedding in time},
$\E\sup_{0\le r\le s\wedge\tau}\|Y_m(r)\|_{L_p}^p
\le N\|Y_m\|_{\cH_p^\eta(s\wedge\tau)}^p$, which implies that 
\[
\|Y_m\|_{\cH_p^\eta(t\wedge\tau)}^p
\le
N+N\int_0^t\|Y_m\|_{\cH_p^\eta(s\wedge\tau)}^p\,\ud s. 
\]
We now use  Gronwall's inequality to get that 
\begin{equation}\label{eq:Y_m_v_uniform_bound}
\|Y_m\|_{\cH_p^\eta(\tau)}^p\le N,
\end{equation}
with $N$ independent of $n$ and $m$.

Since $\eta p>3$, we may choose $\tfrac1p<\ba<\bb<\tfrac12\left(\eta-\tfrac1p\right)$, and
Theorem~\ref{thm:embedding_theorems}\eqref{embedding in time and space} gives
\[
\E\|Y_m\|_{C^{\ba-1/p}([0,\tau];C^{\eta-2\bb-1/p}(\bR))}^p
\le N\|Y_m\|_{\cH_p^\eta(\tau)}^p\le N .
\]
In particular,
$\E\left[\sup_{0\le t\le\tau,\,x\in\bR}|Y_m(t,x)|^p\right]\le N$, so by Jensen's inequality
\[
    \E\left[\sup_{0\le t\le\tau,\,x\in\bR}v_n(t,x)\zeta_m(x)\right]\le N .
\]
Since $v_n\ge0$ and $\zeta_m\uparrow1$ as $m\to\infty$, monotone convergence yields
\begin{equation}\label{eq:uniform_v_bound}
\sup_{n\ge1}\E\left[\sup_{0\le t\le\tau,\,x\in\bR}v_n(t,x)\right]\le N .
\end{equation}

We now consider the estimate for $w_n$.  The argument parallels that for the estimate for $v_n$, so we
only indicate the differences. Set $\tilde Y_m:=w_n\zeta_m$. Since
$D_n:=\sigma^2(v_n+w_n)-\sigma^2(v_n)\ge0$ by
Assumption \ref{asp:sigma}(iii), $\tilde Y_m$ satisfies the analogue of
\eqref{eq:localized_v_sup} with $\sigma(v_n)\xi_1$ replaced by
$\sqrt{D_n}\,\xi_2$ and initial function
$\tilde Y_m(0,\cdot)=u_0(1-\rho(\cdot/n))\zeta_m$. The initial function and
the drift correction are estimated exactly as before, i.e., 
\[
\|\tilde Y_m(0)\|_{U_p^\eta}^p\le N\|u_0\|_{U_p^\eta}^p,
\qquad
\left\|2a(w_n\zeta_m')_x-aw_n\zeta_m''+bw_n\zeta_m'
\right\|_{\bH_p^{\eta-2}(t\wedge\tau)}^p
\le N\,\E\int_0^{t\wedge\tau}\|\tilde Y_m(s)\|_{L_p}^p\,\ud s.
\]
For the noise term, Lemma \ref{lem:white_noise_estimate} gives
\begin{equation}\label{eq:w_noise_term}
\|\sqrt{D_n}\,\zeta_m\mathbf e\|_{\bH_p^{\eta-1}(t\wedge\tau,\ell_2)}^p
\le N\,\E\int_0^{t\wedge\tau}\!\!\int_{\bR}D_n^{p/2}\zeta_m^p\,\ud x\,\ud s .
\end{equation}
Since $0\le D_n\le\sigma^2(v_n+w_n)$, the growth bound in
Assumption \ref{asp:sigma}(ii) yields
\begin{equation}\label{eq:D_n_split}
D_n^{p/2}\zeta_m^p
\le N(v_n\zeta_m)^p+N(w_n\zeta_m)^p+N(v_n+w_n)^{p\gamma}\zeta_m^p .
\end{equation}
In particular, the right-hand side of \eqref{eq:w_noise_term} is finite by
\eqref{eq:elementary_power} and Proposition \ref{prop:existence_vw}, so Krylov's estimate (Theorem \ref{thm:krylov_lp_estimate}) applies and identifies the unique solution with $w_n\zeta_m$.

It remains to bound the three terms in \eqref{eq:D_n_split}. 
The first is uniformly bounded after integration by \eqref{eq:Y_m_v_uniform_bound}, 
and the second contributes $N\,\E\int_0^{t\wedge\tau}\|\widetilde Y_m(s)\|_{L_p}^p\,\ud s$. The
last one is treated exactly as the term $v_n^{p\gamma}\zeta_m^p$ above, using
the subadditivity $(v_n+w_n)^{p\gamma}\le v_n^{p\gamma}+w_n^{p\gamma}$ when
$p\gamma<1$ and \eqref{eq:elementary_power} when $p\gamma\ge1$ The terms
involving $v_n$ are controlled by \eqref{eq:Y_m_v_uniform_bound} and
Remark \ref{rmk:L_epsilon_bound_v_n}. Hence
\[
\|\sqrt{D_n}\,\zeta_m\mathbf e\|_{\bH_p^{\eta-1}(t\wedge\tau,\ell_2)}^p
\le N+N\,\E\int_0^{t\wedge\tau}\|\widetilde Y_m(s)\|_{L_p}^p\,\ud s,
\]
and the proof is completed exactly as for $v_n$.
\end{proof}

%%%%%%%%%%%%%%%%%%%%%%%%%%%%%%%%%%%%%%%%%%%%%%%%%%%%%%

\section{Tail extinction and proof of Theorem \ref{thm:CSP_regular_initial}}
\label{sec:tail_extinction}

In Section \ref{subsec:splitting_construction} we introduced the splitting
\[
    u_0=u_0\rho(\cdot/n)+u_0(1-\rho(\cdot/n))
\]
and constructed the corresponding processes $v_n$ and $w_n$. The process $v_n$ starts from the compactly supported initial datum $u_0\rho(\cdot/n)$, whereas $w_n$ starts from the tail of $u_0$, i.e., $u_0(1-\rho(\cdot/n))$. The purpose of this section is to show that $w_n$ becomes extinct before any fixed positive time with probability tending to one as $n\to\infty$. This tail extinction estimate is the main probabilistic input in the proof of Theorem \ref{thm:CSP_regular_initial}.

In addition to the assumptions made at the beginning of Section \ref{sec:splitting_estimates}, we assume throughout this section that $u_0$ satisfies the moment condition (i) or (ii) of Theorem \ref{thm:CSP_regular_initial}. We continue to work on the filtered probability space of Proposition \ref{prop:existence_vw} carrying $(v_n,w_n)_{n\ge1}$, without distinguishing the copy of $(u_0,a,b,c)$ notationally. In particular, Lemma \ref{lem:uniform_bound} and Corollary \ref{cor:L_epsilon_bound_w_n} apply to $v_n$ and $w_n$.

\subsection{Extinction of the tail}
We first provide a one-dimensional extinction criterion for nonnegative continuous semimartingales, which is purely probabilistic and does not involve the hypotheses above. 
The following lemma is an analog of \cite[Proposition A.4]{KKMS23}.

\begin{lemma}\label{lem:extinction}
    Let $\alpha \in (0, 2)$. Suppose $X=(X_t)_{t\ge0}$ is a nonnegative continuous semimartingale
of the form $\ud X_t=C_t\,\ud t+\ud M_t$, where $(C_t)_{t\ge0}$ is a predictable process with $C_t\le\kappa X_t$ for all $t\ge0$ for some
constant $\kappa>0$, and $(M_t)_{t\ge0}$ is a continuous local
martingale with $M_0=0$. Then, we have that for every $A>0$, $t>0$, and $D>0$,
\[ 
\P\left( \inf_{0\leq s\leq t} X_s >0\,, X_{0} < D, \,  \int_0^t \frac{e^{-\kappa(2-\alpha) s }}{X_s^\alpha} \, \ud\langle M \rangle_s \geq A \  \right) \leq \frac{2D^{1-\frac{\alpha}{2}}}{\sqrt{A\left(1-\frac{\alpha}{2}\right)^2}}.
\] 
\end{lemma}

\begin{proof}
    Define 
    \[ \tau:=\inf\{t\geq 0: X_t =0\} \quad \text{and}\quad  Y_t:=e^{-\kappa t} X_t.\]
  By It\^o's formula, we have that for $t< \tau$,
\[ \ud Y_t =e^{-\kappa t} \left( C_t -\kappa X_t\right) \ud t + e^{-\kappa t } \ud M_t.\] 
Letting $\beta:=1-\alpha/2<1$ and applying It\^o's formula to $Y_t^\beta$, we have 
\[ Y_t^\beta \leq Y_0^\beta+ N_t,\]
where $(N_t)_{t\geq 0}$ is a continuous local martingale defined as 
\[ N_t:= \beta \int_0^t Y_s^{\beta-1} e^{-\kappa s} \ud M_s,\quad \text{for $t< \tau$}.\]
Thus, we have that for $t< \tau$
\[ \langle N \rangle_t = \beta^2 \int_0^t  Y_s^{2\beta-2} e^{-2\kappa s} \ud\langle M\rangle_s =  \beta^2   \int_0^t\frac{e^{-2\beta \kappa s}}{ X_s^\alpha} \,  \ud\langle M \rangle_s. \] 
By the Dubins-Schwarz theorem and the reflection principle, there exists a standard Brownian motion $B=(B_t)_{t\geq 0}$ (possibly on some enlarged probability space) such that  
\begin{align*}
 \P\left( \inf_{0\leq s\leq t} X_s >0\,, X_{0}<D\,, \,  \int_0^t\frac{e^{-2\beta \kappa s}}{ X_s^\alpha} \,  \ud\langle M \rangle_s \geq A \  \right) 
 & \leq \P\left( \inf_{0\leq s\leq t} N_s > -D^\beta\,, \langle N \rangle_t \geq \beta^2 A \right) \\
 & \leq \P\left( \inf_{0\leq s\leq t} B_{\langle N\rangle_s} > - D^\beta\,, \langle N \rangle_t \geq \beta^2 A \right)\\
 & \leq \P\left( \inf_{0\leq s\leq \beta^2 A} B_s > -D^\beta \right) \\
 &\leq \frac{2D^\beta}{\sqrt{A\beta^2}}, 
\end{align*}
which completes the proof. 
\end{proof}

We now show below that $w_n$ becomes extinct before any fixed positive time with probability tending to one as $n\to\infty$. \begin{lemma}
\label{lem:tail_extinction}
For every $t_*\in(0,T]$,
\[
    \lim_{n\to\infty}
    \bP\left(
        \inf_{0\le t\le t_*}
        \int_{\bR}w_n(t,x)\,\ud x>0
    \right)=0.
\]
\end{lemma}

\begin{proof}
Set
\[
    M_n(t):=\int_{\bR}w_n(t,x)\,\ud x
\]
and
\[
    D_n(t,x):=
   \sigma^2(w_n(t,x)+v_n(t,x))-\sigma^2(v_n(t,x)).
\]
By Lemma \ref{lem:mtg_bound_w}, $M_n(t)<\infty$ for every
$t\in[0,T]$ almost surely, and there exists $K_0=K_0(\nu)>0$ such that
\[
    \ud M_n(t)\le K_0M_n(t)\,\ud t+\ud N_n(t),
\]
where $N_n(t):= \int_0^t\int_{\bR}\sqrt{D_n(s,x)}\,\xi_2(\ud s,\ud x)$ is a continuous local martingale with quadratic variation
\[
    \langle N_n\rangle_t
    =
    \int_0^t\int_{\bR}D_n(s,x)\,\ud x\,\ud s.
\]
Fix $t_*\in(0,T]$ and split the proof into the cases $\lambda=1$ and $\lambda\in(1,2)$.

\medskip
\noindent
\textbf{Case 1: $\lambda=1$.}
For $K\ge1$, define the event 
\begin{equation}\label{eq:A_K}
    \fA_K^{(n)}
    :=
    \left\{
        \sup_{\substack{0\le s\le t_*\\x\in\bR}}
        v_n(s,x)\le K
    \right\}
    \cap
    \left\{
        \sup_{\substack{0\le s\le t_*\\x\in\bR}}
        w_n(s,x)\le K
    \right\}.
\end{equation}
On $\fA_K^{(n)}$, we have $0\le v_n\le v_n+w_n\le2K$. Therefore, by
Assumption \ref{asp:sigma}(iii), there exists $c_K>0$ such that
\[
    D_n(s,x)
    \ge
    c_K w_n(s,x),
    \qquad 0\le s\le t_*.
\]
Consequently, on $\fA_K^{(n)}$,
\[
    \frac{\ud}{\ud s}\langle N_n\rangle_s
    =
    \int_{\bR} D_n (s,x)\,\ud x
    \ge
    c_KM_n(s).
\]
Hence, on the event
\[
    \left\{\inf_{0\le s\le t_*}M_n(s)>0\right\}\cap \fA_K^{(n)},
\]
we have
\[
\begin{aligned}
    \int_0^{t_*}
    \frac{e^{-K_0s}}{M_n(s)}\,\ud\langle N_n\rangle_s
    &\ge
    c_K\int_0^{t_*}e^{-K_0s}\,\ud s
    =: A_{K,t_*}>0 .
\end{aligned}
\]
Therefore, for every $D>0$, Lemma \ref{lem:extinction} with $\alpha=1$ gives
\[
\begin{aligned}
&\bP\left(
    \left\{\inf_{0\le s\le t_*}M_n(s)>0\right\}
    \cap \fA_K^{(n)}
\right)                                                   \\
&\quad\le
\bP\left(
    \inf_{0\le s\le t_*}M_n(s)>0,\,
    M_n(0)<D,\,
    \int_0^{t_*}
    \frac{e^{-K_0s}}{M_n(s)}\,\ud\langle N_n\rangle_s
    \ge A_{K,t_*}
\right)
+
\bP(M_n(0)\ge D)                                    \\
&\quad\le
\frac{4D^{1/2}}{\sqrt{A_{K,t_*}}}
+
\frac{\E[M_n(0)]}{D}.
\end{aligned}
\]
On the other hand, Lemma \ref{lem:uniform_bound} shows that there exists $N>0$, independent of $n$, such that 
\[
    \bP\bigl((\fA_K^{(n)})^c\bigr)
    \le
    \frac{N}{K}.
\]
Thus, we have that 
\[
\begin{aligned}
\bP\left(
    \inf_{0\le s\le t_*}M_n(s)>0
\right)
&\le
\frac{4 D^{1/2}}{\sqrt{A_{K,t_*}}}
+
\frac{\E[M_n(0)]}{D}
+
\frac{N}{K}.
\end{aligned}
\]
Since $M_n(0)=\int_{\bR}u_0(x)(1-\rho(x/n))\,\ud x \longrightarrow0$ in $L_1(\Omega)$, we first let $n\to\infty$, then $D\downarrow0$, and finally
$K\to\infty$. This proves the claim in the case $\lambda=1$.

\noindent
\textbf{Case 2: $\lambda\in(1,2)$.}
By assumption we may choose $q>\frac{\lambda-1}{2-\lambda}$, so that
$1<\lambda+\frac{\lambda-1}{q}<2$ and we may fix
\[
    \alpha\in\left(\lambda+\frac{\lambda-1}{q},\,2\right)
    \qquad\text{and set}\qquad
    r:=\frac{\alpha-\lambda}{\alpha-1} .
\]
Then $\alpha\in(1,2)$ and $r\in\bigl(\frac1{q+1},1\bigr)$, the latter being
exactly the range in which Corollary \ref{cor:L_epsilon_bound_w_n} applies.

Since the interpolation requires control of $\int_\bR w_n^r\,\ud x$ as well, we
keep the event $\fA^{(n)}_K$ of \eqref{eq:A_K} and define
\[
    \fB_K^{(n)}
    :=
    \left\{
        \sup_{0\le s\le t_*}
        \int_{\bR}w_n(s,x)^r\,\ud x
        \le K
    \right\}.
\]
On $\fA_K^{(n)}$, Assumption \ref{asp:sigma}(iii) gives
\[
    D_n (s,x)\ge c_K w_n(s,x)^\lambda.
\]
Moreover, by H\"older's inequality,
\[
\begin{aligned}
    M_n(s)
    =
    \int_{\bR}w_n(s,x)\,\ud x
    &\le
    \left(
        \int_{\bR}w_n(s,x)^\lambda\,\ud x
    \right)^{1/\alpha}
    \left(
        \int_{\bR}w_n(s,x)^r\,\ud x
    \right)^{1-1/\alpha}.
\end{aligned}
\]
On $\fB_K^{(n)}$, this implies that 
\[
    \int_{\bR}w_n(s,x)^\lambda\,\ud x
    \ge
    K^{1-\alpha}M_n(s)^\alpha .
\]
Hence, on $\fA_K^{(n)}\cap \fB_K^{(n)}$, 
\[
    \frac{\ud}{\ud s}\langle N_n\rangle_s
    =
    \int_{\bR}D_n (s,x)\,\ud x
    \ge
    c_KK^{1-\alpha}M_n(s)^\alpha .
\]
Consequently, on the event $\left\{\inf_{0\le s\le t_*}M_n(s)>0\right\} \cap \fA_K^{(n)}\cap \fB_K^{(n)}$, 
we have
\[
    \int_0^{t_*}
    \frac{e^{-K_0(2-\alpha)s}}{M_n(s)^\alpha}\,
    \ud\langle N_n\rangle_s
    \ge
    c_KK^{1-\alpha}
    \int_0^{t_*}e^{-K_0(2-\alpha)s}\,\ud s=: A_{K,t_*}^{(\alpha)}>0.
\]
Using Lemma \ref{lem:extinction}, we get that  for every $D>0$,
\[
\bP\left(\left\{\inf_{0\le s\le t_*}M_n(s)>0\right\}\cap \fA_K^{(n)}\cap \fB_K^{(n)}\right)                                                   
\quad\le
\frac{2D^{1-\alpha/2}}
{\sqrt{A_{K,t_*}^{(\alpha)}}\left(1-\frac{\alpha}{2}\right)}
+ \frac{\E[M_n(0)]}{D}.
\]
By Lemma \ref{lem:uniform_bound}, we have 
\[
    \bP\bigl((\fA_K^{(n)})^c\bigr)
    \le
    \frac{N}{K}.
\]
In addition, since $W_q(w_n(0))\le W_q(u_0)$, Corollary \ref{cor:L_epsilon_bound_w_n} implies that 
\[
    \bP\bigl((\fB_K^{(n)})^c\bigr)
    \le
    \frac{1}{K}
    \E\left[
        \sup_{0\le s\le t_*}
        \int_{\bR}w_n(s,x)^r\,\ud x
    \right]
    \le
    \frac{N}{K},
\]
where $N$ is independent of $n$. Therefore,
\[
\begin{aligned}
\bP\left(
    \inf_{0\le s\le t_*}M_n(s)>0
\right)
&\le
\frac{2D^{1-\alpha/2}}
{\sqrt{A_{K,t_*}^{(\alpha)}}\left(1-\frac{\alpha}{2}\right)}
+
\frac{\E[M_n(0)]}{D}
+
\frac{N}{K}.
\end{aligned}
\]
Since $M_n(0)= \int_{\bR}u_0(x)(1-\rho(x/n))\,\ud x \longrightarrow0$ in $L_1(\Omega)$, we first let $n\to\infty$, then $D\downarrow0$, and finally
$K\to\infty$. This proves the claim in the case $\lambda\in(1,2)$, which completes the proof. 
\end{proof}

\subsection{Proof of Theorem \ref{thm:CSP_regular_initial}} \label{subsec:CSP_proof_regular}

Now we are ready to prove Theorem \ref{thm:CSP_regular_initial}.

\begin{proof}[Proof of Theorem \ref{thm:CSP_regular_initial}]
Let $u$ be an arbitrary nonnegative solution of \eqref{eq:SPDE} with
initial function $u_0$, on any filtered probability space as in
Section \ref{sec:main_results}. For $n\ge1$, let
$y_n:=v_n+w_n$ be the process from Lemma~\ref{lem:summing}, which is
a nonnegative solution of \eqref{eq:SPDE} on the filtered
probability space of Proposition~\ref{prop:existence_vw}, driven by
the space-time white noise $\xi^{(n)}$, with initial function
$\bar u_0$ and coefficient field $(\bar a,\bar b,\bar c)$. Since
$\mathcal L(\bar u_0,\bar a,\bar b,\bar c)
=\mathcal L(u_0,a,b,c)$, the uniqueness-in-law assumption of
Theorem~\ref{thm:CSP_regular_initial} yields
\begin{equation}\label{eq:law_identify}
    \mathcal L(u)=\mathcal L(y_n)
    \quad\text{on } C([0,T];C_{\rm tem})
    \qquad\text{for every } n\ge1.
\end{equation}
For $m\ge 1$ and $R\ge1$, define the path event
\[
    \fC_{m,R}:=\{ f\in C([0,T] ; \Ctem)\, : \, f(t,x) = 0 \text{ for all $t\in[T/(2m),T]$ and all $x$ such that $|x|\ge R$} \}.
\] 
By continuity, $\fC_{m,R}$ is measurable since it can be written as a countable intersection of measurable sets.

Our goal is to show 
\[
     \P\left(u\in\bigcap_{m=1}^\infty\bigcup_{R=1}^\infty\fC_{m,R}\right)=1, \qquad\text{equivalently}\qquad \lim_{m \to \infty } \P\left(u\in\bigcup_{R=1}^\infty\fC_{m,R}\right)=1,
\] 
where the equivalence follows from the continuity of probability along the
decreasing intersection over $m\ge 1$. By \eqref{eq:law_identify}, for
every $n\ge1$,
\[ 
    \P\left(u\in\bigcup_{R=1}^\infty\fC_{m,R}\right)= \P\left(y_n\in\bigcup_{R=1}^\infty\fC_{m,R}\right). 
\]
Hence it suffices to show that, for every $m\ge 1$ 
\begin{equation}\label{eq:CSP_goal}
    \lim_{n\to\infty}
    \P\left(y_n\in\bigcup_{R=1}^\infty\fC_{m,R}\right)=1.
\end{equation}

We first consider  the $v_n$-part. Since $v_n(0)=u_0\rho(\cdot/n)$ has compact
support and satisfies the initial-data condition of
\cite[Assumption 2.6]{han2023compact} by the Sobolev--H\"older embedding
(Lemma~\ref{lem:prop_of_bessel_space}\eqref{sobolev-embedding}), applying
\cite[Theorem 2.11]{han2023compact} to $v_n$ gives that for every $n,m\ge 1$,
\begin{equation}\label{eq:V_n_full}
    % \P\bigl(v_n(t,x)= 0 \text{ for every $t\in[0,T]$ and $x\in \R$ such that $|x|>R$ for some $R\in \N$}\bigr)=1.
     \P\left(v_n\in\bigcup_{R=1}^\infty\fC_{m,R}\right)=1.
\end{equation} 

Next, set $M_n(t):=\int_{\bR}w_n(t,x)\,\ud x$, which is finite for all
$t\in[0,T]$ almost surely, and recall from Lemma \ref{lem:mtg_bound_w} that
$\bigl(e^{-Kt}M_n(t)\bigr)_{t\in[0,T]}$ is a nonnegative continuous
supermartingale. We claim that 
\begin{equation}\label{eq:extinction_absorbing}
   \fD_{n,m}:=\left\{\inf_{0\le t\le 1/m}M_n(t)=0\right\}
    \subseteq
    \left\{w_n(t)\equiv0\text{ for every }t\in[1/m,T]\right\}
    \quad\text{up to a null set}.
\end{equation}
Indeed, let $\theta_n:=\inf\{t \geq 0:M_n(t)=0\}\wedge T/(2m)$ be a stopping time. Then, the continuity of $M_n$ yields  $M_n(\theta_n)=0$ on $\fD_{n,m}$. Moreover, optional sampling for the nonnegative supermartingale $e^{-Kt}M_n(t)$ gives, for
$t\ge\theta_n$,
\[
    \E\bigl[e^{-Kt}M_n(t)\mid\mathcal F_{\theta_n}\bigr]
    \le
    e^{-K\theta_n}M_n(\theta_n)=0, \qquad \text{on $\fD_{n,m}$.}
\]
Since $\theta_n\le T/(2m)$, this gives, for each fixed $t\in[T/(2m),T]$, that
$M_n(t)=0$ almost surely on $D_{n,m}$. Using the continuity of $t\mapsto M_n(t)$, we conclude that almost surely on $D_{n,m}$ we have $M_n(t)=0$ for all $t\in[T/(2m),T]$. As
$w_n\ge0$ is continuous, $M_n(t)=0$ forces $w_n(t,\cdot)\equiv0$.

On the intersection of the events in \eqref{eq:V_n_full} and
\eqref{eq:extinction_absorbing}, we have $y_n(t)=v_n(t)$ for every
$t\in[T/(2m),T]$, so that $y_n(t)$ has compact support for all such $t$.
Consequently, for every $m\ge 1$
\[
    \P\left(\inf_{0\le t\le T/(2m)}M_n(t)=0\right) \le 
    \P\left(y_n\in\bigcup_{R=1}^\infty\fC_{m,R}\right),
\]
and the left-hand side tends to $1$ as $n\to\infty$ by
Lemma \ref{lem:tail_extinction}. This proves \eqref{eq:CSP_goal} and completes
the proof.
\end{proof}

%%%%%%%%%%%%%%%%%%%%%%%%%%%%%%%%%%%%%%%%%%%%%%%%%%%%%%%%%%%%%%%%%%%%%%%%%%%%%%%%%%%%%%%%%%%%%%%%

\section{Construction from measure-valued initial data}
\label{sec:construction-measure}

The goal of this section is to construct a nonnegative solution starting from measure-valued initial data. Throughout this section, Assumptions~\ref{asp:sigma} and
\ref{asp:coefficients} are in force. We let $\mu_0$ be a
nonnegative $\mathcal F_0$-measurable random Radon measure on $\bR$, independent of the driving white noise, such that
\[
    \mu_0\in
    L_2\bigl(\Omega;H_2^{-1/2-\kappa}(\bR)\bigr)
    \qquad
    \text{for every }\kappa\in(0,1/2).
\]
The additional weighted-moment assumptions from
Theorem \ref{thm:CSP_measure_initial} will only be used later in Section \ref{sec:proof_thm_measure}.

We first approximate both the coefficient $\sigma$ and the initial measure
$\mu_0$, and obtain a family of regularized solutions $(u_n)_{n\ge1}$.
We then prove uniform localized estimates, establish tightness, and
identify every subsequential limit as a solution of
\eqref{eq:SPDE} with initial measure $\mu_0$.

\subsection{Approximation and uniform estimates}
\label{subsec:approximation-uniform-estimates}
We begin with the approximation scheme. Let
$\psi\in C_c^\infty(\bR)$ be symmetric and satisfy
\[
    0\le \psi\le1,\qquad
    \psi(z)=1\text{ for }|z|\le1,\qquad
    \psi(z)=0\text{ for }|z|\ge2.
\]
Set
\[
    \psi_n(z):=\psi(z/n).
\]
Let $\Theta\in C_c^\infty(\bR)$ be a nonnegative mollifier such that
\[
    \int_{\bR}\Theta(z)\,\ud z=1,
    \qquad
    \operatorname{supp}\Theta\subset[0,1].
\]
For $n\ge1$, define
\begin{equation}
\label{def:sigma-n}
    \sigma_n(u)
    :=
    n\int_{\bR}\sigma(z)\Theta(n(u-z))\,\ud z\,\psi_n(u)
    =
    \int_0^1\sigma(u-z/n)\Theta(z)\,\ud z\,\psi_n(u).
\end{equation}
By Assumption \ref{asp:sigma}, we have
\begin{equation}
\label{eq:sigma-n-properties}
    \sigma_n(0)=0,\qquad
    |\sigma_n(u)|\le N(u^\gamma+u),\qquad u\ge0,
\end{equation}
where $N$ is independent of $n$. Moreover, $\sigma_n$ is Lipschitz continuous, with
a Lipschitz constant depending on $n$, and
\[
    \sigma_n\to\sigma
    \qquad\text{uniformly on compact subsets of }[0,\infty).
\]
We approximate the initial measure by
\[
    u_{n,0}(x)
    :=
    n\int_{\bR}\Theta(n(x-y))\,\mu_0(\ud y),
    \qquad x\in\bR.
\]Note that the initial datum $u_{n,0}$ is a nonnegative
smooth function. 

Consider the regularized equation
\begin{equation}
\label{eq:regularized-measure-equation}
\begin{aligned}
    \partial_t u_n(t,x)
    &=
    L u_n(t,x)+\sigma_n(u_n(t,x))\xi(t,x),
    \qquad (t,x)\in(0,T]\times\bR,\\
    u_n(0,\cdot)
    &=
    u_{n,0}.
\end{aligned}
\end{equation}

We shall use the exponential weight $\zeta$ and $\zeta_m$ introduced in \eqref{eq:zeta}. Recall that \eqref{eq:zeta_bounds} holds uniformly in $m$. 

\begin{lemma}
\label{lem:regularized_solutions}
For each $n\ge1$, equation \eqref{eq:regularized-measure-equation} has a unique
nonnegative solution $u_n$. Moreover, for every bounded stopping time
$\tau\le T$, every $p\ge2$, and every $\kappa\in(0,1/2)$,
\[
    u_n\in\mathcal H_p^{1/2-\kappa}(\tau).
\]
In addition, if $(1-2\kappa)p>6$ and
\[
    \frac1p<\ba<\bb<
    \frac14-\frac{\kappa}{2}-\frac1{2p},
\]
then
\[
    u_n\in
    C^{\ba -1/p}
    \bigl([0,\tau];C^{1/2-\kappa-2\bb-1/p}(\bR)\bigr)
    \qquad\text{a.s.}
\]
Furthermore, for every $\rho>0$, there exists $N_n= N(\nu,\kappa,p, \ba, \bb, \bc_1 ,T, n)$ such that 
\begin{equation}
\label{eq:regularized_holder_weighted}
    \E
    \left\|
        u_n\zeta_{\rho/p}
    \right\|_{C^{\ba -1/p}
    ([0,\tau];C^{1/2-\kappa-2\bb -1/p})}^p
    \le
    N_n+
    N_n\|\mu_0\zeta_{\rho/p}\|_{U_p^{1/2-\kappa}}^p.
\end{equation} Consequently,
\begin{equation}
\label{eq:regularized_tempered_bound}
    \E
    \left[
        \sup_{t\le\tau,\ x\in\bR}
        u_n(t,x)^p e^{-\rho|x|}
    \right]
    \le N_n.
\end{equation}
\end{lemma}

\begin{proof}
This follows from the standard existence theorem for SPDEs with Lipschitz noise
coefficient and smooth initial datum (see e.g. \cite[Lemma A.7]{han2023compact}). 
\end{proof}

The next estimate is a localized Sobolev bound used for compactness.

\begin{lemma}
\label{lem:uniform_localized_sobolev}
Let $\kappa\in(0,1/2)$ and set $s:=\frac12 -\kappa$. Assume that
\[
    \mu_0\in U_{2}^{s} = L_2\bigl(\Omega;H_2^{-1/2-\kappa}(\bR)\bigr).
\]
Then, for every $m\ge1$,
\begin{equation}\label{eq:localized_H2_estimate}
    \sup_{n\ge1}
    \|u_n\zeta_m\|_{\cH_2^s(T)}^2
    \le
    N\|\mu_0\|_{U_2^s}^2
    +
    N\int_{\bR}\zeta_m(x)^2\,\ud x,
\end{equation}
where $N=N(\nu,\kappa,\bc_1,T)$ is independent of $n$ and $m$. Moreover, if
$\eta\in\bR$ and $\ep\in(0,s)$ satisfy $\eta\le s-\ep$, then for every
$\delta\in(0,T)$,
\begin{equation}\label{eq:localized_small_time_estimate}
    \sup_{n\ge1}
    \|u_n\zeta_m\|_{\bH_2^\eta(\delta)}^2
    \le
    N\delta^{\ep}
    \left(
        \|\mu_0\|_{U_2^s}^2
        +
        \int_{\bR}\zeta_m(x)^2\,\ud x
    \right).
\end{equation}
\end{lemma}

\begin{proof} 
Fix $m\ge1$ and set $Y_n^m:=u_n\zeta_m$. Since $u_n\in\cH_2^{s}(T)$ by
Lemma \ref{lem:regularized_solutions} and $\zeta_m$ is a pointwise multiplier
on the relevant spaces (Lemma \ref{lem:prop_of_bessel_space}\eqref{pointwise_multiplier}), we have
$Y_n^m\in\cH_2^s(T)$, and, in the sense of distributions,
\begin{equation}\label{eq:localized_measure_eq}
\begin{aligned}
\ud Y_n^m
&=
\bigl[L Y_n^m
-\{2a(u_n\zeta_m')_x - a u_n\zeta_m'' + b u_n\zeta_m'\}\bigr]\,\ud t
+
\sigma_n(u_n)\zeta_m\,\xi(\ud t,\ud x),
\\
Y_n^m(0,\cdot)
&=
u_{n,0}\zeta_m .
\end{aligned}
\end{equation}
We estimate the free terms in \eqref{eq:localized_measure_eq} in order to apply Krylov's $L_p$-estimate (see Theorem \ref{thm:krylov_lp_estimate}). 

Let us first consider the initial datum. Lemma \ref{lem:prop_of_bessel_space}\eqref{pointwise_multiplier} and \eqref{convolution} imply that 
\begin{equation}\label{eq:initial_measure_localized_bound}
    \sup_{n\ge1}
    \|u_{n,0}\zeta_m\|_{U_2^s}^2
    \le
    N\|\mu_0\|_{U_2^s}^2 .
\end{equation} 

Next,  thanks to \eqref{eq:zeta_bounds}, Lemma \ref{lem:prop_of_bessel_space}\eqref{pointwise_multiplier}, \eqref{bounded_operator}, and \eqref{norm_bounded} (using $s-1<0$) imply that 
\[
    \|2a(u_n\zeta_m')_x\|_{H_2^{s-2}}
    \le
    N\|(u_n\zeta_m')_x\|_{H_2^{s-2}}
    \le
    N\|u_n\zeta_m'\|_{H_2^{s-1}}
    \le
    N\|u_n\zeta_m'\|_{L_2}
    \le
    N\|Y_n^m\|_{L_2}.
\]
 The terms $au_n\zeta_m''$
and $bu_n\zeta_m'$ are handled in the same way. Hence, for every $t\le T$,
\begin{equation}\label{eq:localized_drift_measure_bound}
    \bigl\|
    2a(u_n\zeta_m')_x - a u_n\zeta_m'' + b u_n\zeta_m'
    \bigr\|_{\bH_2^{s-2}(t)}^2
    \le
    N\|Y_n^m\|_{\bL_2(t)}^2  <\infty.
\end{equation}

For the noise term, by \eqref{eq:sigma-n-properties}, $|\sigma_n(r)|^2\le N(r^{2\gamma}+r^2)\le N(1+r^2)$ for $r\ge0$. Hence, Lemma \ref{lem:white_noise_estimate} yields
\begin{equation}\label{eq:localized_noise_measure_bound}
    \|\sigma_n(u_n)\zeta_m\mathbf e\|_{\bH_2^{s-1}(t,\ell_2)}^2
    \le
    N\E\int_0^t\int_{\bR}
    (1+u_n^2)\zeta_m^2\,\ud x\,\ud r
    \le
    N\int_{\bR}\zeta_m(x)^2\,\ud x
    +
    N\|Y_n^m\|_{\bL_2(t)}^2<\infty.
\end{equation}
Thus, applying Krylov's $L_2$-estimate (Theorem \ref{thm:krylov_lp_estimate}) to 
\eqref{eq:localized_measure_eq}, and using
\eqref{eq:initial_measure_localized_bound}--\eqref{eq:localized_noise_measure_bound}, we obtain
\begin{equation}\label{eq:pre_gronwall_localized_measure}
    \|Y_n^m\|_{\cH_2^s(t)}^2
    \le
    N A_m
    +
    N\|Y_n^m\|_{\bL_2(t)}^2,
    \qquad 
    A_m
    :=
    \|\mu_0\|_{U_2^s}^2
    +
    \int_{\bR}\zeta_m(x)^2\,\ud x.
\end{equation}
It remains to control the last term. By
Lemma \ref{lem:prop_of_bessel_space}\eqref{multi_ineq}, for every $\delta_0>0$, 
\[
    \|f\|_{L_2}^2
    \le
    \delta_0\|f\|_{H_2^s}^2
    +
    N_{\delta_0}\|f\|_{H_2^{s-1}}^2,
\]
and by Theorem \ref{thm:embedding_theorems}\eqref{embedding in time 2},
\[
    \E\sup_{0\le r\le t}\|Y_n^m(r)\|_{H_2^{s-1}}^2
    \le
    N\|Y_n^m\|_{\cH_2^s(t)}^2.
\]
Combining these two estimates,
\[
    \|Y_n^m\|_{\bL_2(t)}^2
    \le
    \delta_0\|Y_n^m\|_{\bH_2^s(t)}^2
    +
    N_{\delta_0}\int_0^t\|Y_n^m\|_{\cH_2^s(r)}^2\,\ud r.
\]
Choosing $\delta_0$ sufficiently small, we have 
\[
    \|Y_n^m\|_{\cH_2^s(t)}^2
    \le
    N A_m
    +
    N\int_0^t\|Y_n^m\|_{\cH_2^s(r)}^2\,\ud r,
\]
and \eqref{eq:localized_H2_estimate} follows from Gronwall's inequality.

We now prove the small-time estimate. By \eqref{eq:localized_H2_estimate} and Theorem \ref{thm:embedding_theorems}\eqref{embedding in time 2},
\begin{equation}\label{eq:time_bound}
    \sup_{n\ge1}
    \E\sup_{0\le r\le T}\|Y_n^m(r)\|_{H_2^{s-1}}^2
    \le
    N A_m
    \qquad\text{and}\qquad
    \sup_{n\ge1}
    \|Y_n^m\|_{\bH_2^s(T)}^2
    \le
    N A_m .
\end{equation}
Since $\eta\le s-\ep$ and $s-\ep=\ep(s-1)+(1-\ep)s$,
Lemma \ref{lem:prop_of_bessel_space}\eqref{norm_bounded} and \eqref{multi_ineq} give
\[
    \|f\|_{H_2^\eta}
    \le
    \|f\|_{H_2^{s-\ep}}
    \le
    \|f\|_{H_2^{s-1}}^{\ep}\,\|f\|_{H_2^s}^{1-\ep}.
\]
Hence, for $0<\delta<T$, H\"older's inequality with exponents $1/\ep$ and
$1/(1-\ep)$ on $\Omega\times[0,\delta]$ and \eqref{eq:time_bound} yield
\[
\begin{aligned}
    \E\int_0^\delta\|Y_n^m(r)\|_{H_2^\eta}^2\,\ud r
    &\le
    \left(
        \E\int_0^\delta\|Y_n^m(r)\|_{H_2^{s-1}}^2\,\ud r
    \right)^{\ep}
    \left(
        \E\int_0^\delta\|Y_n^m(r)\|_{H_2^s}^2\,\ud r
    \right)^{1-\ep}
    \\
    &\le
    \delta^{\ep}
    \left(
        \E\sup_{0\le r\le\delta}\|Y_n^m(r)\|_{H_2^{s-1}}^2
    \right)^{\ep}
    \left(
        \|Y_n^m\|_{\bH_2^s(T)}^2
    \right)^{1-\ep}
    \le
    N\delta^{\ep}A_m .
\end{aligned}
\]
This proves \eqref{eq:localized_small_time_estimate}.
\end{proof}

\subsection{Tightness and identification of the limit}
\label{subsec:tightness_identification}

By combining Lemma \ref{lem:uniform_localized_sobolev} with the bootstrap
technique of \cite[Remark~8.8]{kry1999analytic}, we obtain the following
tightness result.

\begin{proposition}
\label{prop:tightness}
The sequence $(u_n)_{n\ge1}$ is tight in $C((0,T];C_{\rm tem})$. 
\end{proposition}

\begin{proof}
The strategy of the proof is to select a good time at which the solution has positive
Sobolev regularity on an event of high probability, and then to apply
Krylov's estimate (Theorem~\ref{thm:krylov_lp_estimate}) twice, each
application raising the regularity by $2/p$, until the resulting
$\cH_p^{\theta}$-bound is strong enough for the H\"older embedding of
Theorem~\ref{thm:embedding_theorems}\eqref{embedding in time and space}.
Throughout the proof we write $N_m$ in place of $N$ to indicate possible dependence on $m$.

We start with the following elementary observation, which will be used
twice to select good times: if $\phi:(0,T)\to[0,\infty]$ is measurable and
$\int_I\phi(t)\,\ud t\le B$ for some interval $I\subset(0,T)$ of length $d$,
then, by Chebyshev's inequality,
\begin{equation}\label{eq:time_selection}
    \phi(t_*)\le\frac{2B}{d}
    \qquad\text{for some }t_*\in I.
\end{equation}
The proof consists of the following five steps. 

\medskip
\emph{Step 1: Parameters.}
Choose $p>8$ and
\begin{equation}\label{eq:choice_of_epsilon_0}
    0<\varepsilon_0<1\wedge\frac{p-8}{2},
\end{equation}
and set
\[
    \theta_0:=\frac{\varepsilon_0}{p},
    \qquad
    \theta_1:=\frac{2+\varepsilon_0}{p},
    \qquad
    \theta:=\frac{4+\varepsilon_0}{p},
    \qquad
    \kappa_0:=\frac{1-\varepsilon_0}{p}.
\]
Then $\kappa_0\in(0,1/2)$, and
\begin{equation}\label{eq:trace_relations}
    \theta_1-\frac2p=\theta_0,
    \qquad
    \theta-\frac2p=\theta_1,
    \qquad
    0<\theta_0<\theta_1<\theta<\frac12,
    \qquad
    \theta p >3.
\end{equation}
Choose $\ba,\bb$ such that
\begin{equation}\label{ab}
    \frac1p<\ba<\bb<\frac12\left(\theta-\frac1p\right),
\end{equation}
which is possible since $\theta p>3$, and set
\[
    \rho_t:=\ba-\frac1p>0,
    \qquad
    \rho_x:=\theta-2\bb-\frac1p>0,
    \qquad
    \delta:=\rho_t\wedge\rho_x>0.
\]
Note that $\delta$ is independent of $m$ and $n$. Finally, fix $m\ge2$,
choose
\begin{equation}\label{eq:choice_of_d}
    0<d=d_m<\frac{T}{4m}\wedge1,
\end{equation}
and write $Y_n=Y_n^m:=u_n\zeta_m$ where the superscript $m$ is suppressed as $m$ is fixed until the last step.

\medskip
\emph{Step 2: A good time and a good event.}
Since $\mu_0\in L_2\bigl(\Omega;H_2^{-1/2-\kappa}(\bR)\bigr)$ for every $\kappa\in(0,1/2)$, Lemma \ref{lem:uniform_localized_sobolev} applies with
$\kappa=\kappa_0$ and gives
\begin{equation}\label{eq:localized_H2_bound}
    \int_0^T\E\|Y_n(t)\|_{H_2^{1/2-\kappa_0}}^2\,\ud t
    \le
    \|Y_n\|_{\cH_2^{1/2-\kappa_0}(T)}^2
    \le
    I_m
    :=
    N\|\mu_0\|_{U_2^{1/2-\kappa_0}}^2
    +
    N\int_{\bR}\zeta_m(x)^2\,\ud x,
\end{equation}
where $I_m$ is independent of $n$. By \eqref{eq:time_selection} applied to
$\phi(t)=\E\|Y_n(t)\|_{H_2^{1/2-\kappa_0}}^2$ on $I=(d,2d)$, there exists a
deterministic \emph{good time} $s_d\in(d,2d)$, possibly depending on $n$ and $m$, such that
\begin{equation}\label{eq:first_good_time}
    \E\|Y_n(s_d)\|_{H_2^{1/2-\kappa_0}}^2
    \le
    \frac{2I_m}{d}.
\end{equation}
Since $\frac12-\kappa_0-(\frac12-\frac1p)=\theta_0$, the Sobolev embedding (Lemma~\ref{lem:prop_of_bessel_space}\eqref{norm_bounded}) gives
\begin{equation}\label{eq:initial_sobolev_embedding}
    H_2^{1/2-\kappa_0}(\bR)\hookrightarrow H_p^{\theta_0}(\bR).
\end{equation}
For $R>0$, define a \emph{good event}
\[
    \fA:=\bigl\{\|Y_n(s_d)\|_{H_p^{\theta_0}}\le R\bigr\}\in\cF_{s_d}.
\]
By \eqref{eq:initial_sobolev_embedding}, \eqref{eq:first_good_time}, and
Chebyshev's inequality,
\begin{equation}\label{eq:prob_of_A}
    \P(\fA^c)
    \le
    \P\bigl(\|Y_n(s_d)\|_{H_2^{1/2-\kappa_0}}>R/N\bigr)
    \le
    \frac{NI_m}{dR^2}.
\end{equation}

\medskip
\emph{Step 3: First regularity estimate.}
Set $Z:=\1_{\fA}Y_n$, $\rho_{1,m}:=\zeta_m'/\zeta_m$, and
$\rho_{2,m}:=\zeta_m''/\zeta_m$. Since $\fA\in\cF_{s_d}$, the indicator
$\1_{\fA}$ may be moved inside both the deterministic and the stochastic
integrals over $(s_d,t]$, and $\1_{\fA}u_n\zeta_m'=Z\rho_{1,m}$,
$\1_{\fA}u_n\zeta_m''=Z\rho_{2,m}$. Hence, on $[s_d,T]$,
\begin{equation}\label{eq:localized_equation_for_Z}
    \ud Z
    =
    \bigl[LZ-F_m(Z)\bigr]\,\ud t
    +
    G_{n,m}\,\xi(\ud t,\ud x),
    \qquad
    Z(s_d)=\1_{\fA}Y_n(s_d),
\end{equation}
where
\[
    F_m(Z):=2a(Z\rho_{1,m})_x-aZ\rho_{2,m}+bZ\rho_{1,m},
    \qquad
    G_{n,m}:=\1_{\fA}\sigma_n(u_n)\zeta_m,
\]
and, by the definition of $\fA$,
\begin{equation}\label{eq:initial_theta0_bound}
    \E\|Z(s_d)\|_{H_p^{\theta_0}}^p\le R^p.
\end{equation}
Note that for each fixed $n$, Lemma \ref{lem:regularized_solutions} gives $u_n\in\cH_p^{1/2-\kappa'}(T)$ for every $\kappa'\in(0,1/2)$. Fixing $\kappa'$ with $\theta<1/2-\kappa'$, which is possible since $\theta<1/2$ by \eqref{eq:trace_relations}, we obtain $Z\in\cH_p^{\theta}(s_d,T)\subset\cH_p^{\theta_1}(s_d,T)$. Thanks to this, Theorem \ref{thm:krylov_lp_estimate} can be applied in this step and the next. 

We now estimate the drift and noise terms as in the proof of Lemma \ref{lem:uniform_bound}. Since $\theta_1-1<0$,
Lemma \ref{lem:prop_of_bessel_space}\eqref{pointwise_multiplier},
\eqref{bounded_operator} and \eqref{norm_bounded} give
\begin{equation}\label{eq:localized_drift_estimate}
    \|F_m(Z)\|_{H_p^{\theta_1-2}}
    \le
    N_m\|Z\|_{H_p^{\theta_1-1}}
    \le
    N_m\|Z\|_{L_p},
\end{equation}
while, since $u_n\ge0$, \eqref{eq:sigma-n-properties} gives
$|G_{n,m}|\le N\zeta_m+NZ$, so that, as $\theta_1\in(0,1/2)$,
Lemma \ref{lem:white_noise_estimate} yields
\begin{equation}\label{eq:localized_noise_estimate}
    \|G_{n,m}\mathbf e\|_{H_p^{\theta_1-1}(\ell_2)}^p
    \le
    N\|G_{n,m}\|_{L_p}^p
    \le
    N_m+N_m\|Z\|_{L_p}^p.
\end{equation}
We now apply Theorem \ref{thm:krylov_lp_estimate} with $\eta=\theta_1$. By
\eqref{eq:trace_relations} the initial term is controlled by
\eqref{eq:initial_theta0_bound}, so that
\eqref{eq:localized_drift_estimate}--\eqref{eq:localized_noise_estimate} give
\begin{equation}\label{eq:first_krylov_estimate}
    \|Z\|_{\cH_p^{\theta_1}(s_d,t)}^p
    \le N_m(1+R^p)+N_m\int_{s_d}^t\E\|Z(r)\|_{L_p}^p\,\ud r.
\end{equation}
Setting $\bb_0:=\theta_1/2$ and $\ba_0:=\frac1p+\frac{\ep_0}{4p}$, we have
$\frac1p<\ba_0<\bb_0<\frac12$ and $\theta_1-2\bb_0=0$, so
Theorem \ref{thm:embedding_theorems}\eqref{embedding in time} on $[s_d,r]$
gives
\begin{equation}\label{eq:first_time_embedding}
    \E\sup_{s_d\le\ell\le r}\|Z(\ell)\|_{L_p}^p
    \le
    N\|Z\|_{\cH_p^{\theta_1}(s_d,r)}^p.
\end{equation}
Combining \eqref{eq:first_krylov_estimate} and
\eqref{eq:first_time_embedding}, Gronwall's inequality yields
\begin{equation}\label{eq:first_regularization_bound}
    \|Z\|_{\cH_p^{\theta_1}(s_d,T)}^p
    +
    \E\sup_{s_d\le t\le T}\|Z(t)\|_{L_p}^p
    \le
    N_m(1+R^p).
\end{equation}

\medskip
\emph{Step 4: Second regularity estimate and H\"older bound.}
By \eqref{eq:first_regularization_bound} and \eqref{eq:time_selection} applied
to $\phi(t)=\E\|Z(t)\|_{H_p^{\theta_1}}^p$ on $I=(s_d+d,s_d+2d)$, there exists
a deterministic time $r_d\in(s_d+d,s_d+2d)$, possibly depending on $n,m$ and $R$,
such that
\begin{equation}\label{eq:second_good_time}
    \E\|Z(r_d)\|_{H_p^{\theta_1}}^p\le\frac{N_m(1+R^p)}{d},
\end{equation}
and, by \eqref{eq:choice_of_d},
\begin{equation}\label{eq:r_d_small}
    r_d<s_d+2d<4d<\frac Tm.
\end{equation}
We apply Theorem \ref{thm:krylov_lp_estimate} once more to
\eqref{eq:localized_equation_for_Z}, now on $[r_d,T]$ with initial value
$Z(r_d)$ and with $\eta=\theta$. Since $0<\theta<1/2$, the drift and noise
terms are estimated exactly as in \eqref{eq:localized_drift_estimate} and
\eqref{eq:localized_noise_estimate}, with $\theta$ in place of $\theta_1$. By
\eqref{eq:trace_relations} the initial term is controlled by
\eqref{eq:second_good_time} and the $\bL_p$-terms by
\eqref{eq:first_regularization_bound}. Hence, using $d<1$, 
\begin{equation}\label{eq:second_krylov_estimate}
    \|Z\|_{\cH_p^{\theta}(r_d,T)}^p
    \le
    N\,\E\|Z(r_d)\|_{H_p^{\theta_1}}^p+N_m
    +N_m\int_{r_d}^T\E\|Z(t)\|_{L_p}^p\,\ud t
    \le
    \frac{N_m(1+R^p)}{d}.
\end{equation}
By the choice of $\ba,\bb$ in \eqref{ab},
Theorem \ref{thm:embedding_theorems}\eqref{embedding in time and space} on
$[r_d,T]$ together with \eqref{eq:r_d_small} gives
\begin{equation}\label{eq:holder_bound}
    \E\|Z\|_{C^{\rho_t}([T/m,T];\,C^{\rho_x}(\bR))}^p
    \le
    N\|Z\|_{\cH_p^{\theta}(r_d,T)}^p
    \le
    \frac{N_m(1+R^p)}{d},
\end{equation}
and since $\delta=\rho_t\wedge\rho_x$, so that
$\|f\|_{C^\delta([T/m,T]\times\bR)}\le N\|f\|_{C^{\rho_t}([T/m,T];C^{\rho_x}(\bR))}$,
\begin{equation}\label{eq:holder_bound_delta}
    \E\|Z\|_{C^{\delta}([T/m,T]\times\bR)}^p
    \le
    \frac{N_m(1+R^p)}{d}.
\end{equation}

\medskip
\emph{Step 5: Probability estimate and tightness.}
It remains to prove tightness in $C((0,T];\Ctem)$. Write $J_m:=[T/m,T]$, so
that $D_m=J_m\times[-m,m]$. We first show that, for any
positive constants $(S_m)_{m\ge2}$, the set
\[
    K:=\Bigl\{f\in C((0,T];\Ctem):
    \|f\zeta_m\|_{C^\delta(J_m\times\R)}\le S_m\ \text{ for every }m\ge2\Bigr\}
\]
is relatively compact.

Let $(f_j)_{j\ge1}\subset K$. Since $\zeta_m^{-1}$ is smooth and bounded on
$[-m,m]$ and $D_m\subset J_m\times\R$, we have
$\sup_{j\ge1}\|f_j\|_{C^\delta(D_m)}\le N_mS_m$ for every $m\ge2$. Hence, by
the Arzel\`a--Ascoli theorem and a diagonal argument, a subsequence, still
denoted by $(f_j)_j$, converges locally uniformly on $(0,T]\times\R$ to a
continuous function $f$. Now fix $T_0>0$ and $\lambda>0$, and choose $m\ge2$
with $T/m\le T_0$ and $1/m<\lambda$. Since
$\zeta_m^{-1}(x)=\cosh(x/m)\le e^{|x|/m}$, every $g\in K$ satisfies
\[
    \sup_{t\in[T_0,T],\,|x|\ge r}|g(t,x)|e^{-\lambda|x|}
    \le S_m e^{-(\lambda-1/m)r},\qquad r>0,
\]
and the same bound holds for $f$, being the pointwise limit of the $f_j$.
Therefore
\[
    \sup_{t\in[T_0,T],\,x\in\R}|f_j(t,x)-f(t,x)|e^{-\lambda|x|}
    \le\sup_{t\in[T_0,T],\,|x|\le r}|f_j(t,x)-f(t,x)|+2S_me^{-(\lambda-1/m)r},
\]
and letting $j\to\infty$ and then $r\to\infty$ gives $f_j\to f$ in
$C((0,T];\Ctem)$. Thus $K$ is relatively compact, and we may replace it by its
closure.

We now choose the constants. Since
$\1_{\fA}\|u_n\zeta_m\|_{C^\delta(J_m\times\R)}=\|Z\|_{C^\delta(J_m\times\R)}$
pathwise, \eqref{eq:prob_of_A}, \eqref{eq:holder_bound_delta} and Chebyshev's
inequality give, for every $S>0$,
\begin{equation}\label{eq:prob_estimate}
    \P\Bigl(\|u_n\zeta_m\|_{C^\delta(J_m\times\R)}>S\Bigr)
    \le\frac{NI_m}{dR^2}+\frac{N_m(1+R^p)}{dS^p}.
\end{equation}
Let $\ep>0$. For each $m\ge2$ choose first $R_m$ and then $S_m$, both
independent of $n$, so that each term on the right-hand side of
\eqref{eq:prob_estimate} is at most $\ep2^{-m-1}$, and let $K_\ep$ be the
corresponding set $K$. Then
\[
    \inf_{n\ge1}\P(u_n\in K_\ep)
    \ge1-\sup_{n\ge1}\sum_{m=2}^\infty
    \P\Bigl(\|u_n\zeta_m\|_{C^\delta(J_m\times\R)}>S_m\Bigr)
    \ge1-\ep,
\]
which shows that $(u_n)_{n\ge1}$ is tight in $C((0,T];\Ctem)$.
\end{proof}

Since Proposition \ref{prop:tightness} provides tightness of $(u_n)_{n\ge1}$ in $C((0,T];C_{\rm tem})$, we now pass to the limit $n\to\infty$ (along the subsequence) and then identify every subsequential limit as a solution of \eqref{eq:SPDE} with initial measure $\mu_0$.
\begin{proposition}
\label{prop:existence_measure_initial}
There exist a filtered probability space, a space-time white noise $\widetilde\xi$, a nonnegative random Radon measure $\widetilde\mu_0$, a coefficient field $(\widetilde a,\widetilde b,\widetilde c) $ satisfying Assumption~\ref{asp:coefficients}, and a nonnegative process 
\[ 
\ u\in C((0,T];C_{\mathrm{tem}}) 
\] 
such that $L(\widetilde\mu_0,\widetilde a,\widetilde b,\widetilde c) = L(\mu_0,a,b,c) $ and $u$ is a solution of 
\[ 
\partial_t u = \widetilde L u + \sigma( u)\widetilde\xi, \qquad  u(0,\cdot)=\widetilde\mu_0, 
\] 
where $\widetilde L = \widetilde a\,\partial_x^2 + \widetilde b\,\partial_x + \widetilde c$. More precisely, for every $\varphi\in C_c^\infty(\R)$ and every $t\in(0,T]$,
\begin{equation}
\label{eq:measure_limit}
\begin{aligned}
    (u(t),\varphi)
    &=
    (\widetilde\mu_0,\varphi)
    +
    \int_0^t (u(s),\widetilde L^*\varphi(s))\,\ud s  +
    \int_0^t\int_{\bR}
        \sigma(u(s,x))\varphi(x)\,\widetilde\xi(\ud s,\ud x)
    \qquad\text{a.s.}
\end{aligned}
\end{equation}
Moreover, if $\bE W_q(\mu_0)<\infty$, then for every
$\alpha\in(0,1)$,
\begin{equation}\label{eq:uniform_weighted_u}
    \bE\left[
        \sup_{0<t\le T} W_q(u(t))^\alpha
    \right]
    <\infty .
\end{equation}
In particular,
\[
    W_q(u(t))<\infty
    \qquad\text{for every }t\in(0,T]\text{ a.s.}
\]
\end{proposition}

\begin{proof}
The construction of a solution is standard when the initial datum is a regular function. Two points require care here. First, no estimate on $u_n$ is uniform in $n$ up to $t=0$, so we use the small-time estimates of Lemma \ref{lem:uniform_localized_sobolev} to control this. Second, no moment of the quadratic variation of $M_n^\varphi$  beyond the first is uniform in $n$ either, which forces a localization by stopping times.

Throughout, $\varphi\in C_c^\infty(\bR)$ is fixed, $K_\varphi$ is a compact set
containing $\operatorname{supp}\varphi$, and $m\ge1$ is chosen with
$K_\varphi\subset[-m,m]$. We also fix $\kappa\in(0,1/2)$. 

\medskip
\emph{Step 1: Skorokhod representation and the martingale problem.}
By Proposition~\ref{prop:tightness} the laws of $(u_n)_{n\ge1}$ are tight in
$C((0,T];C_{\mathrm{tem}})$. Recalling Assumption~\ref{asp:coefficients}, we
apply Skorokhod's representation theorem to the laws of the single random
variable $(\mu_0,a,b,c,u_n)$, with values in the Polish space
$H_2^{-1/2-\kappa}\times\cC\times C((0,T];C_{\rm tem})$. Thus, after passing to
a subsequence, we may work on a new probability space on which
$\cL(\mu_0^n,a_n,b_n,c_n,u_n)=\cL(\mu_0,a,b,c,u_n)$ for every $n\ge1$ and
\begin{equation}\label{eq:convergence_law}
    (\mu_0^n,a_n,b_n,c_n,u_n)
    \longrightarrow
    (\widetilde\mu_0,\widetilde a,\widetilde b,\widetilde c,u)
    \quad a.s.\quad\text{in }
    H_2^{-1/2-\kappa}\times\cC\times C((0,T];C_{\rm tem}).
\end{equation}
In particular $u_n\to u$ locally uniformly on $(0,T]\times\bR$ almost surely,
and $\cL(\widetilde\mu_0,\widetilde a,\widetilde b,\widetilde c)=\cL(\mu_0,a,b,c)$. 
For simplicity we keep the notation $(\mu_0,a,b,c)$ for the limit and also use $L:=a\,\partial_x^2+b\,\partial_x+c$ while $L_n:= a_n\,\partial_x^2+b_n\,\partial_x+c_n$.

On the new probability space, let $(\cG^n_t)_{0\le t\le T}$ and
$(\cG_t)_{0\le t\le T}$ be the usual augmentations of the filtrations generated
respectively by
\[
    \mu_0^n,\quad (a_n,b_n,c_n)\big|_{[0,t]\times\bR},
    \quad \{u_n(r,\cdot):r\in(0,t]\}
    \qquad\text{and}\qquad
    \mu_0,\quad (a,b,c)\big|_{[0,t]\times\bR},
    \quad \{u(r,\cdot):r\in(0,t]\} .
\]
By the almost sure convergence in \eqref{eq:convergence_law}, the coefficient field $(a,b,c)$ satisfies the conditions in Assumption~\ref{asp:coefficients} up to a null set. Therefore, defining $a=1$, $b=0$ and $c=0$ on this null set, we have that $(a,b,c)$ satisfies Assumption~\ref{asp:coefficients}.  Since $\cL(\mu_0^n,a_n,b_n,c_n,u_n)=\cL(\mu_0,a,b,c,u_n)$, the estimates
\eqref{eq:localized_H2_estimate} and \eqref{eq:localized_small_time_estimate}
continue to hold for the copies. We use this without further comment.

Define, for $n\ge1$,
\[
    u_{n,0}:=\Theta_n*\mu_0^n,\qquad \Theta_n(x):=n\Theta(nx),
    \qquad L_n:=a_n\partial_x^2+b_n\partial_x+c_n ,
\]
and
\begin{equation}\label{eq:def_Mn}
    M_n^\varphi(t):=(u_n(t),\varphi)-(u_{n,0},\varphi)
        -\int_0^t\bigl(u_n(s),L_n^*\varphi(s)\bigr)\,\ud s ,
    \qquad
    A_n^\varphi(t):=\int_0^t\!\!\int_{\bR}\sigma_n(u_n(s,x))^2\varphi(x)^2\,\ud x\,\ud s .
\end{equation}
We claim that
\begin{equation}\label{eq:copy_martingale}
    M_n^\varphi
    \quad\text{and}\quad
    (M_n^\varphi)^2-A_n^\varphi
    \qquad\text{are $(\cG^n_t)$-martingales.}
\end{equation}
Indeed, in Section~\ref{subsec:approximation-uniform-estimates} the
corresponding processes are martingales, since $u_n$ solves
\eqref{eq:regularized-measure-equation}. As they are adapted to the filtration
generated by the tuple, they are martingales for that filtration as well, and
\eqref{eq:copy_martingale} follows. 

\medskip 
\emph{Step 2: The limiting processes.}
Define, for $t\in(0,T]$,
\begin{equation}\label{eq:def_M_limit}
     M^\varphi(t):=(u(t),\varphi)-(\mu_0,\varphi)
        -\int_0^t\bigl(u(s),L^*\varphi(s)\bigr)\,\ud s,
    \qquad
    A^\varphi(t):=\int_0^t\!\!\int_{\bR}\sigma(u(s,x))^2\varphi(x)^2\,\ud x\,\ud s.
\end{equation}
In this step we show that these processes are well defined and continuous, and
that $M_n^\varphi(t)\to M^\varphi(t)$ in probability for every $t\in(0,T]$.

We begin with two convergences. Since $u_n\to u$ locally uniformly on
$(0,T]\times\bR$, $\sigma_n\to\sigma$ uniformly on compact subsets of
$[0,\infty)$, and $L_n^*\varphi\to L^*\varphi$ uniformly on
$[0,T]\times K_\varphi$ by \eqref{eq:convergence_law}, we have that  for every $\ep\in(0,T)$, almost surely, 
\begin{equation}\label{eq:unif_conv_compacts}
    \sup_{[\ep,T]\times K_\varphi}\bigl|\sigma_n(u_n)^2-\sigma(u)^2\bigr|
    \longrightarrow0
    \qquad\text{and}\qquad
    \sup_{s\in[\ep,T]}
    \bigl|(u_n(s),L_n^*\varphi(s))-(u(s),L^*\varphi(s))\bigr|
    \longrightarrow0.
\end{equation}

Next we control the quadratic variation and the drift term of $M_n^\varphi$
near $t=0$, uniformly in $n$. Define, for $\delta\in(0,T]$,
\[
    h(\delta):=\sup_{n\ge1}\bE A_n^\varphi(\delta)
    \qquad\text{and}\qquad
    \rho(\delta):=\sup_{n\ge1}\bE\int_0^\delta
        \bigl|(u_n(s),L_n^*\varphi(s))\bigr|\,\ud s,
\]
both of which are finite at $\delta=T$ by \eqref{eq:localized_H2_estimate}.
Since $L_n^*\varphi$ is supported in $K_\varphi$ and the coefficients are
uniformly bounded together with the derivatives appearing in $L_n^*$,
\eqref{eq:localized_small_time_estimate} with $\eta=0$ implies
\begin{equation}\label{eq:small_time_drift}
    \rho(\delta)\le N_\varphi\,\delta^{1/2}
        \sup_{n\ge1}\|u_n\zeta_m\|_{\bL_2(\delta)}
    \longrightarrow0\qquad\text{as }\delta\downarrow0,
\end{equation}
and, since $|\sigma_n(r)|\le N(r^\gamma+r)\le N(1+r)$ for $r\ge0$ by \eqref{eq:sigma-n-properties}, \eqref{eq:localized_small_time_estimate} again implies
\begin{equation}\label{eq:small_time_qv}
    h(\delta)\le N_\varphi\delta
        +N_\varphi\sup_{n\ge1}\|u_n\zeta_m\|_{\bL_2(\delta)}^2
    \longrightarrow0\qquad\text{as }\delta\downarrow0.
\end{equation}

We now transfer these bounds to the limits. Fix $\ep\in(0,\delta)$. By \eqref{eq:unif_conv_compacts} and Fatou's lemma, we obtain 
\[
    \E \int_\ep^\delta\!\!\int_{\bR}\sigma(u)^2\varphi^2\,\ud x\,\ud s
    \leq  \liminf_{n\to\infty} \E \int_\ep^\delta\!\!\int_{\bR}\sigma_n(u_n)^2\varphi^2\,\ud x\,\ud s
    \le h(\delta).
\] 
Letting
$\ep\downarrow0$, and arguing in the same way for the
drift term $\int_0^\delta
        \bigl|(u(s),L^*\varphi(s))\bigr|\,\ud s$, we obtain
\begin{equation}\label{eq:limit_small_time}
    \bE A^\varphi(\delta)\le h(\delta)
    \qquad\text{and}\qquad
    \bE\int_0^\delta\bigl|(u(s),L^*\varphi(s))\bigr|\,\ud s\le\rho(\delta),
    \qquad\delta\in(0,T].
\end{equation}
In particular, taking $\delta=T$,
\begin{equation}\label{eq:A_limit_finite}
    \bE A^\varphi(T)<\infty
    \qquad\text{and}\qquad
    \bE\int_0^T\bigl|(u(s),L^*\varphi(s))\bigr|\,\ud s<\infty.
\end{equation}
Hence $M^\varphi$ and $A^\varphi$ are well defined, and both are continuous. In particular $A^\varphi$ is finite and nondecreasing on $[0,T]$ with $A^\varphi(0)=0$.

It remains to prove the convergence of $M_n^\varphi$ and $A_n^\varphi$. For the
initial term, we have $(u_{n,0},\varphi)=(\mu_0^n,\widetilde\Theta_n*\varphi)$
with $\widetilde\Theta_n(x):=\Theta_n(-x)$, and
$\widetilde\Theta_n*\varphi\to\varphi$ in $H_2^{1/2+\kappa}(\bR)$ while
$\mu_0^n\to\mu_0$ in $H_2^{-1/2-\kappa}$, so that
\begin{equation}\label{eq:initial_convergence}
    (u_{n,0},\varphi)\longrightarrow(\mu_0,\varphi)\qquad\text{a.s.}
\end{equation}
For the drift term, fix $\delta\in(0,T)$. By the second convergence in \eqref{eq:unif_conv_compacts} and the local uniform convergence of $u_n$ to $u$ (see \eqref{eq:convergence_law}), we have  
\begin{equation}\label{eq:large_time_drift}
    \Gamma_n(\delta):=\sup_{t\in[\delta,T]}
    \Bigl|\bigl\{M_n^\varphi(t)-M_n^\varphi(\delta)\bigr\}
        -\bigl\{M^\varphi(t)-M^\varphi(\delta)\bigr\}\Bigr|
    \longrightarrow0\qquad\text{a.s.}
\end{equation}
Combining this with \eqref{eq:small_time_drift}, \eqref{eq:limit_small_time}
and \eqref{eq:initial_convergence}, we obtain
\begin{equation}\label{eq:M_convergence}
    M_n^\varphi(t)\longrightarrow M^\varphi(t)
    \qquad\text{in probability, for every }t\in(0,T].
\end{equation}
Similarly, since $A_n^\varphi$ and $A^\varphi$ are nondecreasing, for every
$t\in[0,T]$
\begin{equation}\label{eq:large_time_qv}
    \bigl|A_n^\varphi(t)-A^\varphi(t)\bigr|
    \le A_n^\varphi(\delta)+A^\varphi(\delta)
    +\int_\delta^T\!\!\int_{\bR}
        \bigl|\sigma_n(u_n(s,x))^2-\sigma(u(s,x))^2\bigr|\varphi(x)^2\,\ud x\,\ud s,
\end{equation}
where the last term tends to $0$ almost surely by the first convergence in
\eqref{eq:unif_conv_compacts}, while the first two have expectation at most
$2h(\delta)$ by \eqref{eq:limit_small_time}. Hence, by Chebyshev's inequality
and \eqref{eq:small_time_qv}, we have 
\begin{equation}\label{eq:convergence_qv}
    \sup_{0\le t\le T}\bigl|A_n^\varphi(t)-A^\varphi(t)\bigr|\longrightarrow0
    \qquad\text{in probability.}
\end{equation}

\medskip 
\emph{Step 3: Uniform convergence of $M_n^\varphi$ in time.}
In this step we show that $M_n^\varphi\to M^\varphi$ in probability, uniformly
in time, and that $M^\varphi$ extends continuously to $t=0$ with
$M^\varphi(0)=0$.

By \eqref{eq:copy_martingale} and Doob's maximal inequality,
\begin{equation}\label{eq:BDG_M}
    \sup_{n\ge1}\bE\sup_{0\le t\le\delta}|M_n^\varphi(t)|^2
    \le 4\sup_{n\ge1}\bE A_n^\varphi(\delta)=4h(\delta).
\end{equation}
For a finite $F\subset(0,\delta]$, \eqref{eq:M_convergence} gives
$\max_{t\in F}|M_n^\varphi(t)|\to\max_{t\in F}|M^\varphi(t)|$ in probability, so
Fatou's lemma and \eqref{eq:BDG_M} yield
$\bE\max_{t\in F}|M^\varphi(t)|^2\le4h(\delta)$. Letting $F$ increase to a
countable dense subset of $(0,\delta]$ and using the continuity of $M^\varphi$
on $(0,T]$, we obtain
\begin{equation}\label{eq:BDG_M_limit}
    \bE\sup_{0<t\le\delta}|M^\varphi(t)|^2\le4h(\delta),\qquad\delta\in(0,T].
\end{equation}
Since $\delta\mapsto\sup_{0<t\le\delta}|M^\varphi(t)|$ decreases as $\delta\downarrow0$ and
tends to $0$ in probability by \eqref{eq:BDG_M_limit} and
\eqref{eq:small_time_qv}, it tends to $0$ almost surely. Hence
$M^\varphi(t)\to0$ a.s. as $t\downarrow0$, and setting $M^\varphi(0):=0$ makes
$M^\varphi$ continuous on $[0,T]$. Equivalently, by \eqref{eq:small_time_drift} and \eqref{eq:limit_small_time},
\begin{equation}\label{eq:initial_attained}
    (u(t),\varphi)\longrightarrow(\mu_0,\varphi)\qquad\text{a.s. as }t\downarrow0.
\end{equation}

Now fix $\delta\in(0,T)$. Then, we have 
\[
    \sup_{0\le t\le T}\bigl|M_n^\varphi(t)-M^\varphi(t)\bigr|
    \le\sup_{r\le\delta}|M_n^\varphi(r)|+\sup_{r\le\delta}|M^\varphi(r)|
    +\bigl|M_n^\varphi(\delta)-M^\varphi(\delta)\bigr|+\Gamma_n(\delta).
\]
The last two terms tend to $0$ in probability by \eqref{eq:M_convergence} and
\eqref{eq:large_time_drift}, while the first two have expectation at most
$4h(\delta)^{1/2}$ by \eqref{eq:BDG_M}, \eqref{eq:BDG_M_limit} and the
Cauchy--Schwarz inequality. Letting first $n\to\infty$ and then
$\delta\downarrow0$, and using \eqref{eq:small_time_qv}, we conclude that
\begin{equation}\label{eq:M_uniform_convergence}
    \sup_{0\le t\le T}\bigl|M_n^\varphi(t)-M^\varphi(t)\bigr|\longrightarrow0
    \qquad\text{in probability.}
\end{equation}

\medskip 
\emph{Step 4: Convergence of the stopping times.}
In this step we construct a sequence of deterministic levels $R_j\uparrow\infty$
along which the corresponding hitting times of $A_n^\varphi$ converge to those
of $A^\varphi$. This will be used in the next step, where the martingale identities
are passed to the limit after stopping. For $R>0$, define stopping times (with respect to $(\cG^n_t)$ and $(\cG_t)$ respectively) 
\[
    \tau_{n,R}:=\inf\{t\ge0:A_n^\varphi(t)\ge R\}\wedge T
    \qquad\text{and} \qquad 
    \tau_{R}:=\inf\{t\ge0:A^\varphi(t)\ge R\}\wedge T.
\]
We claim that there exist $R_j\uparrow\infty$ such that, for every $j\ge1$,
\begin{equation}\label{eq:stopping_time_convergence}
    \tau_{n,R_j}\longrightarrow\tau_{R_j}
    \qquad\text{in probability as }n\to\infty.
\end{equation}

By \eqref{eq:A_limit_finite} we have $A^\varphi(T)<\infty$ almost surely. Thus, on
every $\omega$ with $A^\varphi(T)(\omega)<\infty$, we define 
\[
    \fB(\omega):=\Bigl\{R>0:
        \inf\{t\ge0:A^\varphi(t)(\omega)\ge R\}
        <\inf\{t\ge0:A^\varphi(t)(\omega)>R\}\Bigr\},
\]
and $\fB(\omega):=\varnothing$ otherwise. If $R\in\fB(\omega)$, then
$A^\varphi(\cdot)(\omega)$ is constant equal to $R$ on a nondegenerate
interval. Since distinct such levels give disjoint intervals, $\fB(\omega)$ is at
most countable and $\mathrm{Leb}(\fB(\omega))=0$. Thus, we apply Tonelli's theorem to get that 
\[
    \int_0^\infty\P(R\in\fB)\,\ud R=\bE\bigl[\mathrm{Leb}(\fB)\bigr]=0. 
\]
Therefore,  we have $\P(R\in\fB)=0$ for a.e.\ $R>0$, which allows us to choose deterministic $R_j\in(j,j+1)$ with
\begin{equation}\label{eq:continuity_point}
    \inf\{t\ge0:A^\varphi(t)\ge R_j\}=\inf\{t\ge0:A^\varphi(t)>R_j\}
    \quad\text{a.s., for every }j\ge1.
\end{equation}
Now we prove \eqref{eq:stopping_time_convergence}. Indeed, it suffices to consider an arbitrary subsequence and show that it has a further subsequence along which the convergence holds almost
surely. By \eqref{eq:convergence_qv} we may pass to a further subsequence,
still denoted by $n$, along which
$\sup_{0\le t\le T}|A_n^\varphi(t)-A^\varphi(t)|\to0$ almost surely. Fix
$\omega$ in the almost sure event on which this holds and on which
\eqref{eq:continuity_point} holds for every $j$, and suppose
$\tau_{R_j}(\omega)\in(0,T)$. Then, for all sufficiently small $\ep>0$,
\[
    A^\varphi(\tau_{R_j}-\ep)(\omega)<R_j<A^\varphi(\tau_{R_j}+\ep)(\omega).
\]
Hence, for all sufficiently large $n$,
\[
    A_n^\phi(\tau_{R_j}-\ep)(\omega)
    <R_j<
   A_n^\varphi(\tau_{R_j}+\ep)(\omega),
\]
which implies
\[
    \tau_{R_j}-\ep
    <\tau_{n,R_j}
    <\tau_{R_j}+\ep.
\]
The cases $\tau_{R_j}=0$ and $\tau_{R_j}=T$ follow by the corresponding
one-sided argument. This proves \eqref{eq:stopping_time_convergence}.

\medskip 
\emph{Step 5: The martingale problem for the limit.}
In this step we pass to the limit and show that $\langle M^\varphi\rangle=A^\varphi$. 
Fix $j\ge1$ and write $R=R_j$. By \eqref{eq:M_uniform_convergence},
\eqref{eq:convergence_qv}, \eqref{eq:stopping_time_convergence} and the uniform
continuity of the paths of $M^\varphi$ and $A^\varphi$,
\begin{equation}\label{eq:stopped_MA_convergence}
    \sup_{0\le t\le T}\bigl|M_n^\varphi(t\wedge\tau_{n,R})-M^\varphi(t\wedge\tau_{R})\bigr|
    +\sup_{0\le t\le T}\bigl|A_n^\varphi(t\wedge\tau_{n,R})-A^\varphi(t\wedge\tau_{R})\bigr|
    \longrightarrow0
\end{equation}
in probability. Since $A_n^\varphi(\tau_{n,R})\le R$ by continuity, the
Burkholder--Davis--Gundy inequality gives
\begin{equation}\label{eq:M_unif_integrability}
    \sup_{n\ge1}\bE\Bigl[\sup_{0\le r\le T}
        \bigl|M_n^\varphi(r\wedge\tau_{n,R})\bigr|^4\Bigr]
    \le N\sup_{n\ge1}\bE\bigl[A_n^\varphi(\tau_{n,R})^2\bigr]\le NR^2. 
\end{equation}
This shows that  $\{M_n^\varphi(t\wedge\tau_{n,R})\}_n$ and $\{M_n^\varphi(t\wedge\tau_{n,R})^2\}_n$ are uniformly integrable.

Fix $0\le s<t\le T$, $0<r_1<\cdots<r_k\le s$, and a bounded continuous $G$ on
$H_2^{-1/2-\kappa}\times\cC_s\times(C_{\rm tem})^k$, where $\cC_s$ is the space
\eqref{eq:def_coefficient_space} with $[0,T]$ replaced by $[0,s]$. Put
\[
    F_n:=G\bigl(\mu_0^n,(a_n,b_n,c_n)|_{[0,s]\times\bR},u_n(r_1),\dots,u_n(r_k)\bigr),
\]
and let $F$ be defined analogously from the limiting variables, so that
$\|F_n\|_\infty\le\|G\|_\infty$ and $F_n\to F$ a.s. by
\eqref{eq:convergence_law}. By \eqref{eq:copy_martingale} and optional
stopping,
\[
    \bE\Bigl[F_n\bigl(M_n^\varphi(t\wedge\tau_{n,R})-M_n^\varphi(s\wedge\tau_{n,R})\bigr)\Bigr]=0,
\]
\[
    \bE\Bigl[F_n\bigl\{M_n^\varphi(t\wedge\tau_{n,R})^2-A_n^\varphi(t\wedge\tau_{n,R})
        -M_n^\varphi(s\wedge\tau_{n,R})^2+A_n^\varphi(s\wedge\tau_{n,R})\bigr\}\Bigr]=0.
\]
By \eqref{eq:stopped_MA_convergence} and \eqref{eq:M_unif_integrability}, we get that as $n \to \infty$ 
\begin{equation}\label{eq:limit_identity}
\begin{aligned}
    &\bE\Bigl[F\bigl(M^\varphi(t\wedge\tau_{R})-M^\varphi(s\wedge\tau_{R})\bigr)\Bigr]=0,\\
    &\bE\Bigl[F\bigl\{M^\varphi(t\wedge\tau_{R})^2-A^\varphi(t\wedge\tau_{R})
        -M^\varphi(s\wedge\tau_{R})^2+A^\varphi(s\wedge\tau_{R})\bigr\}\Bigr]=0.
\end{aligned}
\end{equation}
By a monotone class argument, $M^\varphi(\cdot\wedge\tau_R)$ and
$M^\varphi(\cdot\wedge\tau_R)^2-A^\varphi(\cdot\wedge\tau_R)$ are
$(\cG_t)$-martingales, and  $\langle M^\varphi(\cdot\wedge\tau_R)\rangle_t=A^\varphi(t\wedge\tau_R)$. By
\eqref{eq:A_limit_finite} we have $\tau_{R_j}=T$ for all large $j$ almost
surely, so letting $j\to\infty$ shows that $M^\varphi$ is a continuous
square-integrable $(\cG_t)$-martingale with
\begin{equation}\label{eq:limit_bracket}
    \langle M^\varphi\rangle_t=A^\varphi(t)
    =\int_0^t\!\!\int_{\bR}\sigma(u(s,x))^2\varphi(x)^2\,\ud x\,\ud s,
\end{equation}
and, by polarization,
\begin{equation}\label{eq:limit_bracket_joint}
    \langle M^\varphi,M^\psi\rangle_t
    =\int_0^t\!\!\int_{\bR}\sigma(u(s,x))^2\varphi(x)\psi(x)\,\ud x\,\ud s,
    \qquad\varphi,\psi\in C_c^\infty(\bR).
\end{equation}
Therefore, we conclude that there exists a space-time white noise $\widetilde\xi$, on an enlarged probability space and relative to the correspondingly enlarged filtration, such that
\[
    M^\varphi(t)=\int_0^t\!\!\int_{\bR}\sigma(u(s,x))\varphi(x)\,\widetilde\xi(\ud s,\ud x),
    \qquad\varphi\in C_c^\infty(\bR).
\]
Together with \eqref{eq:def_M_limit} and
\eqref{eq:initial_attained} this proves \eqref{eq:measure_limit}, and
conditions (i) and (ii) of Definition \ref{def:weak_soln} hold by
\eqref{eq:convergence_law} and \eqref{eq:A_limit_finite}. Hence $u$ is a
nonnegative solution of \eqref{eq:SPDE} with initial measure $\mu_0$.

\medskip
\emph{Step 6: The weighted mass estimate.}
Assume $\bE W_q(\mu_0)<\infty$ and let $\alpha\in(0,1)$. Following the proof of
Lemma~\ref{lem:mtg_bound_w} with Remark~\ref{rmk:L_epsilon_bound_v_n},
\begin{equation*}
    \bE\Bigl[\sup_{0\le s\le T}W_q(u_n(s))^\alpha\Bigr]
    \le N\,\bE\bigl[W_q(u_{n,0})^\alpha\bigr],
\end{equation*}
with $N=N(\alpha,q,\nu,T)$ independent of $n$. Since $\operatorname{supp}\Theta\subset[0,1]$, the change of variables
$x=y+z/n$ gives
\[
    W_q(u_{n,0})
    =\int_{\bR}\int_0^1\Phi_q(y+z/n)\Theta(z)\,\ud z\,\mu_0^n(\ud y)
    \le N_qW_q(\mu_0^n),
\]
where we used that $\Phi_q(y+w)\le N_q\Phi_q(y)$ for $|w|\le1$. Hence, by
Jensen's inequality and $\cL(\mu_0^n)=\cL(\mu_0)$, we have 
\begin{equation}\label{eq:uniform_weighted_un}
    \sup_{n\ge1}\bE\bigl[W_q(u_{n,0})^\alpha\bigr]
    \le N_q\bigl(\bE W_q(\mu_0)\bigr)^\alpha<\infty. 
\end{equation}

Fix $\delta\in(0,T)$ and $R>0$. Since $u_n\to u$ uniformly on
$[\delta,T]\times[-R,R]$, 
\[
    \sup_{t\in[\delta,T]}\int_{-R}^{R}\Phi_q(x)u(t,x)\,\ud x
    =\lim_{n\to\infty}\sup_{t\in[\delta,T]}\int_{-R}^{R}\Phi_q(x)u_n(t,x)\,\ud x
    \le\liminf_{n\to\infty}\sup_{0\le t\le T}W_q(u_n(t)).
\]
By monotone convergence, letting $R\to\infty$ gives
\[
    \sup_{t\in[\delta,T]}W_q(u(t))
    \le
    \liminf_{n\to\infty}
    \sup_{0\le t\le T}W_q(u_n(t)).
\]
Thus, Fatou's lemma and \eqref{eq:uniform_weighted_un} implies \eqref{eq:uniform_weighted_u}, which completes the proof. 
\end{proof}

\begin{remark}
We use the classical Skorokhod representation theorem to the law of the entire tuple $(\mu_0,a,b,c,u_n)$, so that the copies of both the data and the solution to depend on $n$; see \cite{OS25} for related discussion.
% We use the classical Skorokhod representation theorem, applied to the law of
% the single random variable $(\mu_0,a,b,c,u_n)$, so that the copy of the entire
% tuple depends on $n$. We do not use the strengthened version  in which the data are held fixed while only the solutions converge almost surely, which is false
% by \cite{OS25}.
\end{remark}

% =====================================================================

\section{Positive-time regularization and proof of Theorem \ref{thm:CSP_measure_initial}}
\label{sec:proof_thm_measure}

In the previous section we constructed a solution from the measure-valued initial data and obtained the weighted mass bounds needed below. By uniqueness in law for the initial law $\mathcal L(\mu_0)$, these path-law properties also hold for any nonnegative solution with initial measure $\mu_0$.

The purpose of this section is to prove the compact support property for such solutions. The key point is that, although the initial state may be a measure, the solution becomes sufficiently regular at every positive time. We prove this by a time-cutoff argument in the following section. 

\subsection{Time-cutoff regularization}
Recall that $\gamma\in(0,1)$ and $\lambda\in[1,2)$ are the constants in
Assumption \ref{asp:sigma}. Also recall that
\[
    W_q(f)=\int_{\bR}(1+x^2)^{q/2}\,f(\ud x),
\]
with the convention $f(\ud x)=f(x)\,\ud x$ when $f$ is a function.
\begin{proposition}
\label{prop:positive_time_Sobolev}
Let $u\in C((0,T];C_{\rm tem})$ be a nonnegative solution of \eqref{eq:SPDE} with initial measure $\mu_0$, and fix $t_0\in(0,T)$. Assume that for some $q_*\ge0$,
\begin{equation}
\label{eq:weight_bound}
    \sup_{r\in[\delta,T]} W_{q_*}(u(r))<\infty
    \qquad
    \text{a.s. for every }\delta\in(0,T).
\end{equation}
Assume also one of the following two alternatives:
\begin{enumerate}
\item[\rm (a)]
$q_*=0$ and $p>6$ is chosen so that
\[
    p\gamma\ge1.
\]

\item[\rm (b)]
$q_*>0$ and $p>6$ is chosen so that
\[
    p\gamma>\frac1{q_*+1}.
\]
\end{enumerate}
If $\lambda\in(1,2)$, we additionally require
\[
    q_*>\frac{\lambda-1}{2-\lambda}.
\]

Then there exists a non-random $\beta>1/p$, depending only on $p$,  such that
\[
    u(t_0,\cdot)\in H_p^\beta(\bR)
    \qquad\text{a.s.}
\]
\end{proposition}

\begin{proof}
Let $t_0 \in (0, T)$. Choose $0<s<t_0$ and $\chi\in C^\infty([0,T])$ such that
\[
    0\le\chi\le1,\qquad
    \chi(t)=0\text{ for }t\le s/2,\qquad
    \chi(t)=1\text{ for }t\ge s.
\]
Set
\[
    v(t,x):=\chi(t)u(t,x).
\]
We first identify the equation satisfied by $v$.
For every $\varphi\in C_c^\infty(\bR)$, the weak formulation for $u$ gives us the following:
\[
\begin{aligned}
    (v(t),\varphi)
    &=
    \int_0^t (v(r),L^*\varphi(r))\,\ud r
    +
    \int_0^t (\chi'(r)u(r),\varphi)\,\ud r +
    \int_0^t\int_{\bR}
        \chi(r)\sigma(u(r,x))\varphi(x)\,\xi(\ud r,\ud x).
\end{aligned}
\]
Equivalently, in the sense of distributions,
\begin{equation}
\label{eq:cutoff_equation}
    dv=(L v+\chi'u)\,\ud t+\chi\sigma(u)\xi(\ud t,\ud x), \qquad v(0)=0.
\end{equation}
We now consider the drift and noise terms in \eqref{eq:cutoff_equation} to apply Theorem \ref{thm:krylov_lp_estimate}.  In order to do this, for $R,L\geq 1$, we introduce the stopping times
\[
\tau_R:=
    \inf\bigl\{r\in[s/2,T]:W_{q_*}(u(r))>R\bigr\}\wedge T \quad \text{and} \quad 
    \tau_{R,L}
    :=\tau_R\wedge \inf\left\{t\in[0,T]: \int_0^t\|v(r)\|_{L_p}^p\,\ud r>L
    \right\}.
\]
Choose
\[\theta\in\left(\frac3p\,,\frac12\right).\]  We first estimate the drift term.  Using Remark \ref{rem:bessel_kernel}, we have $R_{2-\theta}\in L_p(\bR)$ and hence
\[
\begin{aligned}
    \|\chi'(r)u(r)\|_{H_p^{\theta-2}}=|\chi'(r)|\,\|R_{2-\theta}*u(r)\|_{L_p} \le
    |\chi'(r)|\,\|R_{2-\theta}\|_{L_p}\int_{\bR}u(r,x)\,\ud x     \le N|\chi'(r)|\,W_{q_*}(u(r)).
\end{aligned}
\]
Since $\chi'=0$ on $[0,s/2]$, we have 
\begin{equation}
\label{eq:drift_bound}
    \|\chi'u\|_{\mathbb H_p^{\theta-2}(\tau_{R, L})}^p \le N_R .
\end{equation}
Next we estimate the noise term. 
By Lemma \ref{lem:white_noise_estimate},
\[
    \|\chi\sigma(u)\mathbf e\|_
    {\mathbb H_p^{\theta-1}(\tau_{R, L};\ell_2)}^p
    \le N\bE\int_0^{\tau_{R, L}} \|\chi(r)\sigma(u(r))\|_{L_p}^p\,\ud r.
\]
Assumption \ref{asp:sigma} implies  $|\sigma(y)|\le N(y^\gamma+y)$ for $y\ge0$.  Since $0\le\chi\le1$ and $v=\chi u$,
\[
    |\chi\sigma(u)|^p
    \le
    N v^{p\gamma}+Nv^p.
\]
We control the term involving $v^{p\gamma}$ according to the two alternatives in the statement. If $q_*=0$, then alternative {\rm (a)} gives $p\gamma\ge1$. Since
$\gamma<1$, we also have $p\gamma\le p$, and therefore by Young's inequality,
\[
    y^{p\gamma}\le N(y+y^p),
    \qquad y\ge0.
\]
Thus, on $[0,\tau_{R, L}]$,
\[
    \int_{\bR}v(r,x)^{p\gamma}\,\ud x
    \le N\int_{\bR}v(r,x)\,\ud x +
    N\|v(r)\|_{L_p}^p  
    \le N_R+N\|v(r)\|_{L_p}^p .
\]
If $q_*>0$, there are two subcases. If $p\gamma<1$, then alternative {\rm (b)} gives
\[
    p\gamma>\frac1{q_*+1}.
\]
Thus, using the inequality \eqref{eq:Holder_epsilon}, we have that on $[0,\tau_R]$,
\[
    \int_{\bR} v(r, x)^{p\gamma}\,\ud x
    \le W_{q_*}(v)^{p\gamma}
    \left(\int_{\bR} \Phi_{q_*}(x)^{-p\gamma/(1-p\gamma)}\,\ud x \right)^{1-p\gamma}  \le N_R.
\]
If $q_*>0$ and $p\gamma\ge1$, then the same inequality
$y^{p\gamma}\le N(y+y^p)$ gives
\[
    \int_{\bR}v(r,x)^{p\gamma}\,\ud x
    \le N_R+N\|v(r)\|_{L_p}^p.
\]
Combining all cases, we obtain
\begin{equation}
\label{eq:noise_bound1}
    \|\chi\sigma(u)\mathbf e\|_{\mathbb H_p^{\theta-1}(\tau_{R, L};\ell_2)}^p
    \le N_R +  N\|v\|_{\mathbb L_p(\tau_{R, L})}^p .
\end{equation}
At this point the right-hand side is finite by the definition of $\tau_{R,L}$,
so that Theorem \ref{thm:krylov_lp_estimate} is applicable. Before applying it
we remove the $\mathbb L_p$-term by interpolation. By Lemma \ref{lem:prop_of_bessel_space}\eqref{multi_ineq}, for every $\varepsilon>0$,
\[
    \|v(r)\|_{L_p}^p
    \le \varepsilon\|v(r)\|_{H_p^\theta}^p +
    N_\varepsilon\|v(r)\|_{H_p^{-2}}^p .
\]
Using Remark \ref{rem:bessel_kernel}, we have that for $r\le\tau_R$,
\[
    \|v(r)\|_{H_p^{-2}}\le \|R_2*v(r)\|_{L_p} \le \|R_2\|_{L_p}\int_{\bR}v(r,x)\,\ud x 
    \le N W_{q_*}(u(r))
    \le N_R.
\]
 Therefore
\begin{equation}\label{eq:v_Lp_bound}
    \|v\|_{\mathbb L_p(\tau_{R,L})}^p
    \le \varepsilon \|v\|_{\mathbb H_p^\theta(\tau_{R,L})}^p + N_R .
\end{equation}
Substituting this into
\eqref{eq:noise_bound1}, we obtain
\begin{equation}
\label{eq:noise_bound2}
    \|\chi\sigma(u)\mathbf e\|_
    {\mathbb H_p^{\theta-1}(\tau_{R,L};\ell_2)}^p
    \le N_R + N\varepsilon \|v\|_{\mathbb H_p^\theta(\tau_{R,L})}^p.
\end{equation}

Now we apply  Theorem \ref{thm:krylov_lp_estimate} to
\eqref{eq:cutoff_equation} stopped at $\tau_{R,L}$. Using
\eqref{eq:drift_bound} and
\eqref{eq:noise_bound2}, we get
\[
    \|v\|_{\mathcal H_p^\theta(\tau_{R,L})}^p
    \le N_R + N\varepsilon \|v\|_{\mathbb H_p^\theta(\tau_{R,L})}^p.
\]
Since $\|v\|_{\mathbb H_p^\theta(\tau_{R,L})} \le\|v\|_{\mathcal H_p^\theta(\tau_{R,L})}$ by Definition \ref{def of cH},   we choose $\varepsilon>0$ sufficiently small  to obtain 
\begin{equation}\label{eq:Htheta_bound1}
    \|v\|_{\mathcal H_p^\theta(\tau_{R,L})}^p
    \le N_R,
\end{equation}
where $N_R$ is independent of $L$.

This estimate allows us to remove the auxiliary localization. Indeed, \eqref{eq:v_Lp_bound} and \eqref{eq:Htheta_bound1} imply
\[
    \sup_{L\ge1} \bE\int_0^{\tau_{R,L}}\|v(r)\|_{L_p}^p\,\ud r
    \le N_R.
\]
Hence, 
\[
    \bP(\tau_{R,L}<\tau_R)
    \le \bP\left(\int_0^{\tau_{R,L}}\|v(r)\|_{L_p}^p\,\ud r\ge L\right)
    \le \frac{N_R}{L} \longrightarrow0 \quad \text{as $L\to \infty$.}
\]
Since $\tau_{R,L}$ is increasing in $L$, it follows that $\tau_{R,L}\uparrow\tau_R$ a.s.

By Fatou's lemma, \eqref{eq:Htheta_bound1} gives
\[
    \|v\|_{\mathcal H_p^\theta(\tau_R)}^p
    \le N_R.
\]

Now choose
\[
    \frac1p<\ba<\bb<\frac12\left(\theta-\frac1p\right),
\]
and set
\[
    \beta:=\theta-2\bb.
\]
Then $\beta>1/p$. By Theorem \ref{thm:embedding_theorems}\eqref{embedding in time}, we have 
\[
    \bE
    \left[\sup_{0\le r\le\tau_R} \|v(r)\|_{H_p^\beta}^p\right]
    \le N_R,
\]
which implies, since $v=u$ on $[s,T]$, 
\[
\bE\Bigl[
\sup_{s\le r\le\tau_R}\|u(r,\cdot)\|_{H_p^\beta}^p
\Bigr]
\le N_R<\infty.
\]
Finally, $\tau_R\uparrow T$ a.s. as
$R\to\infty$ by \eqref{eq:weight_bound}, so $\mathbf 1_{\{\tau_R>t_0\}}\uparrow1$
a.s. and therefore
\[
    u(t_0,\cdot)\in H_p^\beta(\bR)\qquad\text{a.s.},
\]
which finishes the proof.
\end{proof}      

\subsection{Restarting and completion of the proof of Theorem \ref{thm:CSP_measure_initial}}
We now complete the proof of Theorem \ref{thm:CSP_measure_initial} by restarting the equation at a fixed positive time. This restart is made possible by Proposition \ref{prop:positive_time_Sobolev}, which shows that the solution has the Sobolev regularity required to apply Theorem \ref{thm:CSP_regular_initial}.
 
\begin{proof}[Proof of Theorem \ref{thm:CSP_measure_initial}]
Let $u$ be an arbitrary nonnegative solution of \eqref{eq:SPDE} with
initial measure $\mu_0$. We first choose the weighted moment that will be
used in the restart argument. If $\lambda=1$, set
\[
    q_*:=0 \qquad\text{and}\qquad
    \bE \mu(\bR) <\infty .
\]
If $\lambda\in(1,2)$, choose $q_*$ such that
\[
    q_*>\frac{\lambda-1}{2-\lambda}
    \qquad\text{and}\qquad
    \bE W_{q_*}(\mu_0)<\infty .
\]
Such a $q_*$ exists by the assumptions of the theorem.

We first verify \eqref{eq:weight_bound} to apply Proposition \ref{prop:positive_time_Sobolev}. By Proposition \ref{prop:existence_measure_initial} (see \eqref{eq:uniform_weighted_u}), the solution constructed
from $\mu_0$ satisfies \eqref{eq:weight_bound}. Since uniqueness in law
holds for the initial law $\mathcal L(\mu_0)$, the same path-law property
holds for the present solution $u$. Hence
\[
    \sup_{r\in[\delta,T]}W_{q_*}(u(r))<\infty
    \qquad
    \text{a.s. for every }\delta\in(0,T).
\]

Fix $t_0\in(0,T)$. Choose $p>6$ as follows. If $q_*=0$, choose $p$ so large
that
\[
    p\gamma\ge1.
\]
If $q_*>0$, choose $p$ so large that
\[
    p\gamma>\frac1{q_*+1}.
\]
By Proposition \ref{prop:positive_time_Sobolev}, there exists
$\beta>1/p$ such that
\[
    u(t_0,\cdot)\in H_p^\beta(\bR)
    \qquad\text{a.s.}
\]
Setting $\eta_0:=\beta+2/p$, we have
$u(t_0,\cdot)\in H_p^{\eta_0-2/p}(\bR)$ a.s.\ and $p>3/\eta_0$, since
$\beta>1/p$.

We now check the remaining hypotheses of
Theorem \ref{thm:CSP_regular_initial} after conditioning. For $R\ge1$, define
\[
    \fA_R:=\left\{ \|u(t_0,\cdot)\|_{H_p^\beta} +W_{q_*}(u(t_0))\le R\right\}.
\]
Then, $\fA_R\in\mathcal F_{t_0}$ and $\fA_R\uparrow\Omega$ a.s.
If $\bP(\fA_R)=0$, there is nothing to prove for this $R$. Otherwise,
we work under the conditional probability measure
$\bP_R:=\bP(\,\cdot\,|\fA_R)$.
Under $\bP_R$, the shifted initial function $Z_0:=u(t_0,\cdot)$ satisfies
$Z_0\in L_p(\Omega;H_p^{\eta_0-2/p}(\bR))$.
Indeed,
\[
    \bE_R\|Z_0\|_{H_p^{\eta_0-2/p}}^p =
    \bE_R\|u(t_0,\cdot)\|_{H_p^\beta}^p
    \le R^p,
\] where $\E_R$ denotes the corresponding expectation with respect to $\P_R$.
Moreover, the moment condition in Theorem \ref{thm:CSP_regular_initial} is also satisfied under $\bP_R$. That is, if $q_*=0$, then $\bE_R W_0(Z_0) \le R$, and our choice of $p$ gives $p\gamma\ge1$. If $q_*>0$, then $\bE_R W_{q_*}(Z_0)\le R$, and our choice of $p$ gives $p\gamma>\frac1{q_*+1}$.  In the case $\lambda\in(1,2)$, the chosen $q_*$ also satisfies $q_*>\frac{\lambda-1}{2-\lambda}$.
Thus the regularity and moment assumptions of
Theorem \ref{thm:CSP_regular_initial} are satisfied for the shifted initial
law under $\bP_R$.

We now define the future noise
\[
    \xi^{(t_0)}(\ud s,\ud x):=\xi(t_0+\ud s,\ud x),
\]
which is independent of $\mathcal F_{t_0}$, and hence independent of $Z_0$ under
the conditional law $\bP_R$. Consider the shifted equation on
$[0,T-t_0]$:
\[
    \partial_s z(s,x)
    = L^{(t_0)}z(s,x)
    + \sigma(z(s,x))\xi^{(t_0)}(s,x), \qquad
    z(0,\cdot)=Z_0,
\]
where
\[
    L^{(t_0)}
    = a(t_0+s,x)\partial_x^2
    + b(t_0+s,x)\partial_x
    + c(t_0+s,x).
\]
The shifted coefficients satisfy the same coefficient assumptions as the original coefficients. By the shifted uniqueness-in-law assumption (Assumption~\ref{ass:shifted_uniqueness}), the uniqueness-in-law hypothesis required in Theorem \ref{thm:CSP_regular_initial} is available for this shifted equation.

For $m,R\ge 1$ and $a,b \ge 0$ with $a<b$, define a path event
\[
\fC_{m,R}(a,b) : = \left\{ f\, : \, f(t,x) = 0 \text{ for all $t\in\left[a+ \frac{(b-a)}{2m}, b\right]$ and all $x\in\R$ with $|x|\ge R$} \right\},
\] and 
\[
\fC(a,b) := \bigcap_{m=1}^\infty \bigcup_{R=1}^\infty \fC_{m,R}(a,b).
\] Note that $f\in \fC(a,b)$ implies $f(t)$ has compact support for every $t\in (a,b]$. 
The shifted process
\[
    u^{(t_0)}(s,x):=u(t_0+s,x),
    \qquad 0\le s\le T-t_0,
\]
is a nonnegative solution of the shifted equation with initial function $Z_0$. Therefore, Theorem \ref{thm:CSP_regular_initial}, applied under $\bP_R$ gives
\[
 \P_R\left( u^{(t_0)} \in \fC(0,T-t_0) \right)=1.
\]
Thus, we have 
\[
    \bP\left(
        \fA_R\cap
        \left\{u \in \fC(t_0,T)^c\right\}\right)=0.
\]
Letting $R\to\infty$ and using $\fA_R\uparrow\Omega$, we obtain $\P(u\in\fC(t_0,T))=1$ for every $t_0\in(0,T)$. Taking the intersection over $t_0=T/2m$, $m\ge1$, we conclude that $u(t)$ has compact support for every $t\in(0,T]$ almost surely, which completes the proof.
\end{proof}

\bibliographystyle{alpha}      
\bibliography{refs}

% \begin{thebibliography}{1}
% \bibitem{Krylov97}
% N.\,V.~Krylov, \emph{On a result of C. Mueller and E. Perkins}, Probab.\ Theory Relat.\ Fields \textbf{108} (1997), 543--557.
% \end{thebibliography}

\end{document}